\documentclass[10pt,a4paper]{article}
\usepackage{amssymb}
\usepackage{graphicx}
\usepackage{amsmath,amsthm}
\usepackage{url,paralist}
\usepackage{anysize}
\usepackage{diagrams}

\newtheorem{theorem}{Theorem}[section]
\newtheorem{lemma}[theorem]{Lemma}
\newtheorem{proposition}[theorem]{Proposition}
\newtheorem{corollary}[theorem]{Corollary}
\theoremstyle{definition}

\newtheorem{example}[theorem]{Example}
\newtheorem{remark*}[theorem]{Remark}
\newtheorem{rremark*}[theorem]{Final remark}
\newtheorem{problem}[theorem]{Problem}

\newcommand{\mathsmall}[1]{{\begingroup\everymath{\scriptstyle}\small{#1}\endgroup}}

\begin{document}

\title{{\huge The ideal-valued index for a dihedral group action,}\\
{\huge and mass partition by two hyperplanes}}
\author{Pavle V. M. Blagojevi\'c\thanks{%
The research leading to these results has received funding from the European Research
Council under the European Union's Seventh Framework Programme (FP7/2007-2013) /
ERC Grant agreement no.~247029-SDModels. Also supported by the grant ON 174008 of the Serbian
Ministry of Science and Environment.} \\
Mathemati\v cki Institut\\
Knez Michailova 35/1\\
11001 Beograd, Serbia\\
\url{pavleb@mi.sanu.ac.rs} \and \setcounter{footnote}{6} G\"unter M. Ziegler%
\thanks{The research leading to these results has received funding from the European Research
Council under the European Union's Seventh Framework Programme (FP7/2007-2013) /
ERC Grant agreement no.~247029-SDModels.} \\
Inst.\ Mathematics, MA 6-2\\
TU Berlin\\
D-10623 Berlin, Germany\\
\url{ziegler@math.tu-berlin.de}}
\date{{\small December 9, 2010}}
\maketitle

\begin{abstract}
\noindent We compute the complete Fadell--Husseini index of the dihedral
group $D_{8}=(\mathbb{Z}_{2})^{2}\rtimes \mathbb{Z}_{2}$ acting on~$%
S^{d}\times S^{d}$ for $\mathbb{F}_{2}$ and for $\mathbb{Z}$ coefficients,
that is, the kernels of the maps in equivariant cohomology
\begin{equation*}
H_{D_{8}}^{\ast }(\mathrm{pt},\mathbb{F}_{2})\ \ \longrightarrow \ \ H_{D_{8}}^{\ast
}(S^{d}\times S^{d},\mathbb{F}_{2})
\end{equation*}%
and%
\begin{equation*}
H_{D_{8}}^{\ast }(\mathrm{pt},\mathbb{Z})\ \ \longrightarrow \ \ H_{D_{8}}^{\ast }(S^{d}\times S^{d},
\mathbb{Z}).
\end{equation*}%
This establishes the complete cohomological lower bounds, with $\mathbb{F}%
_{2}$ and with $\mathbb{Z}$ coefficients, for the two-hyperplane case of Gr%
\"{u}nbaum's 1960 mass partition problem: For which $d$ and $j$ can any $j$
arbitrary measures be cut into four equal parts each by two suitably chosen
hyperplanes in $\mathbb{R}^{d}$? In both cases, we find that the ideal
bounds are not stronger than previously established bounds based on one of
the maximal abelian subgroups of $D_{8}$.
\end{abstract}

\tableofcontents

\section{Introduction}

\subsection{The hyperplane mass partition problem}

A \textit{mass distribution} on $\mathbb{R}^{d}$ is a finite Borel measure $\mu (X)=\int_{X}fd\mu $ determined by an
integrable density function $f:\mathbb{R}^{d}\rightarrow\mathbb{R}$.

Every affine hyperplane $H=\{x\in\mathbb{R}^{d}~|~\langle x,v\rangle =\alpha \}$ in
$\mathbb{R}^{d}$ determines two open halfspaces%
\begin{equation*}
H^{-}=\{x\in\mathbb{R}^{d}~|~\langle x,v\rangle < \alpha \}\text{ and }H^{+}=\{x\in
\mathbb{R}^{d}~|~\langle x,v\rangle > \alpha \}\text{.}
\end{equation*}%
An \textit{orthant} of an arrangement of $k$ hyperplanes $\mathcal{H}%
=\{H_1,H_2,\ldots,H_k\}$ in $\mathbb{R}^{d}$ is an intersection of halfspaces $%
\mathcal{O}=H_1^{\alpha_{1}}\cap\cdots\cap H_k^{\alpha _{k}}$, for some $%
\alpha _{j}\in\mathbb{Z}_{2}$. Thus there are $2^{k}$ orthants determined by
$\mathcal{H}$ and they are naturally indexed by elements of the group $%
\left(\mathbb{Z}_{2}\right) ^{k}$.

An arrangement of hyperplanes $\mathcal{H}$ \textit{equiparts} a collection
of mass distributions $\mathcal{M}$ in $\mathbb{R}^{d}$ if for each orthant~$\mathcal{O}$ and each measure $\mu \in \mathcal{M}
$ we have
\begin{equation*}
\mu (\mathcal{O})=\tfrac{1}{2^{k}}\mu (\mathbb{R}^{d}).
\end{equation*}

A triple of integers $(d,j,k)$ is \textit{admissible} if for every
collection $\mathcal{M}$ of $j$ mass distributions in $\mathbb{R}^{d}$ there exists an arrangement of $k$ hyperplanes $\mathcal{H}$
equiparting them.


The general problem formulated by Gr\"{u}nbaum \cite{Grunb} in 1960 can be
stated as follows.


\begin{problem}
\label{GruenbaumProblem}Determine the function $\Delta :\mathcal{\mathbb{N}}^{2}\rightarrow \mathcal{
\mathbb{N}}$ given by
\begin{equation*}
\Delta (j,k)=\min \{d~|~(d,j,k)\text{ {is an admissible triple}}\}.
\end{equation*}
\end{problem}


The case of one hyperplane, $\Delta (j,1)=j$, is the famous ham sandwich
theorem, which is equivalent to the Borsuk--Ulam theorem. The equality $%
\Delta (2,2)=3$, and consequently $\Delta (1,3)=3$, was proven by Hadwiger
\cite{Hadw}. Ramos \cite{Ramos} gave a general lower bound for the function $%
\Delta $,%
\begin{equation}
\Delta (j,k)\geq \tfrac{2^{k}-1}{k}j.  \label{eq:LowerBound}
\end{equation}%
Recently, Mani-Levitska, Vre\'{c}ica and \v{Z}ivaljevi\'{c} \cite{Main-Sinisa-Rade}
applied Fadell--Husseini index theory for an elementary abelian subgroup $(\mathbb{Z}_2)^k$
of the Weyl group $W_k=(\mathbb{Z}_2)^k\rtimes S_k$ to obtain a new upper bound for the function $\Delta $,%
\begin{equation}
\Delta (2^{q}+r,k)\leq 2^{k+q-1}+r.  \label{eq:UpperBound}
\end{equation}%
In the case of $j=2^{l+1}-1$ measures and $k=2$ hyperplanes these bounds
yield the equality%
\begin{equation*}
\Delta (j,2)=\lceil \tfrac{3}{2}j\rceil .
\end{equation*}

\subsection{Statement of the main result ($\mathit{k=\mathrm{2}}$)}

This paper addresses Problem \ref{GruenbaumProblem} for $k=2$ using two
different but related Configuration Space/Test Map schemes (Section \ref%
{sec:CS/TM}, Proposition \ref{prop:TestSpace}).

\begin{compactitem}
\item The \textbf{product scheme} is the classical one, already considered
in \cite{guide2} and \cite{Main-Sinisa-Rade}. The problem is translated to
the problem of the existence of a $W_{k}$-equivariant map,%
\begin{equation*}
Y_{d,k}:=\left( S^{d}\right) ^{k}\longrightarrow S\left( \left(
R_{2^{k}}\right) ^{j}\right) ,
\end{equation*}%
where $W_{k}=(\mathbb{Z}_{2})^{k}\rtimes S_{k}$ is the Weyl group.

\item The \textbf{join scheme} is a new one. It connects the problem with
classical Borsuk--Ulam properties in the spirit of Marzantowicz \cite{Marza}%
. It asks the question whether there exists a $W_{k}$-equivariant map%
\begin{equation*}
X_{d,k}:=\left( S^{d}\right) ^{\ast k}\longrightarrow S\left( U_{k}\times
\left( R_{2^{k}}\right) ^{j}\right).
\end{equation*}%
The $W_{k}$-representations $R_{2^{k}}$ and $U_{k}$ are introduced in
Section \ref{Sec:TestMap}.
\end{compactitem}

\smallskip
\noindent Obstruction theory methods cannot be applied to either scheme
directly for $k>1$, since the $W_{k}$-actions on the respective
configuration spaces $\left( S^{d}\right) ^{k}$ and $\left( S^{d}\right)
^{\ast k}$ are \textbf{not free} (compare \cite[Section 2.3.3]%
{Main-Sinisa-Rade}, assumptions on the manifold $M^{n}$). Therefore we
analyze the associated equivariant question for $k=2$ via the
Fadell--Husseini ideal index theory method. We show that the join scheme
considered from the Fadell--Husseini point of view, with either $\mathbb{F}%
_{2}$ or $\mathbb{Z}$ coefficients, yields no obstruction to the existence of the
equivariant map in question (Remarks \ref{Rem:Fail-1} and \ref{Rem:Fail-2}%
). In the case of the product scheme we give the ideal
bounds obtained from the use of the full group of symmetries by proving the following theorem.


\begin{theorem}
\label{Th:Main}Let $\pi _{d}$, $d\geq 0$, be polynomials in $\mathbb{F}%
_{2}[y,w]$ given by%
\begin{equation*}
\pi _{d}(y,w)=\sum\limits_{i}\binom{d-1-i}{i}_{\mathrm{mod}\,2}~w^{i}y^{d-2i}
\end{equation*}%
and $\Pi _{d}$, $d\geq 0$, be polynomials in $\mathbb{Z}\lbrack \mathcal{Y},%
\mathcal{M},\mathcal{W}]/\langle 2\mathcal{Y},2\mathcal{M},4\mathcal{W},%
\mathcal{M}^{2}-\mathcal{WY}\rangle $ given by
\begin{equation*}
\Pi _{d}(\mathcal{Y},\mathcal{W})=\sum\limits_{i}\binom{d-1-i}{i}_{\mathrm{%
mod}\,2}~\mathcal{W}^{i}\mathcal{Y}^{d-2i}.
\end{equation*}%

\begin{compactenum}[\rm(A)]

\item $\mathbb{F}_{2}$-bound: The triple $(d,j,2)\in\mathbb{N}^{3}$ is admissible if
\begin{equation*}
y^{j}w^{j}\notin \langle \pi _{d+1},\pi _{d+2}\rangle \subseteq \mathbb{F}%
_{2}[y,w].
\end{equation*}

\item $\mathbb{Z}$-bound: The triple $(d,j,2)\in
\mathbb{N}^{3}$ is
admissible if%
\begin{equation*}
\left\langle
\begin{tabular}{ll}
\multicolumn{2}{l}{$(j-1)_{\mathrm{mod}\,2}~\mathcal{Y}^{\frac{j}{2}}%
\mathcal{W}^{\frac{j}{2}},$} \\
$j_{\mathrm{mod}\,2}~\mathcal{Y}^{\frac{j+1}{2}}\mathcal{W}^{\frac{j-1}{2}}%
\mathcal{M},$ & $j_{\mathrm{mod}\,2}~\mathcal{Y}^{\frac{j+1}{2}}\mathcal{W}^{%
\frac{j+1}{2}}$%
\end{tabular}%
\right\rangle \subseteq \left\langle
\begin{tabular}{ll}
$(d-1)_{\mathrm{mod}\,2~}\Pi _{\frac{d+2}{2}},$ & $(d-1)_{\mathrm{mod}%
\,2~}\Pi _{\frac{d+4}{2}},$ \\
$(d-1)_{\mathrm{mod}\,2}~\mathcal{M}\Pi _{\frac{d}{2}},$ &  \\
$d_{\mathrm{mod}\,2~}\Pi _{\frac{d+1}{2}},$ & $d_{\mathrm{mod}\,2~}\Pi _{%
\frac{d+3}{2}}$%
\end{tabular}%
\right\rangle
\end{equation*}
in the ring $\mathbb{Z}\lbrack \mathcal{Y},\mathcal{M},\mathcal{W}]/\langle 2\mathcal{Y},2\mathcal{M%
},4\mathcal{W},\mathcal{M}^{2}-\mathcal{WY}\rangle $.
\end{compactenum}
\end{theorem}

\begin{remark*}
\label{Rem:definition}Let $\hat{\Pi}_{d}$, $d\geq 0$, be the sequence of
polynomials in $\mathbb{Z}\mathfrak{[}Y,W]$ defined by $\hat{\Pi}_{0}=0$, $\hat{\Pi}_{1}=Y$ and $\hat{%
\Pi}_{d+1}=Y\hat{\Pi}_{d}+W\hat{\Pi}_{d-1}$ for $d\geq 2$. Then the
sequences of polynomials $\Pi _{d}$ and $\pi _{d}$ are reductions of the
polynomials $\hat{\Pi}_{d}$. The polynomials $\hat{\Pi}_{d}$ can be also
described by the generating function (formal power series)%
\begin{equation*}
\sum\limits_{d\geq 0}\hat{\Pi}_{d}=\frac{{Y}}{{1-Y-W}}
\end{equation*}%
where $\hat{\Pi}_{d}$ is homogeneous of degree $2d$ if we set $\deg (Y)=2$
and $\deg \left( W\right) =4$.
\end{remark*}


Theorem \ref{Th:Main} is a consequence of a topological result, the complete
and explicit computation of the relevant Fadell--Husseini indexes of the $%
D_{8}$-space $S^{d}\times S^{d}$ and the $D_{8}$-sphere $S(R_{4}^{\oplus j})$%
.

\begin{theorem}
\label{Th:Main-2} \qquad \newline
\begin{compactenum}[\rm(A)]

\item $\mathrm{Index}_{D_{8},\mathbb{F}_{2}}^{3j}S(R_{4}^{\oplus j})=\mathrm{%
Index}_{D_{8},\mathbb{F}_{2}}S(R_{4}^{\oplus j})=\langle y^{j}w^{j}\rangle.$

\item $\mathrm{Index}_{D_{8},\mathbb{F}_{2}}^{d+2}(S^{d}\times
S^{d})=\langle \pi _{d+1},\pi _{d+2}\rangle .$

\item $\mathrm{Index}_{D_{8},\mathbb{\mathbb{Z}}}^{3j+1}S(R_{4}^{\oplus j})=\left\{
\begin{array}{lll}
\langle \mathcal{Y}^{\frac{j}{2}}\mathcal{W}^{\frac{j}{2}}\rangle  & {,} &
{\text{for }j\text{ even,}} \\
\langle \mathcal{Y}^{\frac{j+1}{2}}\mathcal{W}^{\frac{j-1}{2}}\mathcal{M},%
\mathcal{Y}^{\frac{j+1}{2}}\mathcal{W}^{\frac{j+1}{2}}\rangle  & {,}  &{
\text{for }j\text{ odd}.}%
\end{array}%
\right.$

\item \textrm{\ }$\mathrm{Index}_{D_{8},\mathbb{\mathbb{Z}}}^{d+2}S^{d}\times S^{d}=\left\{
\begin{array}{lll}
\langle \Pi _{\frac{d+2}{2}},\Pi _{\frac{d+4}{2}},\mathcal{M}\Pi _{\frac{d}{2%
}}\rangle & {,} & { \text{for }}d{ \text{ even,}} \\

\langle \Pi _{\frac{d+1}{2}},\Pi _{\frac{d+3}{2}}\rangle &{,} & { \text{%
for }d\text{ odd.}}
\end{array}%
\right. $
\end{compactenum}
\end{theorem}

The sequence of Fadell--Husseini indexes will be introduced in Section \ref%
{Sec:Fadell--Husseini}. The actions of the dihedral group $D_{8}$ and the
definition of the representation space $R_{4}^{\oplus j}$ are given in
Section \ref{sec:CS/TM}. Even though it does not seem to have any relevance
to our study of Problem \ref{GruenbaumProblem}, the complete index $\mathrm{%
Index}_{D_{8},\mathbb{F}_{2}}(S^{d}\times S^{d})$ will also be computed in
the case of $\mathbb{F}_{2}$ coefficients,
\begin{equation}
\mathrm{Index}_{D_{8},\mathbb{F}_{2}}(S^{d}\times S^{d})=\langle \pi
_{d+1},\pi _{d+2},w^{d+1}\rangle .  \label{eqAlmostTh}
\end{equation}

\medskip

\begin{rremark*}
The preprint versions of this paper, posted on the arXiv in April 2007 and July 2008, arXiv0704.1943v1--v2, have been referenced in diverse applications: see Gonzalez and Landweber \cite{GoLa}, Adem and Reichstein \cite{AdemZi}, as well as \cite{B-Z}.
\end{rremark*}

\subsection{Proof overview}

The Problem \ref{GruenbaumProblem} about mass partitions by hyperplanes can
be connected with the problem of the existence of equivariant maps as
discussed in Section \ref{sec:CS/TM}, Proposition \ref{prop:TestSpace}. The
topological problems we face, about the existence of $W_{k}=(\mathbb{Z}_{2})^{k}\rtimes S_{k}$-equivariant maps, for the product / join schemes,
\begin{equation*}
\begin{array}{lll}
\left( S^{d}\right) ^{k}\longrightarrow S\left( R_{2^{k}}^{\oplus
j}\right) , &  & \left( S^{d}\right) ^{\ast k}\longrightarrow
S\left( U_{k}\times R_{2^{k}}^{\oplus j}\right) ,%
\end{array}%
\end{equation*}%
have to be treated with care because the actions of the Weyl groups $W_{k}$
are not free. Note that there is no naive Borsuk--Ulam theorem for fixed
point free actions. Indeed, in the case $k=2$ when $W_{2}=D_{8}$ there
exists a $W_{2}$-equivariant map \cite[Theorem 3.22, page 49]{Bart}%
\begin{equation*}
S\left( \left( V_{+-}\oplus V_{-+}\right) ^{10}\right) \longrightarrow
S\left( \left( U_{2}\oplus V_{--}\right) ^{8}\right)
\end{equation*}%
even though $\dim \left( V_{+-}\oplus V_{-+}\right) ^{10}>\dim \left(
U_{2}\oplus V_{--}\right) ^{8}$. The $W_{2}=D_{8}$-representations $%
V_{+-}\oplus V_{-+}$, $V_{--}$ and $U_{2}$ are introduced in Section \ref%
{Sec:TestMap}.

In this paper we focus on the case of $k=2$ hyperplanes. Theorem \ref%
{Th:Main} gives the best possible answer to the question about the existence
of $W_{2}=D_{8}$-equivariant maps
\begin{equation*}
S^{d}\times S^{d}\longrightarrow S(R_{4}^{\oplus j})
\end{equation*}%
from the point of view of Fadell--Husseini index theory (Section \ref%
{Sec:Fadell--Husseini}). We explicitly compute the relevant Fadell--Husseini
indexes with $\mathbb{F}_{2}$ and $\mathbb{Z}$ coefficients (Theorem \ref{Th:Main-2}, Sections \ref%
{Sec:IndexOfSpheresD8-F2}, \ref{Sec:IndexOfSpheresD8-Z}, \ref%
{Sec:IndexProductD8} and \ref{Sec:IndexProductD8-Z}). Then Theorem \ref%
{Th:Main} is a consequence of the basic index property, Proposition \ref%
{prop:basic}.

\medskip

The index of the sphere $S(R_{4}^{\oplus j})$, with $\mathbb{F}_{2}$
coefficients, is computed in Section \ref{Sec:IndexOfSpheresD8-F2} by

\begin{compactitem}
\item decomposition of the $D_{8}$-representation $R_{4}^{\oplus j}$ into a
sum of irreducible ones, and

\item computation of indexes of spheres of all irreducible $D_{8}$%
-representations.
\end{compactitem}

\noindent The main technical tool is the restriction diagram derived in
Section \ref{Sec:RestrictionDiagram}, which connects the indexes of the
subgroups of $D_{8}$.

\medskip
The index with $\mathbb{Z}$ coefficients is computed in Section \ref{Sec:IndexOfSpheresD8-Z} using

\begin{compactitem}
\item (for $j$ even) the results for $\mathbb{F}_{2}$ coefficients and
comparison of Serre spectral sequences, and

\item (for $j$ odd) the Bockstein spectral sequence combined with known
results for $\mathbb{F}_{2}$ coefficients and comparison of Serre spectral
sequences.
\end{compactitem}

\medskip

The index of the product $S^{d}\times S^{d}$ is computed in Sections \ref%
{Sec:IndexProductD8} and \ref{Sec:IndexProductD8-Z} by an explicit study of
the Serre spectral sequence associated with the fibration%
\begin{equation*}
S^{d}\times S^{d}\rightarrow \mathrm{E}D_{8}\times _{D_{8}}(S^{d}\times
S^{d})\rightarrow \mathrm{B}D_{8}.
\end{equation*}%
The major difficulty comes from non-triviality of the local coefficients in
the Serre spectral sequence. The computation of the spectral sequence with
non-trivial local coefficients is done by an independent study of $H^{\ast
}(D_{8},\mathbb{F}_{2})$-module and $H^{\ast }(D_{8},\mathbb{\mathbb{Z}})$-module structures of relevant rows in the Serre spectral sequence
(Sections \ref{Sec:Row} and \ref{Sec:Row-Z}).

\subsection{Evaluation of the index bounds}

\subsubsection{$\mathbb{F}_{2}$-evaluation}

It was pointed out to us by Sini\v{s}a Vre\'{c}ica that, with $\mathbb{F}_{2}$%
-coefficients, the $D_{8}$ index bound gives the same bounds as the $%
H_{1}=\left( \mathbb{\mathbb{Z}}_{2}\right) ^{2}$ index bound. This observation follows from the
implication
\begin{equation*}
a^{j}b^{j}(a+b)^{j}\in \langle a^{d+1},(a+b)^{d+1}\rangle ~\Rightarrow
~a^{j}b^{j}(a+b)^{j}\in \langle
a^{d+1}+(a+b)^{d+1},a^{d+2}+(a+b)^{d+2}\rangle .
\end{equation*}%
By introducing a new variable $c:=a+b$, it is enough to prove the implication%
\begin{equation}
a^{j}c^{j}(a+c)^{j}\in \langle a^{d+1},c^{d+1}\rangle ~\Rightarrow
~a^{j}c^{j}(a+c)^{j}\in \langle a^{d+1}+c^{d+1},a^{d+2}+c^{d+2}\rangle .
\label{v-1}
\end{equation}

\smallskip

\noindent Let us assume that $a^{j}c^{j}(a+c)^{j}\in \langle
a^{d+1},c^{d+1}\rangle $. The monomials in the expansion of $%
a^{j}c^{j}(a+c)^{j}$ always come in pairs
\begin{equation*}
a^{d+k}c^{3j-d-k}+c^{d+k}a^{3j-d-k}.
\end{equation*}%
This is also true when $j$ is even since $\binom{j}{j/2}=_{\mathrm{{mod}2}}0$
implies there are no middle terms. The sequence of equations%
\begin{equation*}
\begin{array}{lll}
a^{d+1}c^{3j-d-1}+c^{d+1}a^{3j-d-1} & = &
(a^{d+1}+c^{d+1})(c^{3j-d-1}+a^{3j-d-1})+a^{3j}+c^{3j} \\
a^{d+2}c^{3j-d-2}+c^{d+2}a^{3j-d-2} & = &
(a^{d+1}+c^{d+1})(ac^{3j-d-2}+a^{3j-d-2}c)+a^{3j-1}c+ac^{3j-1} \\
\ldots. &  &  \\
a^{3j}+c^{3j} & = &
(a^{d+2}+c^{d+2})(a^{3j-d-2}+c^{3j-d-2})+a^{d+2}c^{3j-d-2}+c^{d+2}a^{3j-d-2}%
\end{array}%
\end{equation*}%
shows that all the binomials
\begin{equation*}
a^{d+1}c^{3j-d-1}+c^{d+1}a^{3j-d-1},\text{ \quad }%
a^{d+2}c^{3j-d-2}+c^{d+2}a^{3j-d-2},\text{ \quad} \ldots \text{ \quad , }a^{3j}+c^{3j}
\end{equation*}%
belong to the ideal $\langle a^{d+1}+c^{d+1},a^{d+2}+c^{d+2}\rangle $ or
none of them do.

\noindent Since for $3j-d-1$ even
\begin{equation*}
\begin{array}{ccc}
a^{d+1+\tfrac{3j-d-1}{2}}c^{\tfrac{3j-d-1}{2}}+c^{d+1+\tfrac{3j-d-1}{2}}a^{%
\tfrac{3j-d-1}{2}} & = & (a^{d+1}+c^{d+1})a^{\tfrac{3j-d-1}{2}}c^{\tfrac{%
3j-d-1}{2}} \\
& \in & \langle a^{d+1}+c^{d+1},a^{d+2}+c^{d+2}\rangle%
\end{array}%
,
\end{equation*}%
and for $3j-d-1$ odd%
\begin{equation*}
\begin{array}{ccc}
a^{d+2+\tfrac{3j-d-2}{2}}c^{\tfrac{3j-d-2}{2}}+c^{d+2+\tfrac{3j-d-2}{2}}a^{%
\tfrac{3j-d-2}{2}} & = & (a^{d+2}+c^{d+2})a^{\tfrac{3j-d-2}{2}}c^{\tfrac{%
3j-d-2}{2}} \\
& \in & \langle a^{d+1}+c^{d+1},a^{d+2}+c^{d+2}\rangle%
\end{array}%
,
\end{equation*}%
the implication (\ref{v-1}) is proved.

\subsubsection{$\mathbb{\mathbb{Z}}$-evaluation}

More is true, even the complete $D_{8}$ index bound, now with $\mathbb{Z}$-coefficients, implies the same bounds
as does the subgroup $H_1=\left(\mathbb{Z}_2\right)^2$ for the $k=2$ hyperplanes mass
partition problem.

\smallskip

\begin{lemma}
\label{lemma:L-1}Let $a=\sum\limits_{i=1}^{k}a_{i}2^{i}$ and $%
b=\sum\limits_{i=1}^{k}b_{i}2^{i}$ be the dyadic expansions. Then%
\begin{equation*}
\binom{b}{a}_{\mathrm{mod}\,2}=\prod\limits_{i=1}^{k}\binom{b_{i}}{a_{i}}_{%
\mathrm{mod}\,2}.
\end{equation*}
\end{lemma}

\smallskip

\noindent This classical fact \cite{Lucas} about binomial coefficients mod $%
2 $ yields the following property for the sequence of polynomials $\Pi _{d}$%
, $d\geq 0$.

\smallskip

\begin{lemma}
\label{lemma:L-2}Let $q>0$ and $i$ be integers. Then
\begin{compactenum}[\rm(A)]
\item $\binom{2^{q}-1-i}{i}=\left\{
\begin{array}{cc}
0, & i\neq 0 \\
1, & i=0%
\end{array}%
\right. ,$

\item $\Pi _{2^{q}}=\mathcal{Y}^{2^{q}}.$
\end{compactenum}
\end{lemma}

\begin{proof}
The statement (B) is a direct consequence of the fact (A) and the
definition of polynomials $\Pi _{d}$. For $i\notin \{1,\ldots,2^{q-1}\}$ the
statement (A) is true from boundary conditions on binomial coefficients. Let
$i\in \{1,\ldots,2^{q-1}\}$ and $i=\sum\limits_{k\in I\subseteq
\{0,\ldots,q-1\}}2^{k}$. Then%
\begin{equation*}
\begin{array}{ccccc}
2^{q}-1-i & = & 2^{0}+2^{1}+2^{2}+\cdots+2^{q-1}-\sum\limits_{k\in I\subseteq
\{0,\ldots,q-1\}}2^{k} & = & \sum\limits_{k\in I^{c}\subseteq
\{0,\ldots,q-1\}}2^{k}%
\end{array}%
\end{equation*}%
where $I^{c}$ is the complementary index set in $\{0,\ldots,q-1\}$. The
statement (A) follows from Lemma \ref{lemma:L-1}
\end{proof}

\smallskip

\noindent Let $j$ be an integer such that $j=2^{q}+r$ where $0\leq r<2^{q}$
and $d=2^{q+1}+r-1$. Let us introduce the following ideals%
\begin{equation*}
A_{j}=\left\{
\begin{array}{ll}
\langle \mathcal{Y}^{\frac{j}{2}}\mathcal{W}^{\frac{j}{2}}\rangle , & {\
\text{ for }j\text{ even,}} \\
\langle \mathcal{Y}^{\frac{j+1}{2}}\mathcal{W}^{\frac{j-1}{2}}\mathcal{M},%
\mathcal{Y}^{\frac{j+1}{2}}\mathcal{W}^{\frac{j+1}{2}}\rangle , & {\ \text{
for }j\text{ odd,}}%
\end{array}%
\right. \quad \text{and}\quad B_{d}=\left\{
\begin{array}{ll}
\langle \Pi _{\frac{d+2}{2}},\Pi _{\frac{d+4}{2}},\mathcal{M}\Pi _{\frac{d}{2%
}}\rangle , & {\ \text{ for }d\text{ even,}} \\
\langle \Pi _{\frac{d+1}{2}},\Pi _{\frac{d+3}{2}}\rangle , & {\ \text{ for }d%
\text{ odd.}}%
\end{array}%
\right.
\end{equation*}%
The fact that the $D_{8}$ index bound with $\mathbb{\mathbb{Z}}$-coefficients does not improve the mass partition bounds
obtained by using the subgroup $H_1=\left(\mathbb{Z}_2\right)^2$ is a consequence of the following facts:

\begin{compactitem}
\item $r=0~\Rightarrow ~A_{j}\subseteq B_{d},$

\item $\left( r\neq 2^{q}-1~\text{and}~A_{j}\subseteq B_{d}\right)
~\Longrightarrow ~A_{j+1}\subseteq B_{d+1},$
\end{compactitem}
that are proved in Lemma \ref{lemma:L3} and \ref{lemma:L4}, respectively.

\smallskip

\begin{lemma}
\label{lemma:L3}$\langle \mathcal{Y}^{2^{q-1}}\mathcal{W}^{2^{q-1}}\rangle
=A_{2^{q}}\subseteq B_{2^{q+1}-1}=\langle \Pi _{2^{q}},\Pi _{2^{q}+1}\rangle
.$
\end{lemma}

\begin{proof}
Since $\mathcal{Y}^{2^{q-1}}=\Pi _{2^{q-1}}$ by Lemma \ref{lemma:L-2},%
\begin{equation*}
\mathcal{Y}^{2^{q-1}}\mathcal{W~}=~\Pi _{2^{q-1}}\mathcal{W~}=~\Pi
_{2^{q-1}+2}+\mathcal{Y}\Pi _{2^{q-1}+1}~\in ~\langle \Pi _{2^{q-1}+1},\Pi
_{2^{q-1}+2}\rangle .
\end{equation*}%
By induction on the power $i$ of $\mathcal{W}$ in $\mathcal{Y}^{2^{q-1}}%
\mathcal{W}^{2i}$,%
\begin{equation*}
\mathcal{Y}^{2^{q-1}}\mathcal{W}^{i}\in \langle \Pi _{2^{q-1}+i},\Pi
_{2^{q-1}+i+1}\rangle ,
\end{equation*}%
and consequently%
\begin{equation*}
\mathcal{Y}^{2^{q-1}}\mathcal{W}^{2^{q-1}}\in \langle \Pi _{2^{q}},\Pi
_{2^{q}+1}\rangle .
\end{equation*}
\end{proof}

\smallskip

\begin{lemma}
\label{lemma:L4}If $r\neq 2^{q}-1$ and $A_{j}\subseteq B_{d}$ then $%
A_{j+1}\subseteq B_{d+1}$.
\end{lemma}

\begin{proof}
We distinguish two cases depending on the parity of $j$.
\begin{compactenum}[\rm(A)]

\item Let $j$ be even and $\mathcal{Y}^{\frac{j}{2}}\mathcal{W}^{\frac{j}{2}%
}\in \langle \Pi _{\frac{d+1}{2}},\Pi _{\frac{d+3}{2}}\rangle $. There are
polynomials $\alpha $ and $\beta $ such that%
\begin{equation*}
\mathcal{Y}^{\frac{j}{2}}\mathcal{W}^{\frac{j}{2}}=\alpha \Pi _{\frac{d+1}{2}%
}+\beta \Pi _{\frac{d+3}{2}}.
\end{equation*}%
Then%
\begin{eqnarray*}
\mathcal{Y}^{\frac{(j+1)+1}{2}}\mathcal{W}^{\frac{(j+1)-1}{2}}\mathcal{M} &=&%
\mathcal{Y}^{\frac{j+2}{2}}\mathcal{W}^{\frac{j}{2}}\mathcal{M}=\mathcal{YM}%
\left( \alpha \Pi _{\frac{d+1}{2}}+\beta \Pi _{\frac{d+3}{2}}\right) \\
&\in &\langle \Pi _{\frac{(d+1)+2}{2}},\mathcal{M}\Pi _{\frac{d+1}{2}%
}\rangle \subseteq \langle \Pi _{\frac{d+3}{2}},\Pi _{\frac{d+5}{2}},%
\mathcal{M}\Pi _{\frac{d+1}{2}}\rangle =B_{d+1},
\end{eqnarray*}%
and%
\begin{eqnarray*}
\mathcal{Y}^{\frac{(j+1)+1}{2}}\mathcal{W}^{\frac{(j+1)+1}{2}} &=&\mathcal{YW%
}\left( \mathcal{Y}^{\frac{j}{2}}\mathcal{W}^{\frac{j}{2}}\right) =\mathcal{%
YW}\left( \alpha \Pi _{\frac{d+1}{2}}+\beta \Pi _{\frac{d+3}{2}}\right)
=\alpha \mathcal{M}^{2}\Pi _{\frac{d+1}{2}}+\beta \mathcal{YW}\Pi _{\frac{d+3%
}{2}} \\
&\in &\langle \mathcal{M}\Pi _{\frac{d+1}{2}},\Pi _{\frac{d+3}{2}}\rangle
\subseteq \langle \Pi _{\frac{d+3}{2}},\Pi _{\frac{d+5}{2}},\mathcal{M}\Pi _{%
\frac{d+1}{2}}\rangle =B_{d+1}.
\end{eqnarray*}%
Thus $A_{j+1}\subseteq B_{d+1}$.

\item Let $j$ be odd and
\begin{equation*}
\langle \mathcal{Y}^{\frac{j+1}{2}}\mathcal{W}^{\frac{j-1}{2}}\mathcal{M},%
\mathcal{Y}^{\frac{j+1}{2}}\mathcal{W}^{\frac{j+1}{2}}\rangle
=A_{j}\subseteq B_{d}=\langle \Pi _{\frac{d+2}{2}},\Pi _{\frac{d+4}{2}},%
\mathcal{M}\Pi _{\frac{d}{2}}\rangle .
\end{equation*}%
There are polynomials $\alpha $, $\beta $ and $\gamma $ such that
\begin{equation*}
\mathcal{Y}^{\frac{j+1}{2}}\mathcal{W}^{\frac{j+1}{2}}=\alpha \Pi _{\frac{d+2%
}{2}}+\beta \Pi _{\frac{d+4}{2}}+\gamma \mathcal{M}\Pi _{\frac{d}{2}}
\end{equation*}%
and no occurrence of the defining relation $\Pi _{\frac{d+4}{2}}=%
\mathcal{Y}\Pi _{\frac{d+2}{2}}+\mathcal{W}\Pi _{\frac{d}{2}}$, Remark \ref{Rem:definition}, can be
subtracted from the presentation. Then $\gamma \mathcal{M}\Pi _{\frac{d}{2}%
}\in \langle \Pi _{\frac{d+2}{2}},\Pi _{\frac{d+4}{2}}\rangle $, and since $%
\mathcal{M}$ is of odd degree $\gamma =\mathcal{M}\gamma ^{\prime }$. In the
first case the inclusion $A_{j+1}\subseteq B_{d+1}$ follows directly.
Consider $\gamma =\mathcal{M}\gamma ^{\prime }$. Since $(\mathcal{Y}+%
\mathcal{X})\mathcal{W}\Pi _{i}=\mathcal{YW}\Pi _{i}$ for every $i>0$, we
have that%
\begin{eqnarray*}
\mathcal{Y}^{\frac{j+1}{2}}\mathcal{W}^{\frac{j+1}{2}} &=&\alpha \Pi _{\frac{%
d+2}{2}}+\beta \Pi _{\frac{d+4}{2}}+\gamma ^{\prime }\mathcal{M}^{2}\Pi _{%
\frac{d}{2}}=\alpha \Pi _{\frac{d+2}{2}}+\beta \Pi _{\frac{d+4}{2}}+\gamma
^{\prime }\mathcal{YW}\Pi _{\frac{d}{2}} \\
&=&\alpha \Pi _{\frac{d+2}{2}}+\beta \Pi _{\frac{d+4}{2}}+\gamma ^{\prime }%
\mathcal{Y}(\mathcal{Y}\Pi _{\frac{d}{2}+1}+\Pi _{\frac{d}{2}+2})\in \langle
\Pi _{\frac{d+2}{2}},\Pi _{\frac{d+4}{2}}\rangle =B_{d+1}.
\end{eqnarray*}

Thus $A_{j+1}\subseteq B_{d+1}$.

\end{compactenum}
\end{proof}

\subsubsection*{Acknowledgements}

We are grateful to Jon Carlson and to Carsten Schultz for useful comments
and insightful observations. The referee provided many useful suggestions and comments that are incorporated in the latest version of the
manuscript.

Some of this work was done in the framework of the MSRI program
\textquotedblleft Computational Applications of Algebraic
Topology\textquotedblright\ in the fall semester 2006.

\section{\label{sec:CS/TM}Configuration space/Test map scheme}

The Configuration Space/Test Map (CS/TM) paradigm (formalized by \v{Z}%
ivaljevi\'{c} in \cite{Zivaljevic-topmeth}, and also beautifully exposited
by Matou\v{s}ek in \cite{MatousekBZ:BU}) has been very powerful in the
systematic derivation of topological lower bounds for problems of
Combinatorics and of Discrete Geometry.

In many instances, the problem suggests natural configuration spaces $X$, $Y$%
, a finite symmetry group~$G$, and a test set $Y_{0}\subset Y$, where one
would try to show that every $G$-equivariant map $f:X\rightarrow Y$ must hit~%
$Y_{0}$. The canonical tool is then Dold's theorem, which says that if the
group actions are free, then the map $f$ must hit the test set $Y_{0}\subset
Y$ if the connectivity of~$X$ is higher than the dimension of $Y\setminus
Y_{0}$.

For the success of this \textquotedblleft canonical
approach\textquotedblright\ one crucially needs that a result such as Dold's
theorem is applicable. Thus the group action must be free, so one often
reduces the group action to a prime order cyclic subgroup of the full
symmetry group, and results may follow only in \textquotedblleft the prime
case\textquotedblright\ , or with more effort and deeper tools in the prime
power case. The main example for this is the Topological Tverberg Problem,
which is still not resolved for $(d,q)$ if $d>1$ and $q$ is not a prime
power \cite[Section 6.4, page 165]{MatousekBZ:BU}. So in general one has to
work much harder when the \textquotedblleft canonical\textquotedblright\
approach fails.

In the following, we present configuration spaces and test maps for the mass
partition problem.

\subsection{Configuration space}

The space of all oriented affine hyperplanes in $\mathbb{R}^{d}$ can be naturally identified with the subspace of the sphere $S^{d}$ obtained by
removing two points, namely the \textquotedblleft oriented hyperplanes at infinity\textquotedblright . Indeed, let
$\mathbb{R}^{d}$ be embedded in $\mathbb{R}^{d+1}$ by $(x_{1},\ldots,x_{d})\longmapsto (x_{1},\ldots,x_{d},1)$. Then every
oriented affine hyperplane $H$ in $\mathbb{R}^{d}$ determines a unique oriented hyperplane $\tilde{H}$ through the origin
in $\mathbb{R}^{d+1}$ such that $\tilde{H}\cap\mathbb{R}^{d}=H$, and conversely if the hyperplane at infinity is included. The
oriented hyperplane uniquely determined by the unit vector $v\in S^{d}$ is
denoted by $H_{v}$ and the assumed orientation is determined by the
half-space $H_{v}^{\pm }$. Then $H_{-v}^{-}=H_{v}^{+}$. The obvious and
classically used candidate for the configuration space associated with the
problem of testing admissibility of $(d,j,k)$ is
\begin{equation*}
Y_{d,k}=\left( S^{d}\right) ^{k}.
\end{equation*}%
The relevant group acting on this space is the Weyl group $W_{k}=(
\mathbb{Z}_{2})^{k}\rtimes S_{k}$. Each $\mathbb{Z}_{2}=\left( \{+1,-1\},\cdot \right) $ acts antipodally on the appropriate
copy of $S^{d}$ (changing the orientation of hyperplanes), while $S_{k}$ acts
by permuting copies. The second configuration space that we can use is%
\begin{equation*}
X_{d,k}=\underset{k\text{ copies}}{\underbrace{S^{d}\ast \cdots\ast S^{d}}}%
~\cong ~S^{dk+k-1}.
\end{equation*}%
The elements of $X_{d,k}$ are denoted by $t_{1}v_{1}+ \cdots +t_{k}v_{k}$, with $%
t_{i}\geq 0$, $\sum t_{1}=1$, $v_{i}\in S^{d}$. The Weyl group $W_{k}$ acts
on $X_{d,k}$ by%
\begin{eqnarray*}
\varepsilon _{i}\cdot \left( t_{1}v_{1}+\cdots +t_{i}v_{i}+\cdots +t_{k}v_{k}\right)
&=&t_{1}v_{1}+\cdots +t_{i}(-v_{i})+\cdots +t_{k}v_{k}\text{,} \\
\pi \cdot \left( t_{1}v_{1}+\cdots +t_{i}v_{i}+\cdots +t_{k}v_{k}\right) &=&t_{\pi
^{-1}(1)}v_{\pi ^{-1}(1)}+\cdots +t_{\pi ^{-1}(i)}v_{\pi ^{-1}(i)}+\cdots +t_{\pi
^{-1}(k)}v_{\pi ^{-1}(k)},
\end{eqnarray*}%
where $\varepsilon _{i}$ is the generator of the $i$-th copy of $\mathbb{Z}_{2}$ and $\pi \in S_{k}$ is an arbitrary permutation.

\subsection{\label{Sec:TestMap}Test map}

Let $\mathcal{M}=\{\mu _{1},\ldots,\mu _{j}\}$ be a collection of mass
distributions in $\mathbb{R}^{d}$. Let the coordinates of $\mathbb{R}^{2^{k}}$ be indexed by the elements of the group
$(\mathbb{Z}_{2})^{k}$. The Weyl group $W_{k}$ acts on $\mathbb{R}^{2^{k}}$ by acting on its coordinate index set
$(\mathbb{Z}_{2})^{k}$ in the following way:
\begin{equation*}
\left( \left( \beta _{1},\ldots,\beta _{k}\right) \rtimes \pi \right) \cdot
\left( \alpha _{1},\ldots,\alpha _{k}\right) =\left( \beta _{1}\alpha _{\pi
^{-1}(1)},\ldots,\beta _{k}\alpha _{\pi ^{-1}(k)}\right) .
\end{equation*}%
The test map $\phi :Y_{d,k}\rightarrow (\mathbb{R}^{2^{k}})^{j}$ used with the configuration space $Y_{d,k}$ is a $W_{k}$%
-equivariant map given by
\begin{equation*}
\phi \left( v_{1},\ldots,v_{k}\right) =\left( \left( \mu _{i}(H_{v_{1}}^{\alpha
_{1}}\cap\cdots\cap H_{v_{k}}^{\alpha _{k}})-\tfrac{1}{2^{k}}\mu _{i}(\mathbb{R}
^{d})\right) _{(\alpha _{1},\ldots,\alpha _{k})\in (\mathbb{Z}_{2})^{k}}\right) _{i\in \{1,\ldots,j\}}.
\end{equation*}%
Denote the $i$-th component of $\phi $ by $\phi _{i}$, $i=1,\ldots,j$.


To define a test map associated with the configuration space $X_{d,k}$, we
discuss the $(\mathbb{Z}_{2})^{k}$- and $W_{k}$-module structures on $\mathbb{R}^{2^{k}}$.

All irreducible representations of the group $(\mathbb{Z}_{2})^{k}$ are $1$-dimensional. They are in bijection with the homomorphisms
(characters) $\chi :(\mathbb{Z}_{2})^{k}\rightarrow\mathbb{Z}_{2}$. These homomorphisms are completely determined by the values on
generators $\varepsilon _{1}$,\ldots,$\varepsilon _{k}$ of $(\mathbb{Z}_{2})^{k}$, i.e. by the vector $(\chi (\varepsilon _{1}),\ldots,\chi
(\varepsilon _{k}))$. For $(\alpha _{1},\ldots,\alpha _{k})\in (\mathbb{Z}_{2})^{k}$ let $V_{\alpha _{1}\ldots\alpha _{k}}=\mathrm{span}\{v_{\alpha
_{1}\ldots\alpha _{k}}\}\subset\mathbb{R}^{2^{k}}$ denote the $1$-dimensional representation given by
\begin{equation*}
\varepsilon _{i}\cdot v_{\alpha _{1}\ldots\alpha _{k}}=\alpha _{i}~v_{\alpha
_{1}\ldots\alpha _{k}}
\end{equation*}%
The vector $v_{\alpha _{1}\ldots\alpha _{k}}\in \{+1,-1\}^{2^{k}}$ is uniquely
determined up to a scalar multiplication by $-1$. Note that%
\begin{equation*}
\langle v_{\alpha _{1}\ldots\alpha _{k}},v_{\beta _{1}\ldots\beta _{k}}\rangle =0
\end{equation*}%
for $\alpha _{1}\ldots\alpha _{k}\neq \beta _{1}\ldots\beta _{k}$. For $k=2$, with
the abbreviation $+$ for $+1$, $-$ for $-1$, the coordinate index set for $\mathbb{R}^{4}$ is $\{++,+-,-+,--\}$. Then%
\begin{equation*}
\begin{array}{lll}
v_{++}=(1,1,1,1) & , & v_{+-}=(1,-1,1,-1), \\
v_{-+}=(1,1,-1,-1) & , & v_{--}=(1,-1,-1,1).%
\end{array}%
\end{equation*}

The following decomposition of $(\mathbb{Z}_{2})^{k}$-modules holds, with the index identification $(\mathbb{Z}_{2})^{k}=\{+,-\}^{k}$,
\begin{equation*}
\mathbb{R}^{2^{k}}\cong V_{+\cdots+}\oplus \sum\limits_{\alpha _{1}\ldots\alpha _{k}\in (\mathbb{Z}
_{2})^{k}\backslash \{+\cdots+\}}V_{\alpha _{1}\ldots\alpha _{k}}
\end{equation*}%
where $V_{+\cdots+}$ is the trivial $(\mathbb{Z}_{2})^{k}$-representation. Let $R_{2^{k}}$ denote the orthogonal complement
of $V_{+\cdots+}$ and $\pi :\mathbb{R}^{2^{k}}\rightarrow R_{2^{k}}$ the associated (equivariant) projection.
Explicitly
\begin{equation}
R_{2^{k}}=\{(x_{1},\ldots,x_{2^{k}})\in
\mathbb{R}^{2^{k}}~|~\sum x_{i}=0\}=\sum\limits_{\alpha _{1}\ldots\alpha _{k}\in (
\mathbb{Z}_{2})^{k}\backslash \{+\cdots+)\}}V_{\alpha _{1}\ldots\alpha _{k}},
\label{decompostition-1}
\end{equation}%
and%
\begin{equation*}
\mathbf{x}=(x_{1},\ldots,x_{2^{k}})\overset{\pi }{\longmapsto }\tfrac{1}{2^{k-1}}%
\left( \langle \mathbf{x},v_{\alpha _{1}\ldots\alpha _{k})}\rangle \right)
_{\alpha _{1}\ldots\alpha _{k}\in (\mathbb{Z}_{2})^{k}\backslash \{+\cdots+\}},
\end{equation*}%
where $\langle \cdot ,\cdot \rangle $ denotes the standard inner product of $
\mathbb{R}^{2^{k}}$. Observe that
\begin{equation*}
\mathrm{im}~\phi =\phi (Y_{d,k})\subseteq \left( R_{2^{k}}\right) ^{j}\text{.}
\end{equation*}%
Let $\alpha _{1}\ldots\alpha _{k}\in (\mathbb{Z}_{2})^{k}$ and let $\eta (\alpha _{1}\ldots\alpha _{k})=\tfrac{1}{2}\left(
k-\sum \alpha _{i}\right) $. The following decomposition of $W_{k}$-modules
holds%

\begin{equation}
\mathbb{R}^{2^{k}}\cong V_{+\cdots +}\oplus \sum\limits_{n=1}^{k}~\sum\limits_{n=\eta
(\alpha _{1},\ldots,\alpha _{k})}V_{\alpha _{1}\ldots\alpha _{k}}\cong
V_{+\cdots+}\oplus R_{2^{k}}.  \label{decompostition-2}
\end{equation}%
The test map $\tau :X_{d,k}\rightarrow U_{k}\times \left( R_{2^{k}}\right)
^{j}$ is defined by%
\begin{equation*}
\begin{array}{ll}
\tau \left( t_{1}v_{1}+\cdots +t_{k}v_{k}\right) = & \left( t_{1}-\tfrac{1}{k}%
,\ldots,t_{k}-\tfrac{1}{k}\right) \times \\
& \left( \left( t_{1}^{\tfrac{1-\alpha _{1}}{2}}\cdots t_{k}^{\tfrac{1-\alpha
_{k}}{2}}\langle \phi _{i}\left( v_{1},\ldots,v_{k}\right) ,v_{\alpha
_{1}\ldots \alpha _{k}}\rangle \right) _{\alpha _{1}\ldots \alpha _{k}\in
(\mathbb{Z}_{2})^{k}\backslash \{+\cdots +\}}\right) _{i=1}^{j}.%
\end{array}%
\end{equation*}%
Here $U_{k}=\{(\xi _{1},\ldots,\xi _{k})\in
\mathbb{R}^{k}~|~\sum \xi _{i}=0\}$ is a $W_{k}$-module with an action given by%
\begin{equation*}
\left( \left( \beta _{1},\ldots,\beta _{k}\right) \rtimes \pi \right) \cdot
\left( \xi _{1},\ldots,\xi _{k}\right) :=\left( \xi _{\pi ^{-1}(1)},\ldots,\xi _{\pi
^{-1}(k)}\right) .
\end{equation*}%
The subgroup $(\mathbb{Z}_{2})^{k}$ acts trivially on $U_{k}$. The action on $U_{k}\times \left(
R_{2^{k}}\right) ^{j}$ is assumed to be the diagonal action. The test map $\tau $ is well defined, continuous and $W_{k}$-equivariant.

\begin{example}
The test map $\tau :X_{d,k}\rightarrow U_{k}\times \left( R_{2^{k}}\right)
^{j}$ is in the case of $k=2$ hyperplanes and $j=1$ measure given by $\tau
:X_{d,2}\rightarrow U_{2}\times R_{4}=U_{2}\times \left( \left( V_{+-}\oplus
V_{-+}\right) \oplus V_{--}\right) $ and
\begin{equation*}
\begin{array}{ll}
\tau \left( t_{1}v_{1}+t_{2}v_{2}\right) = & \left( t_{1}-\tfrac{1}{2},t_{2}-%
\tfrac{1}{2},\right. \\
& \left. t_{1}\langle \phi \left( v_{1},v_{2}\right) ,v_{-+}\rangle
,t_{2}\langle \phi \left( v_{1},v_{2}\right) ,v_{+-}\rangle
,t_{1}t_{2}\langle \phi \left( v_{1},v_{2}\right) ,v_{--}\rangle \right)%
\end{array}%
\end{equation*}%
where%
\begin{equation*}
\phi \left( v_{1},v_{2}\right) =\left( \mu _{i}(H_{v_{1}}^{\alpha _{1}}\cap
H_{v_{2}}^{\alpha _{2}})-\tfrac{1}{4}\mu (\mathbb{R}^{d})\right) _{\alpha _{1}\alpha _{2}\in
(\mathbb{Z}_{2})^{2}}\in\mathbb{R}^{4}\text{.}
\end{equation*}
\end{example}

\subsection{The test space}

The test spaces for the maps $\phi $ and $\tau $ are the origins of $\left(
R_{2^{k}}\right) ^{j}$ and $U_{k}\times \left( R_{2^{k}}\right) ^{j}$,
respectively. The constructions that we perform in this section satisfy the
usual hypotheses for the CS/TM scheme.

\begin{proposition}
\label{prop:TestSpace}
\begin{compactenum}[\rm(i)]
\item For a collection of mass distributions $%
\mathcal{M}=\{\mu _{1},\ldots,\mu _{j}\}$ let $\phi
:Y_{d,k}\rightarrow \left( R_{2^{k}}\right) ^{j}$ and $\tau
:X_{d,k}\rightarrow U_{k}\times \left( R_{2^{k}}\right) ^{j}$ be
the corresponding test maps. If
\begin{equation*}
(0,\ldots,0)\in \phi \left( Y_{d,k}\right) \text{ \qquad or }\qquad
(0,\ldots,0)\in \tau \left( X_{d,k}\right)
\end{equation*}%
then there exists an arrangement of $k$ hyperplanes $\mathcal{H}$ \textit{in
}$\mathbb{R}^{d}$ \textit{equiparting} the collection $\mathcal{M}$.

\item If there is no $W_{k}$-equivariant map with respect to the
actions defined above,%
\begin{equation*}
\begin{array}{llll}
Y_{d,k}\rightarrow \left( R_{2^{k}}\right) ^{j}\backslash \{(0,\ldots,0)\}, &
\text{or} & Y_{d,k}\rightarrow S\left( \left( R_{2^{k}}\right) ^{j}\right)
\approx S^{j(2^{k}-1)-1}, & \text{or} \\
X_{d,k}\rightarrow U_{k}\times \left( R_{2^{k}}\right)
^{j}\backslash \{(0,\ldots,0)\}, & \text{or} & X_{d,k}\rightarrow
S\left( U_{k}\times \left( R_{2^{k}}\right) ^{j}\right) \approx
S^{j(2^{k}-1)+k-2}, &
\end{array}%
\end{equation*}%
then the triple $(d,j,k)$ is admissible.

\item Specifically, for $k=2$, if there is no $D_{8}\cong W_{2}$
equivariant map, with the already defined actions,%
\begin{equation*}
\begin{array}{llll}
Y_{d,2}\rightarrow \left( R_{4}\right) ^{j}\backslash \{(0,\ldots,0)\}, & \text{%
or} & Y_{d,2}\rightarrow S\left( \left( R_{4}\right) ^{j}\right) \approx
S^{3j-1}, & \text{or} \\
X_{d,2}\rightarrow U_{2}\times \left( R_{4}\right) ^{j}\backslash
\{(0,\ldots,0)\}, & \text{or} & S^{2d+1}\approx X_{d,2}\rightarrow S\left(
U_{2}\times \left( R_{4}\right) ^{j}\right) \approx S^{3j}, &
\end{array}%
\end{equation*}%
then the triple $(d,j,2)$ is admissible.
\end{compactenum}
\end{proposition}

\begin{remark*}
{The action of $W_{k}$ on the sphere $S(U_{2}\times (R_{4}) ^{j}) $ is
\emph{fixed point free, but not free.} For $k=2 $, the action of the unique $
\mathbb{Z}_{4}$ subgroup of $W_{2}=D_{8}$ on the sphere $S(U_{2}\times (R_{4}) ^{j}) $
is fixed point free. }
\end{remark*}

The necessary condition for the non-existence of an equivariant $W_{k}$-map
\begin{equation*}
X_{d,k}\rightarrow S(U_{k}\times (R_{2^{k}})^{j})
\end{equation*}%
implied by the equivariant Kuratowski--Dugundji theorem \cite[Theorem 1.3,
page 25]{Bal-Kush} is
\begin{equation}
\begin{array}{ccc}
dk+k-1>j(2^{k}-1)+k-2 & \Longleftrightarrow & d\geq \tfrac{2^{k}-1}{k}j%
\end{array}%
.  \label{eq:NecesseryCondtion-1}
\end{equation}%
For $k=2$ the condition (\ref{eq:NecesseryCondtion-1}) becomes%
\begin{equation}
d\geq \lceil \tfrac{3}{2}j\rceil .  \label{eq:NecesseryCondtion-2}
\end{equation}%

\section{\label{Sec:Fadell--Husseini}The Fadell--Husseini index theory}

\subsection{Equivariant cohomology}

Let $X$ be a $G$ -space and $X\rightarrow \mathrm{E}G\times _{G}X\overset{%
\pi _{X}}{\rightarrow }\mathrm{B}G$ the associated universal bundle, with $X$
as a typical fibre. $\mathrm{E}G$ is a contractible cellular space on which $%
G$ acts freely, and $\mathrm{B}G:=\mathrm{E}G/G$. The space $\mathrm{E}%
G\times _{G}X=\left( \mathrm{E}G\times X\right) /G$ is called the \emph{%
Borel construction} of $X$ with respect to the action of $G$. The \emph{%
equivariant cohomology} of $X$ is the ordinary cohomology of the Borel
construction $\mathrm{E}G\times _{G}X$,%
\begin{equation*}
H_{G}^{\ast }(X):=H^{\ast }(\mathrm{E}G\times _{G}X).
\end{equation*}%
The equivariant cohomology is a module over the ring $H_{G}^{\ast
}(\mathrm{pt})=H^{\ast }(\mathrm{B}G)$. When $X$ is a free $G$-space the homotopy
equivalence $\mathrm{E}G\times _{G}X\simeq X/G$ induces a natural isomorphism%
\begin{equation*}
H_{G}^{\ast }(X)\cong H^{\ast }(X/G).
\end{equation*}%
The universal bundle $X\rightarrow \mathrm{E}G\times _{G}X\overset{\pi _{X}}{%
\rightarrow }\mathrm{B}G$, for coefficients in the ring $R$, induces a Serre
spectral sequence converging to the graded group $\mathrm{Gr}(H_{G}^{\ast
}(X,R))$ associated with $H_{G}^{\ast }(X,R)$ appropriately filtered. In this
paper \textquotedblleft ring\textquotedblright\ means commutative ring with
a unit element. The $E_{2}$-term is given by%
\begin{equation}
E_{2}^{p,q}\cong H^{p}(\mathrm{B}G,\mathcal{H}^{q}(X,R)),
\label{eq:E2Term-1}
\end{equation}%
where $\mathcal{H}^{q}(X,R)$ is a system of local coefficients. For a
discrete group $G$, the $E_{2}$-term of the spectral sequence can be
interpreted as the cohomology of the group $G$ with coefficients in the $G$%
-module $H^{\ast }(X,R)$,
\begin{equation}
E_{2}^{p,q}\cong H^{p}(G,H^{q}(X,R)).  \label{eq:E2Term-2}
\end{equation}

\subsection{$\mathrm{Index}_{G,R}$ and $\mathrm{Index}_{G,R}^{k}$}

Let $X$ be a $G$-space, $R$ a ring and $\pi _{X}^{\ast }$ the ring
homomorphism in cohomology%
\begin{equation*}
\pi _{X}^{\ast }:H^{\ast }(\mathrm{B}G,R)\rightarrow H^{\ast }(\mathrm{E}%
G\times _{G}X,R)
\end{equation*}%
induced by the projection $\mathrm{E}G\times_G X\to\mathrm{E}G\times_G \mathrm{pt}\approx\mathrm{B}G$.

The \emph{Fadell--Husseini (ideal-valued) index} of a $G$-space $X$ is the
kernel ideal of $\pi _{X}^{\ast }$,
\begin{equation*}
\mathrm{Index}_{G,R}X:=\ker \pi _{X}^{\ast }\subseteq H^{\ast }(\mathrm{B}%
G,R)\text{.}
\end{equation*}%
The Serre spectral sequence (\ref{eq:E2Term-1}) yields a representation of
the homomorphism $\pi _{X}^{\ast }$ as the composition%
\begin{equation*}
H^{\ast }(\mathrm{B}G,R)\rightarrow E_{2}^{\ast ,0}\rightarrow E_{3}^{\ast
,0}\rightarrow E_{4}^{\ast ,0}\rightarrow \cdots\rightarrow E_{\infty }^{\ast
,0}\subseteq H^{\ast }(\mathrm{E}G\times _{G}X,R).
\end{equation*}%
The $k$\textbf{-}\emph{th Fadell--Husseini index} is defined by%
\begin{eqnarray*}
\mathrm{Index}_{G,R}^{k}X &=&\ker \big(H^{\ast }(\mathrm{B}G,R)\rightarrow
E_{k}^{\ast ,0}\big),\qquad k\geq 2, \\
\mathrm{Index}_{G,R}^{1}X &=&\{0\}.
\end{eqnarray*}%
From the definitions the following properties of indexes can be derived.

\begin{proposition}
\label{prop:GenralizedIndex}Let $X$, $Y$ be $G$-spaces.
\begin{compactenum}[\rm(1)]
\item  $\mathrm{Index}_{G,R}^{k}X\subseteq H^{\ast }(\mathrm{B}G,R )$
is an ideal, for every $k\in\mathbb{N}$;
\item  $\mathrm{Index}_{G,R}^{1}X\subseteq \mathrm{Index}%
_{G,R}^{2}X\subseteq \mathrm{Index}_{G,R}^{3}X\subseteq \cdots\subseteq \mathrm{%
Index}_{G,R}X$;
\item  $\bigcup_{k\in\mathbb{N}
}\mathrm{Index}_{G,R}^{k}X=\mathrm{Index}_{G,R}X$.
\end{compactenum}
\end{proposition}

\begin{proposition}
\label{prop:basic}Let $X$ and $Y$ be $G$-spaces and $f:X\rightarrow Y$ a $G$%
-map. Then
\begin{equation*}
\mathrm{Index}_{G,R}(X)\supseteq \mathrm{Index}_{G,R}(Y)
\end{equation*}%
and for every $k\in\mathbb{N}$
\begin{equation*}
\mathrm{Index}_{G,R}^{k}(X)\supseteq \mathrm{Index}_{G,R}^{k}(Y).
\end{equation*}
\end{proposition}

\begin{proof}
Functoriality of all constructions implies that the following diagrams
commute:
\begin{diagram}
X & &\rTo{f}& & Y & ~~& \mathrm{E}G\times _{G}X & &\rTo{\hat{f}}& & \mathrm{E}G\times _{G}Y &~~&\\
& \rdTo&&\ldTo && ~~& & \rdTo{\pi_X}&&\ldTo{\pi_Y} &&~~&\\
& & \mathrm{pt} &&&  ~~& & & \mathrm{B}G &&&~~&\\
\end{diagram}
and consequently applying cohomology functor
\begin{diagram}
H^{*}(\mathrm{E}G\times _{G}X,R) & &\lTo{f^{*}}& & H^{*}(\mathrm{E}G\times _{G}Y,R) &\\
& \luTo{\pi_X^{*}}&&\ruTo{\pi_Y^{*}} & \\
& & H^{*}(\mathrm{B}G,R) &  \\
\end{diagram}
\noindent $\pi _{X}=\pi _{Y}\circ\hat{f} $ and $\pi _{X}^{\ast }=
f^{\ast }\circ\pi _{Y}^{\ast }$. Thus $\ker \pi _{X}^{\ast }\supseteq \ker \pi _{Y}^{\ast } $.
\end{proof}

\begin{example}
\label{IndexOfSphere} $S^{n}$ is a $\mathbb{Z}_{2}$-space with the antipodal
action. The action is free and therefore
\begin{equation*}
\mathrm{E}\mathbb{Z}_{2}\times _{\mathbb{Z}_{2}}S^{n}\simeq S^{n}/\mathbb{Z}%
_{2}\approx\mathbb{R}\mathrm{P}^{n}~\Rightarrow ~H_{\mathbb{Z}_{2}}^{\ast }(S^{n},R)\cong H^{\ast
}(\mathbb{R}\mathrm{P}^{n},R).
\end{equation*}

\begin{compactenum}
\item $R=\mathbb{F}_{2}$: The cohomology ring $H^{\ast }(\mathrm{B}\mathbb{Z}%
_{2},\mathbb{F}_{2})=H^{\ast }(\mathbb{R}\mathrm{P}^{\infty },\mathbb{F}_{2})$ is the polynomial ring $\mathbb{F}%
_{2}[t]$ where $\deg (t)=1$. The $\mathbb{Z}_{2}$-index of $S^{n}$ is the
principal ideal generated by $t^{n+1}$:%
\begin{equation*}
\mathrm{Index}_{\mathbb{Z}_{2},\mathbb{F}_{2}}S^{n}=\mathrm{Index}_{\mathbb{Z%
}_{2},\mathbb{F}_{2}}^{n+2}S^{n}=\langle t^{n+1}\rangle \subseteq \mathbb{F}%
_{2}[t]\text{.}
\end{equation*}

\item $R=\mathbb{\mathbb{Z}}$: The cohomology ring $H^{\ast }(\mathrm{B}\mathbb{Z}_{2},\mathbb{
\mathbb{Z}})=H^{\ast }(\mathbb{R}\mathrm{P}^{\infty },\mathbb{\mathbb{Z}})$ is the quotient polynomial ring
$\mathbb{\mathbb{Z}}[\tau ]/\langle 2\tau \rangle $ where $\deg (\tau )=2$. The $\mathbb{Z}_{2}$%
-index of $S^{n}$ is the principal ideal%
\begin{equation*}
\mathrm{Index}_{\mathbb{Z}_{2},\mathbb{\mathbb{Z}}}S^{n}=\mathrm{Index}_{\mathbb{Z}_{2},\mathbb{\mathbb{Z}}}^{n+2}S^{n}=\left\{
\begin{array}{lll}
\langle \tau ^{\frac{n+1}{2}}\rangle , &  & \text{for }n\text{ odd,} \\
\langle \tau ^{\frac{n+2}{2}}\rangle , &  & \text{for }n\text{ even.}%
\end{array}%
\right.
\end{equation*}
\end{compactenum}
\end{example}

\begin{example}
\label{Ex:IndexS0}Let $G$ be a finite group and $H$ a subgroup of index $2$.
Then $H\lhd G$ and $G/H\cong\mathbb{Z}_{2}$. Let $V$ be the $1$-dimensional real representation of $G$ defined for
$v\in V$ by%
\begin{equation*}
g\cdot v=\left\{
\begin{array}{rll}
v, &  & \text{ for }g\in H, \\
-v, &  & \text{ for }g\notin H.
\end{array}%
\right.
\end{equation*}%
There is a $G$-homeomorphism $S(V)\approx\mathbb{Z}_{2}$. Therefore by \cite[last equation on page 34]{Hsiang}:%
\begin{equation*}
\mathrm{E}G\times _{G}S(V)\approx \mathrm{E}G\times _{G}\left( G/H\right)
\approx \left( \mathrm{E}G\times _{G}G\right) /H\approx \mathrm{E}G/H\approx
\mathrm{B}H
\end{equation*}%
and%
\begin{equation}
\mathrm{Index}_{G,R}S(V)=\ker \left( \mathrm{res}_{H}^{G}:H^{\ast
}(G,R)\rightarrow H^{\ast }(H,R)\right) .  \label{eq:IndexOfS0}
\end{equation}
\end{example}

\subsection{The restriction map and the index}

Let $X$ be a $G$-space and $K \subseteq G$ a subgroup. Then there is a
commutative diagram of fibrations \cite[pages 179-180]{Dieck87}:
\begin{equation}
\begin{diagram}
\mathrm{E}G\times _{G}X & \lTo{f} & \mathrm{E}G\times_{K}X \\
\dTo{} & & \dTo{}\\
\mathrm{B}G=\mathrm{E}G/G  & \lTo{\mathrm{B}i} & \mathrm{E}G/K=\mathrm{B}K\\
\end{diagram}
\label{eq:restrictionfibration}
\end{equation}
induced by inclusion $i:K\subset G$. Here $\mathrm{E}G$ in the lower right
corner is understood as a $K$-space and consequently a model for $\mathrm{E}%
K $. The map $\mathrm{B}i$ is a map between classifying spaces induced by
inclusion $i$. Now with coefficients in the ring $R$ we define%
\begin{equation*}
\mathrm{res}_{K}^{G}:=H^{\ast }(f):H^{\ast }(\mathrm{E}G\times
_{G}X,R)\rightarrow H^{\ast }(\mathrm{E}G\times _{K}X,R).
\end{equation*}%
If $G$ is a finite group, then the induced map on the cohomology of the
classifying spaces
\begin{equation*}
\mathrm{res}_{K}^{G}=(\mathrm{B}i)^{\ast }:H^{\ast }(\mathrm{B}%
G,R)\rightarrow H^{\ast }(\mathrm{B}K,R)
\end{equation*}%
coincides with the restriction homomorphism between group cohomologies
\begin{equation*}
\mathrm{res}_{K}^{G}:H^{\ast }(G,R)\rightarrow H^{\ast }(K,R).
\end{equation*}

\begin{proposition}
\label{Prop:Res-1}Let $X$ be a $G$-space, and $K$ and $L$ subgroups of $G$.

\begin{compactenum}[\rm(A)]
\item The morphism of fibrations
(\ref{eq:restrictionfibration}) provides the
following commutative diagram in cohomology:
\begin{equation}
\begin{diagram}
H^{\ast }(\mathrm{E}G\times _{G}X,R )  & \rTo{\mathrm{res}_{K}^{G}} & H^{\ast }(\mathrm{E}G\times _{K}X,R ) \\
\uTo{\pi _{X}^{\ast }} & & \uTo{\pi _{X}^{\ast }}\\
H^{\ast }(\mathrm{B}G,R)  & \rTo{\mathrm{res}_{K}^{G}} & H^{\ast }(\mathrm{B}K,R)\\
\end{diagram}
\label{eq:restrictionCohomology}
\end{equation}

\item For every $x\in H^{\ast }(\mathrm{B}G,R)$ and $y\in H^{\ast }(\mathrm{E}%
G\times _{G}X,R )$,%
\begin{equation*}
\mathrm{res}_{K}^{G}(x\cdot y)=\mathrm{res}_{K}^{G}(x)\cdot \mathrm{res}%
_{K}^{G}(y).
\end{equation*}

\item $L\subset K\subset G~\Rightarrow ~\mathrm{res}_{L}^{G}=\mathrm{res}%
_{L}^{K}\circ \mathrm{res}_{K}^{G}.$

\item The map of fibrations {\rm (\ref{eq:restrictionfibration})}
induces a morphism of Serre spectral sequences
\begin{equation*}
\Gamma _{i}^{\ast ,\ast }:E_{i}^{\ast ,\ast }(\mathrm{E}G\times
_{G}X,R)\rightarrow E_{i}^{\ast ,\ast }(\mathrm{E}K\times _{K}X,R)
\end{equation*}%
such that
\begin{compactenum}[\rm(1)]
\item $\Gamma _{\infty }^{\ast ,\ast }=\mathrm{res}_{K}^{G}:H^{\ast +\ast }(%
\mathrm{E}G\times _{G}X,R )\rightarrow H^{\ast +\ast }(\mathrm{E}G\times
_{K}X,R),$
\item $\Gamma _{2}^{\ast ,0}=\mathrm{res}_{K}^{G}:H^{\ast }(\mathrm{B}G,R
)\rightarrow H^{\ast }(\mathrm{B}K,R ).$
\end{compactenum}
\item Let $R$ and $S$ be commutative rings and $\phi:R \rightarrow
S$ a ring homomorphism. There are morphisms:
\begin{compactenum}[\rm(1)]
\item in equivariant cohomology $\Phi ^{\ast }:H^{\ast }(\mathrm{E}G\times _{G}X,R)\rightarrow H^{\ast }(\mathrm{E}%
G\times _{G}X,S)$,
\item in group cohomology $\Phi ^{\ast }:H^{\ast }(G,R)\rightarrow H^{\ast
}(G,S)$, and
\item between Serre spectral sequences $\Phi _{i}^{\ast ,\ast }:E_{i}^{\ast ,\ast }(\mathrm{E}G\times
_{G}X,R)\rightarrow E_{i}^{\ast ,\ast }(\mathrm{E}G\times _{G}X,S)$,
\end{compactenum}
induced by $\phi$ such that the following diagram commutes:
\begin{equation} \label{diagram:cube}
\begin{diagram}[size=1.5em,textflow]
H^{*}(\mathrm{E}G\times_G X,R) & & \rTo^{} & & H^{*}(\mathrm{E}G\times_K X,R) & & \\
& \rdTo_{\Phi} & & & \uTo & \rdTo_{\Phi} & \\
\uTo & & H^{*}(\mathrm{E}G\times_G X,S) & \rTo & \HonV & & H^{*}(\mathrm{E}G\times_K X,S) \\
& & \uTo & & \vLine & & \\
H^{*}(\mathrm{B}G,R) & \hLine & \VonH & \rTo & H^{*}(\mathrm{B}K,R) & & \uTo \\
& \rdTo_{\Phi} & & & & \rdTo_{\Phi} & \\
& & H^{*}(\mathrm{B}G,S) & & \rTo & & H^{*}(\mathrm{B}K,S) \\
\end{diagram}
\end{equation}
\end{compactenum}
\end{proposition}

\begin{remark*}
By a morphism of spectral sequences in properties (D) and (E) we mean that%
\begin{equation*}
{\Gamma _{i}^{\ast ,\ast }\circ \partial _{i}=\partial _{i}\circ \Gamma
_{i}^{\ast ,\ast }\qquad }\text{{and}}{\qquad {\Phi _{i}^{\ast ,\ast }\circ
\partial _{i}=\partial _{i}\circ \Phi _{i}^{\ast ,\ast }}}.
\end{equation*}%
These relations are applied in the situations where the right hand side is $%
\neq 0$ for a particular element $x$, to imply that the left hand side $%
\Gamma _{i}^{\ast ,\ast }\circ \partial _{i}(x)$ or $\Phi _{i}^{\ast ,\ast
}\circ \partial _{i}(x)$ is also $\neq 0$.\textrm{\ }In particular, then $%
\partial _{i}(x)\neq 0$.
\end{remark*}

\begin{figure}[htb]
\centering
\includegraphics[scale=0.70]{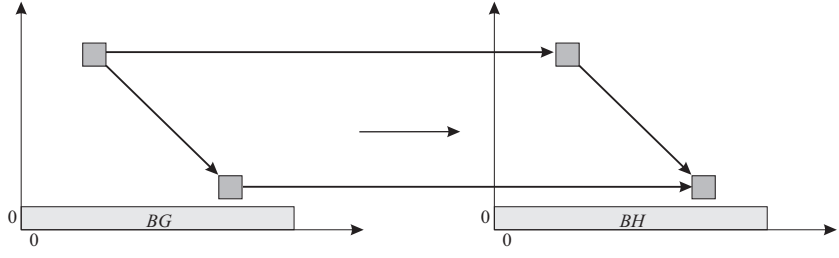}
\caption{{Illustration of Proposition \protect\ref{Prop:Res-1} (D) and (E)}}
\label{fig:mala1}
\end{figure}

\begin{proposition}
\label{prop:Res-Index}Let $X$ be a $G$-space and $K$ a subgroup of $G$. Let $%
R$ and $S$ be rings and $\phi :R\rightarrow S$ a ring homomorphism. Then

\begin{compactenum}[\rm(1)]
\item $\mathrm{res}_{K}^{G}\left( \mathrm{Index}_{G,R}X\right) \subseteq
\mathrm{Index}_{K,R}X$,

\item $\mathrm{res}_{K}^{G}\left( \mathrm{Index}_{G,R}^{r}X\right) \subseteq
\mathrm{Index}_{K,R}^{r}X$ for every $r\in\mathbb{N}$,

\item $\Phi ^{\ast }(\mathrm{Index}_{G,R}X)\subseteq
\mathrm{Index}_{G,S}X,$

\item $\Phi ^{\ast }(\mathrm{Index}_{G,R}^{r}X)\subseteq \mathrm{Index}%
_{G,S}^{r}X.$
\end{compactenum}
\end{proposition}

\begin{proof}
The assertions about the $\mathrm{Index}_{G,R}$ follow from diagrams (\ref%
{eq:restrictionCohomology}) and (\ref{diagram:cube}). The commutative
diagrams%
\[
\begin{diagram}[size=2.5em,textflow]
E_{r}^{\ast ,0}(\mathrm{E}G\times _{G}X,R) & \rTo{\Gamma _{r}^{\ast ,0}} & E_{r}^{\ast ,0}(\mathrm{E}K\times _{K}X,R) \\
\uTo & & \uTo{}\\
H^{\ast }(\mathrm{B}G,R) & \rTo{\mathrm{res}_{K}^{G}} & H^{\ast }(\mathrm{B}K,R)\\
\end{diagram}
\]
and
\[
\begin{diagram}[size=2.5em,textflow]
E_{r}^{\ast ,0}(\mathrm{E}G\times _{G}X,R) & \rTo{\Phi _{r}^{\ast ,0}} & E_{r}^{\ast ,0}(\mathrm{E}G\times _{G}X,S) \\
\uTo & & \uTo{}\\
H^{\ast }(\mathrm{B}G,R) & \rTo{\Phi ^{\ast }} & H^{\ast }(\mathrm{B}G,S)\\
\end{diagram}
\]
imply the partial index assertions.
\end{proof}

\subsection{Basic calculations of the index}

\subsubsection{The index of a product}

Let $X$ be a $G$-space and $Y$ an $H$-space. Then $X\times Y$ has the
natural structure of a $G\times H$-space. What is the relation between the three
indexes $\mathrm{Index}_{G\times H}(X\times Y)$, $\mathrm{Index}_{G}(X) $,
and $\mathrm{Index}_{H}(Y)$? Using the K\"{u}nneth formula one can prove the
following proposition \cite[Corollary 3.2]{Fadell--Husseini}, \cite[%
Proposition 2.7]{guide2} when the coefficient ring is a field.

\begin{proposition}
\label{Prop:Prod}Let $X$ be a $G$-space and $Y$ an $H$-space and
\begin{equation*}
H^{\ast }(\mathrm{B}G,\Bbbk )\cong \Bbbk \lbrack x_{1},\ldots,x_{n}],\text{\qquad \qquad
}H^{\ast }(\mathrm{B}H,\Bbbk )\cong \Bbbk \lbrack y_{1},\ldots,y_{m}]
\end{equation*}%
the cohomology rings of the associated classifying spaces with coefficients
in the field $\Bbbk $. If
\begin{equation*}
\mathrm{Index}_{G,\Bbbk }X=\langle f_{1},\ldots,f_{i}\rangle \text{\qquad and
\qquad }\mathrm{Index}_{H,\Bbbk }(Y)=\langle g_{1},\ldots,g_{j}\rangle \text{,}
\end{equation*}%
then
\begin{equation*}
\mathrm{Index}_{G\times H,\Bbbk }X=\langle
f_{1},\ldots,f_{i},g_{1},\ldots,g_{j}\rangle \subseteq \Bbbk \lbrack
x_{1},\ldots,x_{n},y_{1},\ldots,y_{m}]\text{.}
\end{equation*}
\end{proposition}

The $(\mathbb{Z}_2)^k$-index of a product of spheres can be computed using this
proposition and Example \ref{IndexOfSphere}.

\begin{corollary}
\label{Cor:IndexOfTorus}Let $S^{n_{1}}\times\cdots\times S^{n_{k}}$ be a $(%
\mathbb{Z}_{2})^{k}$-space with the product action. Then%
\begin{equation*}
\mathrm{Index}_{(\mathbb{Z}_{2})^{k},\mathbb{F}_{2}}S^{n_{1}}\times
\cdots\times S^{n_{k}}=\langle t_{1}^{n_{1}+1},\ldots,t_{k}^{n_{k}+1}\rangle
\subseteq \mathbb{F}_{2}[t_{1},\ldots,t_{k}].
\end{equation*}
\end{corollary}

Unfortunately when the coefficient ring is not a field the claim of
Proposition \ref{Prop:Prod} does not hold.

\begin{example}
\label{Ex:SxSwithZ}Let $S^{n}\times S^{n}$ be a $(\mathbb{Z}_{2})^{2}$-space
with the product action. From the previous corollary
\begin{equation}
\mathrm{Index}_{(\mathbb{Z}_{2})^{2},\mathbb{F}_{2}}S^{n}\times
S^{n}=\langle t_{1}^{n+1},t_{2}^{n+1}\rangle \subseteq \mathbb{F}%
_{2}[t_{1},t_{2}]=H^{\ast }((\mathbb{Z}_{2})^{2},\mathbb{F}_{2}).
\label{eq:pom-1}
\end{equation}%
To determine the $\mathbb{Z}$-index we proceed in two steps.

\medskip

\emph{Cohomology ring }$H^{\ast }((\mathbb{Z}_{2})^{2},\mathbb{Z})$%
: Following \cite[Section 4.1, page 508]{Lewis} the short exact sequence of
coefficients%
\begin{equation}
0\rightarrow\mathbb{Z}\overset{\times 2}{\rightarrow }\mathbb{Z}\overset{s}{\rightarrow }\mathbb{F}_{2}\rightarrow 0
\label{eq:coeffseq-1}
\end{equation}%
induces a long exact sequence in group cohomology \cite[Proposition 6.1,
page 71]{Brown} which in this case reduces to a sequence of short exact
sequences for $k>0$,%
\begin{equation}
0\rightarrow H^{k}((\mathbb{Z}_{2})^{2},\mathbb{Z})\overset{s_{\ast }}{%
\rightarrow }H^{k}((\mathbb{Z}_{2})^{2},\mathbb{F}_{2})\rightarrow H^{k+1}((%
\mathbb{Z}_{2})^{2},\mathbb{Z})\rightarrow 0.  \label{eq:Bocstain-1}
\end{equation}%
Therefore, as in \cite[Proposition 4.1, page 508]{Lewis},%
\begin{equation}
H^{\ast }((\mathbb{Z}_{2})^{2},\mathbb{Z})\cong\left( \mathbb{Z[\tau }_{1},\tau
_{2}\mathbb{]\otimes Z}[\mu ]\right)/\mathcal{I}  \label{eq:CohomologyZ_2^2withZ}
\end{equation}%
where $\deg \mathbb{\tau }_{1}=\deg \tau _{2}=2$, $\deg \mathbb{\mu }=3$ and the ideal $\mathcal{I}$ is
generated by the relations
\begin{equation*}
2\tau _{1}=2\tau _{2}=2\mu =0\text{\qquad and\qquad }\mu ^{2}=\tau _{1}\tau
_{2}(\tau _{1}+\tau _{2})\text{.}
\end{equation*}%
The ring morphism $c:\mathbb{Z}\rightarrow \mathbb{F}_{2}$ in the coefficient exact sequence (\ref%
{eq:coeffseq-1}) induces a morphism in group cohomology $c_{\ast }:H^{\ast
}((\mathbb{Z}_{2})^{2},\mathbb{Z})\rightarrow H^{\ast }((\mathbb{Z}_{2})^{2},%
\mathbb{F}_{2})$ given by:%
\begin{equation}
\mathbb{\tau }_{1}\longmapsto t_{1}^{2},~\mathbb{\tau }_{2}\longmapsto
t_{2}^{2},~\mathbb{\mu }\longmapsto t_{1}t_{2}(t_{1}+t_{2})\text{.}
\label{eq:correspodence-1}
\end{equation}%
The arguments used in the computation of the cohomology with integer
coefficients come from the Bockstein spectral sequence \cite{Browder}, \cite[%
pages 104-110]{Carlson} associated with the exact couple
\begin{equation*}
\begin{diagram}[size=2em,textflow]
H^{\ast }((\mathbb{Z}_{2})^{2},\mathbb{Z}) & &\lTo{p}& & H^{\ast }((\mathbb{Z}_{2})^{2},\mathbb{Z}) &\\
& \luTo^{\delta}&&\ldTo{q} & \\
& & H^{\ast }((\mathbb{Z}_{2})^{2},\mathbb{F}_{2}) &  \\
\end{diagram}
\end{equation*}
where $\deg (p)=\deg (q)=0$ and $\deg (\delta )=1$. The first differential $%
d_{1}=q\circ \delta $ coincides with the first Steenrod square $\mathrm{Sq}%
^{1}:H^{\ast }((\mathbb{Z}_{2})^{2},\mathbb{F}_{2})\rightarrow H^{\ast +1}((%
\mathbb{Z}_{2})^{2},\mathbb{F}_{2})$ and therefore is given by%
\begin{equation*}
1\mapsto 0,~t_{1}\mapsto t_{1}^{2},~t_{2}\mapsto t_{2}^{2}.
\end{equation*}%
Consequently, $t_{1}t_{2}\mapsto t_{1}^{2}t_{2}+t_{1}t_{2}^{2}$. The
spectral sequence stabilizes at the second step since the derived couple is
\begin{diagram}[size=2em,textflow]
0 & &\lTo& & 0 &\\
& \luTo^{\delta}&&\ldTo{q} & \\
& & \mathbb{F}_{2} &  \\
\end{diagram}
where $\mathbb{F}_{2}$ is in dimension $0$.

\medskip

$\mathrm{Index}_{(\mathbb{Z}_{2})^{2},\mathbb{Z}}S^{n}\times S^{n}$%
: The $(\mathbb{Z}_{2})^{2}$-action on $S^{n}\times S^{n}$, as a product of antipodal actions,
is free and therefore%
\begin{equation*}
\mathrm{E}(\mathbb{Z}_{2})^{2}\times _{(\mathbb{Z}_{2})^{2}}(S^{n}\times
S^{n})~\simeq ~\left( S^{n}\times S^{n}\right) /(\mathbb{Z}_{2})^{2}\approx
\mathbb{R}P^{n}\times\mathbb{R}P^{n}.
\end{equation*}%
Using equality (\ref{eq:pom-1}), Proposition \ref{Prop:Res-1}.E.3 on the
coefficient morphism $c:\mathbb{Z}\rightarrow \mathbb{F}_{2}$, the isomorphism%
\begin{equation*}
H_{(\mathbb{Z}_{2})^{2}}^{\ast }(S^{n}\times S^{n},\mathbb{Z})\cong H^{\ast }
(\mathbb{R}P^{n}\times\mathbb{R}P^{n},\mathbb{Z})
\end{equation*}%
and the existence of the $(\mathbb{Z}_{2})^{2}$-inclusions
\begin{equation*}
S^{n-1}\times S^{n-1}\subset S^{n}\times S^{n}\subset S^{n+1}\times S^{n+1},
\end{equation*}%
it can be concluded that%
\begin{equation}
\mathrm{Index}_{(\mathbb{Z}_{2})^{2},\mathbb{Z}}S^{n}\times S^{n}=\left\{
\begin{array}{lll}
\langle \tau _{1}^{\frac{n+1}{2}},\tau _{2}^{\frac{n+1}{2}}\rangle , &  &
\text{{\ for }}n\text{{\ \ odd}} \\
\langle \tau _{1}^{\frac{n+2}{2}},\tau _{2}^{\frac{n+2}{2}},\tau _{1}^{\frac{%
n}{2}}\mu ,\tau _{2}^{\frac{n}{2}}\mu \rangle , &  & \text{{\ for }}n\text{{%
\ \ even}}%
\end{array}%
\right. \subseteq H^{\ast }((\mathbb{Z}_{2})^{2},\mathbb{Z})\text{.}
\label{eq:Index-pom}
\end{equation}
\end{example}

\subsubsection{\label{Sec:IndexOfSphere}The index of a sphere}

We need to know how to compute the index of a sphere admitting an action of a finite group
different from the antipodal $\mathbb{Z}_{2}$-action. The following three
propositions will be of some help \cite[Proposition 3.13]{Fadell--Husseini},
\cite[Proposition 2.9]{guide2}.

\begin{proposition}
\label{Prop:IndexCherClasses}Let $G$ be a finite group and $V$ an $n$%
-dimensional complex representation of $G$. Then
\begin{equation*}
\mathrm{Index}_{G,\mathbb{Z}}S(V)=\langle c_{n}(V_G)\rangle \subset H^{\ast }(G,
\mathbb{Z})
\end{equation*}%
where $c_{n}(V_G)$ is the $n$-th Chern class of the bundle $%
V\rightarrow \mathrm{E}G\times _{G}V\rightarrow \mathrm{B}G$.
\end{proposition}

\begin{proof}
If the group $G$ acts on $H^{\ast }(S(V), \mathbb{Z})$ trivially, then from
the Serre spectral sequence of the sphere bundle
\begin{equation*}
S(V)\rightarrow \mathrm{E}G\times _{G}S(V)\rightarrow \mathrm{B}G
\end{equation*}
it follows that
\begin{equation*}
\mathrm{Index}_{G,\mathbb{Z}}S(V)=\langle e(V_{G})\rangle \subset
H^{\ast }(G, \mathbb{Z} ),
\end{equation*}
where $e(V)$ is the Euler class of the bundle $V\rightarrow \mathrm{%
E}G\times _{G}V\rightarrow \mathrm{B}G$. Now $V$ is a complex $G$%
-representation, therefore the group $G$ acts trivially on $H^{\ast }(S(V),%
\mathbb{Z})$. From \cite[Exercise 3, page 261]{Huss} it follows that%
\begin{equation*}
e(V_{G})=c_{n}(V_{G})
\end{equation*}
and the statement is proved.
\end{proof}

\begin{proposition}
\label{Prop:IndexOf Join}Let $U$, $V$ be two $G$-representations and let $%
S(U)$, $S(V)$ be the associated $G$-spheres. Let $R$ be a ring and assume
that $H^{\ast }(S(U),R)$, $H^{\ast }(S(V),R)$ are trivial $G$-modules. If $%
\mathrm{Index}_{G,R}(S(U))=\langle f\rangle \subseteq H^{\ast }(\mathrm{B}%
G,R)$ and\/ $\mathrm{Index}_{G,R}(S(V))=\langle g\rangle \subseteq H^{\ast }(%
\mathrm{B}G,R)$, then
\begin{equation*}
\mathrm{Index}_{G,R}S(U\oplus V)=\langle f\cdot g\rangle \subseteq H^{\ast }(%
\mathrm{B}G,R)\text{.}
\end{equation*}
\end{proposition}

\begin{proposition}
\label{Prop:IndexOfSphere}
\begin{compactenum}[\rm(A)]
\item Let $V$ be the $1$-dimensional $(\mathbb{Z}_{2})^{k}$%
-representation with the associated $\pm 1$ vector $(\alpha
_{1},\ldots,\alpha
_{k})\in (\mathbb{Z}_{2})^{k}$ \textrm{(as defined in Section \ref{sec:CS/TM}%
)}. Then
\begin{equation*}
\mathrm{Index}_{(\mathbb{Z}_{2})^{k},\mathbb{F}_{2}}S(V)=\langle \bar{\alpha}_{1}t_{1}+\cdots+%
\bar{\alpha}_{k}t_{k}\rangle \subseteq
\mathbb{F}_{2}[t_{1},\ldots,t_{k}],
\end{equation*}%
where $\bar{\alpha}_{i}=0$ if $\alpha _{i}=1$, and $\bar{\alpha}_{i}=1$ if $%
\alpha _{i}=-1$.

\item Let $U$ be an $n$-dimensional $(\mathbb{Z}_{2})^{k}$%
-representation with a decomposition $U\cong V_{1}\oplus\cdots\oplus V_{n}$
into $1$-dimensional $(\mathbb{Z}_{2})^{k}$-representations $V_{1},\ldots,V_{n}$%
. If $(\alpha _{1i},\ldots,\alpha _{ki})\in (\mathbb{Z}_{2})^{k}$ is the
associated $\pm 1$ vector of~$V_{i}$, then%
\begin{equation*}
\mathrm{Index}_{(\mathbb{Z}_{2})^{k},\mathbb{F}_{2}}S(U)=\Big\langle \prod\limits_{i=1}^{n}%
\left( \bar{\alpha}_{1i}t_{1}+\cdots+\bar{\alpha}_{ki}t_{k}\right) \Big\rangle %
\subseteq \mathbb{F}_{2}[t_{1},\ldots,t_{k}].
\end{equation*}
\end{compactenum}
\end{proposition}

\begin{example}
Let $V_{-+}$, $V_{+-}$ and $V_{--}$ be $1$-dimensional real $(\mathbb{Z}%
_{2})^{2}$-representations introduced in Section \ref{Sec:TestMap}. Then by
Proposition \ref{Prop:IndexOfSphere}%
\begin{equation*}
\mathrm{Index}_{(\mathbb{Z}_{2})^{2},\mathbb{F}_{2}}S(V_{-+})=\langle
t_{1}\rangle \text{, \quad\ }\mathrm{Index}_{(\mathbb{Z}_{2})^{2},\mathbb{F}%
_{2}}S(V_{+-})=\langle t_{2}\rangle \text{, }\quad \mathrm{Index}_{(\mathbb{Z%
}_{2})^{2},\mathbb{F}_{2}}S(V_{--})=\langle t_{1}+t_{2}\rangle .
\end{equation*}%
On the other hand, Example \ref{Ex:IndexS0} and the restriction diagram (\ref%
{eq:DiagramZ2xZ2withZ}) imply that
\begin{equation*}
\mathrm{Index}_{(\mathbb{Z}_{2})^{2},\mathbb{Z}}S(V_{-+})=\langle \tau
_{1},\mu \rangle \text{, }\quad \mathrm{Index}_{(\mathbb{Z}_{2})^{2},\mathbb{%
Z}}S(V_{+-})=\langle \tau _{2},\mu \rangle \text{, }\quad \mathrm{Index}_{(%
\mathbb{Z}_{2})^{2},\mathbb{Z}}S(V_{--})=\langle \tau _{1}+\tau _{2},\mu
\rangle .
\end{equation*}
\end{example}

\section{The cohomology of $D_{8}$ and the restriction diagram}

The dihedral group $W_{2}=D_{8}=(\mathbb{Z}_{2})^{2}\rtimes\mathbb{Z}_{2}=\left( \langle \varepsilon _{1}\rangle \times \langle \varepsilon
_{2}\rangle \right) \rtimes \langle \sigma \rangle $ can be presented by%
\begin{equation*}
D_{8}=\langle \varepsilon _{1},\sigma ~|~\varepsilon _{1}^{2}=\sigma
^{2}=\left( \varepsilon _{1}\sigma \right) ^{4}=1\rangle .
\end{equation*}%
Then $\langle \varepsilon _{1}\sigma \rangle \cong
\mathbb{Z}_{4}$ and $\varepsilon _{2}=\sigma \varepsilon _{1}\sigma $.

\subsection{The poset of subgroups of $D_{8}$}

The poset $\mathrm{Sub}(G)$ denotes the collection of all nontrivial
subgroups of a given group $G$ ordered by inclusion. The poset $\mathrm{Sub}%
(G)$ can be interpreted as a small category $\mathfrak{G}$ in the usual way:

\begin{compactitem}
\item $\mathrm{Ob}(\mathfrak{G})=\mathrm{Sub}(G)$,

\item for every two objects $H$ and $K$, subgroups of $G$, there is a unique
morphism $f_{H,K}:H\rightarrow K$ if $H\supseteq K$, and no morphism if $%
H\nsupseteq K$, i.e.%
\begin{equation*}
\mathrm{Mor}(H,K)=\left\{
\begin{array}{rll}
{\ \{f}_{H,K}{\ \}}&, & {H\supseteq K,} \\
{\ \emptyset }&, & {H\nsupseteq K.}%
\end{array}%
\right.
\end{equation*}
\end{compactitem}
The Hasse diagram of the poset $\mathrm{Sub}(D_{8})$ is presented in the
following diagram.

\begin{diagram}
&        &        &        &  D_8   &        &        &        &\\
&        &        & \ldTo  & \dTo       &\rdTo   &        &        &\\
&        &
\begin{tabular}{|c|}                                                   
\hline
$H_{1}$ \\ \hline
$\langle \varepsilon _{1},\varepsilon _{2}\rangle $ \\ \hline
$\mathbb{Z}_{2}\times\mathbb{Z}_{2}$ \\ \hline
\end{tabular}
&        &
\begin{tabular}{|c|}                                                   
\hline
$H_{2}$ \\ \hline
$\langle \varepsilon _{1}\sigma \rangle $ \\ \hline
$\mathbb{Z}_{4}$ \\ \hline
\end{tabular}
&        &
\begin{tabular}{|c|}                                                   
\hline
$H_{3}$ \\ \hline
$\langle \varepsilon _{1}\varepsilon _{2},\sigma \rangle $ \\ \hline
$\mathbb{Z}_{2}\times\mathbb{Z}_{2}$ \\ \hline
\end{tabular}
&        &\\
& \ldTo  & \dTo   & \rdTo  & \dTo   & \ldTo  & \dTo   & \rdTo  &\\
\begin{tabular}{|c|}
\hline
$K_{1}$ \\ \hline
$\langle \varepsilon _{1}\rangle $ \\ \hline
$\mathbb{Z}_{2}$ \\ \hline
\end{tabular}
&        &
\begin{tabular}{|c|}
\hline
$K_{2}$ \\ \hline
$\langle \varepsilon _{2}\rangle $ \\ \hline
$\mathbb{Z}_{2}$ \\ \hline
\end{tabular}
&        &
\begin{tabular}{|c|}
\hline
$K_{3}$ \\ \hline
$\langle \varepsilon _{1}\varepsilon _{2}\rangle $ \\ \hline
$\mathbb{Z}_{2}$ \\ \hline
\end{tabular}
&        &
\begin{tabular}{|c|}
\hline
$K_{4}$ \\ \hline
$\langle \sigma \rangle $ \\ \hline
$\mathbb{Z}_{2}$ \\ \hline
\end{tabular}
&        &
\begin{tabular}{|c|}
\hline
$K_{5}$ \\ \hline
$\langle \varepsilon _{1}\varepsilon _{2}\sigma \rangle $ \\ \hline
$\mathbb{Z}_{2}$ \\ \hline
\end{tabular}\\
\end{diagram}

\subsection{\label{Sec:H*(D8,F_2)}The cohomology ring $H^{\ast }(D_{8},%
\mathbb{F}_{2})$}

The dihedral group $D_{8}$ is an example of a wreath product. Therefore the
associated classifying space can, as in \cite[page 117]{Adem-Milgram}, be
written explicitly as
\begin{equation*}
\mathrm{B}D_{8}=\mathrm{B}(\mathbb{Z}_{2})^{2}\times _{\mathbb{Z}_{2}}%
\mathrm{E}\mathbb{Z}_{2}\approx (\mathrm{B}(\mathbb{Z}_{2})^{2})\times _{%
\mathbb{Z}_{2}}\mathrm{E}\mathbb{Z}_{2},
\end{equation*}%
where $\mathbb{Z}_{2}=\langle \sigma \rangle $ acts on $(\mathrm{B}\mathbb{Z}%
_{2})^{2}$ by interchanging coordinates. Presented in this way $\mathrm{B}%
D_{8}$ is the Borel construction of the $\mathbb{Z}_{2}$-space $(\mathrm{B}%
\mathbb{Z}_{2})^{2}$. Thus $\mathrm{B}D_{8}$ fits into a fibration%
\begin{equation}
\mathrm{B}(\mathbb{Z}_{2})^{2}\rightarrow (\mathrm{B}(\mathbb{Z}_{2})^{2})\times _{
\mathbb{Z}_{2}}\mathrm{E}\mathbb{Z}_{2}\rightarrow \mathrm{B}\mathbb{Z}_{2}.  \label{eq:fibration-1}
\end{equation}%
There is an associated Serre spectral sequence with $E_{2}$-term%
\begin{equation}
E_{2}^{p,q}=\left\{
\begin{array}{l}
H^{p}(\mathrm{B}\mathbb{Z}_{2},\mathcal{H}^{q}\left( \mathrm{B}(\mathbb{Z}_{2})^{2},\mathbb{F}_{2}\right) ) \\
H^{p}(\mathbb{Z}_{2},H^{q}\left( (\mathbb{Z}_{2})^{2},\mathbb{F}_{2}\right) )%
\end{array}%
\right. ~\Longrightarrow ~\left\{
\begin{array}{l}
H^{p+q}(\mathrm{B}D_{8},\mathbb{F}_{2}) \\
H^{p+q}(D_{8},\mathbb{F}_{2})%
\end{array}%
\right.  \label{eq:E2}
\end{equation}%
which converges to the cohomology of the group $D_{8}$ with $\mathbb{F}_{2}$%
-coefficients. This spectral sequence is also the Lyndon-Hochschild-Serre
(LHS) spectral sequence \cite[Section IV.1, page 116]{Adem-Milgram}
associated with the group extension sequence:%
\begin{equation*}
1\rightarrow (\mathbb{Z}_{2})^{2}\rightarrow D_{8}\rightarrow D_{8}/(\mathbb{%
Z}_{2})^{2}\rightarrow 1.
\end{equation*}

In \cite[Theorem 1.7, page 117]{Adem-Milgram} it is proved that the spectral
sequence (\ref{eq:E2}) collapses at the $E_{2}$-term. Therefore, to compute
the cohomology of~$D_{8}$ we only need to read the $E_{2}$-term.

\begin{lemma}
\label{Lemma:CohomologyD8}\qquad

\begin{compactenum}[\rm(i)]

\item $H^{\ast }\left( (\mathbb{Z}_{2})^{2},\mathbb{F}_{2}\right) \cong _{\mathrm{ring}}\mathbb{F}_{2}[a,a+b]$%
, where $\deg (a)=\deg (a+b)=1$ and the $\mathbb{Z}_{2}$-action induced by $\sigma $ is given by $\sigma \cdot a=a+b$.

\item $H^{\ast }\left( (
\mathbb{Z}_{2})^{2},\mathbb{F}_{2}\right) ^{\mathbb{Z}_{2}}\cong _{\mathrm{ring}}\mathbb{F}_{2}[b,a(a+b)]$.

\item $H^{i}\left( (\mathbb{Z}_{2})^{2},\mathbb{F}_{2}\right) \cong _{\mathbb{Z}_{2}\text{-}\mathrm{module}}\mathbb{F}_{2}[
\mathbb{Z}_{2}]^{s_{i,1}}\oplus \mathbb{F}_{2}^{s_{i,2}}$, where $s_{i,1}\geq 0$, $s_{i,2}\geq 0$ and $\mathbb{F}_{2}[
\mathbb{Z}_{2}]$ denotes a free $\mathbb{Z}_{2}$-module and $\mathbb{F}_{2}$ a trivial one.

\item $E_{2}^{\ast ,i}=H^{\ast }(\mathbb{Z}_{2},H^{i}\left( (\mathbb{Z}_{2})^{2},\mathbb{F}_{2}\right) )\cong _{\mathrm{ring}}H^{\ast }(
\mathbb{Z}_{2},\mathbb{F}_{2})^{\oplus s_{i,2}}\oplus
\mathbb{F}_{2}^{s_{i,1}}$, where $\mathbb{F}_{2}^{s_{i,1}}$
denotes a ring concentrated in dimension $0$.

\end{compactenum}
\end{lemma}

\begin{proof}
(i) The statement follows from the observation that $\mathrm{B}(\mathbb{Z}_{2})^{2}\approx \left( \mathrm{B}(
\mathbb{Z}_{2})\right) ^{2}$, and consequently
\begin{equation*}
H^{\ast }\left( (\mathbb{Z}_{2})^{2},\mathbb{F}_{2}\right) \cong _{\mathrm{ring}}H^{\ast }\left(
\mathbb{Z}_{2},\mathbb{F}_{2}\right) \otimes H^{\ast }\left(\mathbb{Z}_{2},\mathbb{F}_{2}\right) \cong _{\mathrm{ring}}\mathbb{F}_{2}[a]\otimes
\mathbb{F}_{2}[a+b].
\end{equation*}
The $\mathbb{Z}_{2}$-action interchanges copies on the left hand side. Generators on the
right hand side are chosen such that the $\mathbb{Z}_{2}$-action coming from the isomorphism swaps $a$ and $a+b$.

(ii)\ With the induced $\mathbb{Z}_{2}$-action $b=a+(a+b)$ and $a(a+b)$ are invariant polynomials. They
generate the ring of all invariant polynomials.

(iii) The cohomology $H^{i}\left( (\mathbb{Z}_{2})^{2},\mathbb{F}_{2}\right) $ is a $\mathbb{Z}_{2}$-module and therefore a direct sum of irreducible $
\mathbb{Z}_{2}$-modules. There are only two irreducible $\mathbb{Z}_{2}$-modules over $\mathbb{F}_{2}$: the free one $\mathbb{F}_{2}[\mathbb{Z}_{2}]$ and the trivial one~$\mathbb{F}_{2}$.

(iv) The isomorphism follows from (iii) and the following two properties of
group cohomology \cite[Exercise 2.2, page 190]{Hilton} and \cite[Corollary
6.6, page 73]{Brown}. Let $M$ and $N$ be $G$-modules of a finite group $G$.
Then

\qquad (a) $H^{\ast }(G,M\oplus N)\cong H^{\ast }(G,M)\oplus H^{\ast }(G,N)$

\qquad (b) $M$ is a free $G$-module $~\Rightarrow ~H^{\ast
}(G,M)=H^{0}(G,M)\cong M^{G}$.

Applied in our case, this yields
\begin{eqnarray*}
E_{2}^{\ast ,i} &=_{\text{{ring}}}&H^{\ast }(\mathbb{Z}_{2},H^{i}\left( (%
\mathbb{Z}_{2})^{2},\mathbb{F}_{2}\right) ) \\
&\cong_{\text{{ring}}}& H^{\ast }(\mathbb{Z}_{2},\mathbb{F}_{2}[\mathbb{Z}%
_{2}]^{s_{i,1}}\oplus \mathbb{F}_{2}^{s_{i,2}}) \\
&\cong_{\text{{ring}}}&H^{\ast }(\mathbb{Z}_{2},\mathbb{F}_{2}[\mathbb{Z}%
_{2}])^{\oplus s_{i,1}}\oplus H^{\ast }( \mathbb{Z}_{2},\mathbb{F}%
_{2})^{\oplus s_{i,2}} \\
&\cong_{\text{{ring}}}&H^{0}(\mathbb{Z}_{2},\mathbb{F}_{2}[\mathbb{Z}%
_{2}])^{\oplus s_{i,1}}\oplus H^{\ast }( \mathbb{Z}_{2},\mathbb{F}%
_{2})^{\oplus s_{i,2}} \\
&\cong_{\text{{ring}}}&(\mathbb{F}_{2}[\mathbb{Z}_{2}]^{\mathbb{Z}%
_{2}})^{\oplus s_{i,1}}\oplus H^{\ast }( \mathbb{Z}_{2},\mathbb{F}%
_{2})^{\oplus s_{i,2}} \\
&\cong_{\text{{ring}}}&\mathbb{F}_{2}{}^{\oplus s_{i,1}}\oplus H^{\ast }(%
\mathbb{Z}_{2},\mathbb{F}_{2})^{\oplus s_{i,2}}
\end{eqnarray*}
\end{proof}


Let the cohomology of the base space of the fibration (\ref{eq:fibration-1})
be denoted by
\begin{equation*}
H^{\ast }(\mathbb{Z}_{2},\mathbb{F}_{2})=\mathbb{F}_{2}[x].
\end{equation*}%
The $E_{2}$-term (\ref{eq:E2}) can be pictured as in Figure \ref{Fig-1}.

\begin{figure}[htb]
\centering
\includegraphics[scale=0.50]{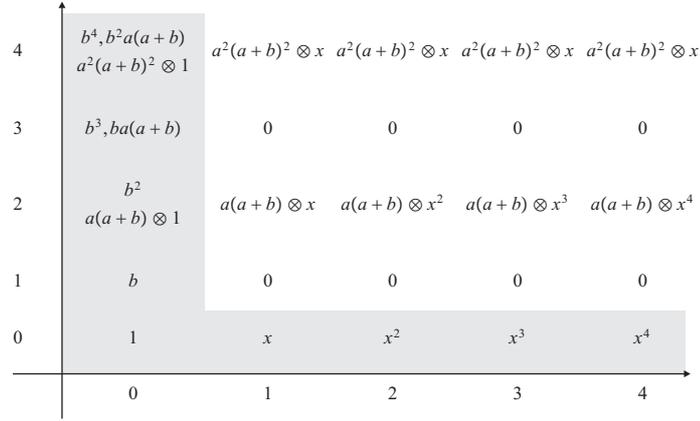}
\caption{{$E_{2}$-term}}
\label{Fig-1}
\end{figure}

The cohomology of $D_{8}$ can be read from the picture. If we denote
\begin{equation}
y:=b,\quad w:=a(a+b)  \label{eq:NotationOfElements}
\end{equation}%
and keep $x$ as we introduced above, then
\begin{equation*}
H^{\ast }(D_{8},\mathbb{F}_{2})=\mathbb{F}_{2}[x,y,w]/\langle xy\rangle .
\end{equation*}%
Also, the restriction homomorphism
\begin{equation}
\mathrm{res}_{H_{1}}^{D_{8}}:H^{\ast }(D_{8},\mathbb{F}_{2})=\mathbb{F}%
_{2}[x,y,w]/\langle xy\rangle \ \ \rightarrow \ \ H^{\ast }(H_{1},\mathbb{F}%
_{2})=\mathbb{F}_{2}[a,a+b]  \label{eq:RestrictionH1-1}
\end{equation}%
can be read off since it is induced by the inclusion of the fibre in the
fibration (\ref{eq:fibration-1}). On generators,
\begin{equation}
\mathrm{res}_{H_{1}}^{D_{8}}(x)=0,~\mathrm{res}_{H_{1}}^{D_{8}}(y)=b,~%
\mathrm{res}_{H_{1}}^{D_{8}}(w)=a(a+b).  \label{eq:RestrictionH1-2}
\end{equation}

\subsection{The cohomology diagram of subgroups with coefficients in $%
\mathbb{F}_{2}$}

Let $G$ be a finite group and $R$ an arbitrary ring. Then the diagram
\textrm{Res}$_{(R)}:\mathfrak{G\rightarrow Ring}$ (covariant functor)
defined by%
\begin{equation*}
\begin{array}{c}
\mathrm{Ob}(\mathfrak{G})\ni {\ H~\longmapsto ~H}^{\ast }(H,R) \\
\left( {\ H\supseteq K}\right) {\ ~\longmapsto ~}\left( \mathrm{res}_{K}^{H}:%
{\ H}^{\ast }(H,R)\rightarrow {\ H}^{\ast }(K,R)\right)%
\end{array}%
\end{equation*}%
is the \textit{cohomology diagram of subgroups} of $G$ with coefficients in
the ring $R$. In this section we assume that $R=\mathbb{F}_{2}$.

\subsubsection{The $\mathbb{Z}_{2}\times \mathbb{Z}_{2}$-diagram}

The cohomology of any elementary abelian $2$-group $\mathbb{Z}_{2}\times
\mathbb{Z}_{2}$ is a polynomial ring $\mathbb{F}_{2}{\ [x,y]}$, $\deg (x)=\deg (y)=1$.
The restrictions to the three subgroups of order $2$ are given by all
possible projections $\mathbb{F}_{2}{\ [x,y]\rightarrow }\mathbb{F}_{2}{\ [t]%
}$, $\deg (t)=1$:%
\begin{equation*}
\left( x\mapsto t,y\mapsto 0\right) \text{ \quad or }\quad \left( x\mapsto
0,y\mapsto t\right) \text{ \quad or }\quad \left( x\mapsto t,y\mapsto
t\right) .
\end{equation*}%
Thus the cohomology diagram of the subgroups of $\mathbb{Z}_{2}\times\mathbb{Z}_{2}$ is%

\begin{equation}
\begin{diagram}
&     & &
\begin{tabular}{|c|}
\hline
$\mathbb{Z}_{2}\times\mathbb{Z}_{2}$ \\ \hline
$\mathbb{F}_{2}{\ [x,y]}$ \\ \hline
\end{tabular}
&     &     &\\
& \ldTo{x\mapsto 0}_{y\mapsto t_1} &   & \dTo{x\mapsto t_2}_{y\mapsto 0} &   & \rdTo{x\mapsto t_3}_{y\mapsto t_3}     &     \\
\begin{tabular}{|c|}
\hline
$\mathbb{Z}_{2}$ \\ \hline
$\mathbb{F}_2 [t_{1}]$ \\ \hline
\end{tabular}
&     &    &
\begin{tabular}{|c|}
\hline
$\mathbb{Z}_{2}$ \\ \hline
$\mathbb{F}_2 [t_{2}]$ \\ \hline
\end{tabular}
&     & &
\begin{tabular}{|c|}
\hline
$\mathbb{Z}_{2}$ \\ \hline
$\mathbb{F}_2 [t_{3}]$ \\ \hline
\end{tabular}     \\
\end{diagram} \label{eq:DiagramZ2xZ2}
\end{equation}

\subsubsection{\label{Sec:RestrictionDiagram}The $D_{8}$-diagram}

For the dihedral group $D_{8}$, from \cite{WebD8} and (\ref%
{eq:RestrictionH1-1}), the two top levels of the diagram can be presented by:

\begin{equation}
\begin{diagram}[size=4em,textflow]
&  &  &\begin{tabular}{|c|}
\hline
${\ D}_{8}$ \\ \hline
$\mathbb{F}_{2}{\ [x,y;w]/\langle xy}{\ \rangle }$ \\ \hline
$\deg {\ :1,1,2}$ \\ \hline
\end{tabular}  &  &  & \\
&\ldTo{x\mapsto 0, y\mapsto b}_{w\mapsto a(a+b)}  &  &\dTo{x,y\mapsto e}_{w\mapsto u}  &  &\rdTo{x\mapsto d, y\mapsto 0}_{w\mapsto c(c+d)}  & \\
\begin{tabular}{|c|}
\hline
${\ H}_{1}$ \\ \hline
$\mathbb{F}_{2}{\ [a,b]}$ \\ \hline
$\deg {\ :1,1}$ \\ \hline
\end{tabular}&  &  &
\begin{tabular}{|c|}
\hline
${\ H}_{2}$ \\ \hline
$\mathbb{F}_{2}{\ [e,u]/\langle e}^{2}{\ \rangle }$ \\ \hline
$\deg {\ :1,2}$ \\ \hline
\end{tabular}  &  &  &
\begin{tabular}{|c|}
\hline
${\ H}_{3}$ \\ \hline
$\mathbb{F}_{2}{\ [c,d]}$ \\ \hline
$\deg {\ :1,1}$ \\ \hline
\end{tabular}\\
\end{diagram}\label{eq:diagramD8-1}
\end{equation}
Let $H^{\ast}(K_{i},\mathbb{F}_{2})=\mathbb{F}_{2}[t_i]$, $%
\deg (t_{i})=1$. From \cite[Corollary II.5.7, page 69]{Adem-Milgram} the
restriction
\begin{equation*}
\mathrm{res}_{K_{3}}^{H_{2}}:\left( H^{\ast }(H_{2},\mathbb{F}_{2})=\mathbb{F%
}_{2}[e,u]{\ /\langle }e^{2}\rangle \right) \longrightarrow \left( H^{\ast
}(K_{3},\mathbb{F}_{2})=\mathbb{F}_{2}[t_{3}]\right)
\end{equation*}%
is given by $e\mapsto 0$, $u\mapsto t_{3}^{2}$. Thus, the restriction $%
\mathrm{res}_{K_{3}}^{D_{8}}$ is given by $x\mapsto 0$, $y\mapsto 0$, $%
w\mapsto t_{3}^{2}$. Using diagrams (\ref{eq:DiagramZ2xZ2}), (\ref%
{eq:diagramD8-1}) with the property (C) from Proposition \ref{Prop:Res-1} we
almost completely reveal the cohomology diagram of subgroups of $D_{8}$. The
equalities%
\begin{equation*}
\mathrm{res}_{K_{3}}^{D_{8}}=\mathrm{res}_{K_{3}}^{H_{2}}\circ \mathrm{res}%
_{H_{2}}^{D_{8}}=\mathrm{res}_{K_{3}}^{H_{1}}\circ \mathrm{res}%
_{H_{1}}^{D_{8}}=\mathrm{res}_{K_{3}}^{H_{3}}\circ \mathrm{res}%
_{H_{3}}^{D_{8}}
\end{equation*}%
imply that

\begin{compactitem}
\item $\mathrm{res}_{K_{3}}^{H_{1}}:\left(H^{\ast}(H_{1},\mathbb{F}
_{2})=\mathbb{F}_{2}[a,b]\right) \longrightarrow \left(H^{\ast}
(K_{3},\mathbb{F}_{2})=\mathbb{F}_{2}[t_3]\right) $ is given by $%
a\mapsto t_{3},~b\mapsto 0$,

\item $\mathrm{res}_{K_{3}}^{H_{3}}:\left(H^{\ast }(H_{3},\mathbb{F}%
_{2})=\mathbb{F}_{2}[c,d]\right) \longrightarrow \left(H^{\ast}
(K_{3},\mathbb{F}_{2})=\mathbb{F}_{2}[t_3]\right) $ is given by $%
c\mapsto t_{3},~d\mapsto 0$.
\end{compactitem}

\begin{equation}
\begin{diagram}
\begin{tabular}{|c|}
\hline
$H_{1}$ \\ \hline
$\mathbb{F}_{2}{\ [a,b]}$ \\ \hline
$\deg : 1,1$ \\ \hline
\end{tabular}
&  & &
\begin{tabular}{|c|}
\hline
$H_{2}$ \\ \hline
$\mathbb{F}_{2}{\ [e,u]/\langle e}^{2}{\ \rangle }$ \\ \hline
$\deg :1,2$ \\ \hline
\end{tabular}  &  &  &
\begin{tabular}{|c|}
\hline
$H_{3}$ \\ \hline
$\mathbb{F}_{2}{\ [c,d]}$ \\ \hline
$\deg : 1,1$ \\ \hline
\end{tabular}
\\
&\rdTo_{a\mapsto t_3}^{b\mapsto 0}&  & \dTo_{e\mapsto 0}^{u\mapsto t_3^2} & & \ldTo_{c\mapsto t_3}^{d\mapsto 0} &  \\
&  &  &\begin{tabular}{|c|}
\hline
$K_{3}$ \\ \hline
$\mathbb{F}_{2}[t_3]$ \\ \hline
$\deg :1$ \\ \hline
\end{tabular}  &  & &\\
\end{diagram}\label{eq:diagramD8-2}
\end{equation}
The cohomology diagram (\ref{eq:DiagramZ2xZ2}) of subgroups of $\mathbb{Z}_{2}\times\mathbb{Z}_{2}$ and the part (\ref{eq:diagramD8-2}) of the $D_{8}$ diagram imply that
\begin{compactitem}
\item $\mathrm{res}_{K_{1}}^{H_{1}}:\mathbb{F}_{2}[a,b]\longrightarrow
\mathbb{F}_{2}[t_1]$ and $\mathrm{res}_{K_{2}}^{H_{1}}:\mathbb{F}%
_{2}[a,b]\longrightarrow \mathbb{F}_{2}[t_2]$ are given by
\begin{equation*}
\left( a\mapsto t_{1},b\mapsto t_{1}\text{ and }a\mapsto 0,b\mapsto
t_{2}\right) \text{ or }\left( a\mapsto 0,b\mapsto t_{1}\text{ and }a\mapsto
t_{2},b\mapsto t_{2}\right) ,
\end{equation*}

\item $\mathrm{res}_{K_{4}}^{H_{3}}:\mathbb{F}_{2}[c,d]\longrightarrow
\mathbb{F}_{2}[t_4]$ and $\mathrm{res}_{K_{5}}^{H_{3}}:\mathbb{F}%
_{2}[a,b]\longrightarrow \mathbb{F}_{2}[t_5]$ are given by
\begin{equation*}
\left( c\mapsto t_{4},d\mapsto t_{4}\text{ and }c\mapsto 0,d\mapsto
t_{5}\right) \text{ or }\left( c\mapsto 0,d\mapsto t_{4}\text{ and }c\mapsto
t_{5},d\mapsto t_{5}\right) .
\end{equation*}
\end{compactitem}

\begin{proposition}
For all $i\neq 3$, $\mathrm{res}_{K_i}^{D_{8}}(w)=0$, while $\mathrm{res}%
_{K_3}^{D_{8}}(w)\neq0$.
\end{proposition}

\begin{proof}
The result follows from the diagram (\ref{eq:diagramD8-1}) in the following way:
\begin{compactenum}[\rm(a)]
\item For $i\in\{1,2\}$:
\[
\mathrm{res}_{K_i}^{D_{8}}(w)=\mathrm{res}_{K_i}^{H_1}\circ\mathrm{res}_{H_1}^{D_{8}}(w)=\mathrm{res}_{K_i}^{H_1}(a(a+b))=0
\]
since either $a\mapsto t_i, b\mapsto t_i$ or $a\mapsto 0, b\mapsto t_i$.
\item For $i\in\{4,5\}$:
\[
\mathrm{res}_{K_i}^{D_{8}}(w)=\mathrm{res}_{K_i}^{H_3}\circ\mathrm{res}_{H_3}^{D_{8}}(w)=\mathrm{res}_{K_i}^{H_3}(c(c+d))=0
\]
since either $c\mapsto t_i, d\mapsto t_i$ or $c\mapsto 0, d\mapsto t_i$.
\end{compactenum}
\end{proof}

\begin{corollary}
The cohomology of the dihedral group $D_{8}$ is%
\begin{equation*}
H^{\ast }(D_{8},\mathbb{F}_{2})=\mathbb{F}_{2}[x,y,w]/\langle xy\rangle
\end{equation*}%
where
\begin{compactenum}[\rm(a)]
\item $x\in H^{1}(D_{8},\mathbb{F}_{2})$ and $\mathrm{res}%
_{H_{1}}^{D_{8}}(x)=0 $,

\item $y\in H^{1}(D_{8},\mathbb{F}_{2})$ and $\mathrm{res}%
_{H_{3}}^{D_{8}}(y)=0 $,

\item $w\in H^{1}(D_{8},\mathbb{F}_{2})$ and $\mathrm{res}%
_{K_{1}}^{D_{8}}(w)=\mathrm{res}_{K_{2}}^{D_{8}}(w)=\mathrm{res}%
_{K_{4}}^{D_{8}}(w)=\mathrm{res}_{K_{5}}^{D_{8}}(w)=0$ and $\mathrm{res}%
_{K_{3}}^{D_{8}}(w)\neq 0$.

\end{compactenum}
\end{corollary}

\noindent \textbf{Assumption }Without lose of generality we can assume that%
\begin{equation}
\begin{array}{cccc}
\mathrm{res}_{K_{1}}^{H_{1}}(a)=t_{1}, & \mathrm{res}%
_{K_{1}}^{H_{1}}(b)=t_{1}, & \mathrm{res}_{K_{2}}^{H_{1}}(a)=0, & \mathrm{res%
}_{K_{2}}^{H_{1}}(b)=t_{2}.%
\end{array}
\label{Assumption-1}
\end{equation}

\subsection{The cohomology ring $H^{\ast }(D_{8},\mathbb{Z})$}

In this section we present the cohomology $H^{\ast }(D_{8},\mathbb{Z})$
based on:

\begin{compactitem}
\item[A.] Evens' approach \cite[Section 5, pages 191-192]{Evans}, where the
concrete generators in $H^{\ast }(D_{8},\mathbb{Z})$ are identified using
the Chern classes of appropriate irreducible complex $D_{8}$%
-representations. We also consider LHS spectral sequences associated with
following two extensions%
\begin{equation}
1\rightarrow H_{1}\rightarrow D_{8}\rightarrow D_{8}/(\mathbb{Z}%
_{2})^{2}\rightarrow 1\text{\qquad and}\qquad 1\rightarrow H_{2}\rightarrow
D_{8}\rightarrow D_{8}/\mathbb{Z}_{4}\rightarrow 1.  \label{eq:TwoExtensions}
\end{equation}%
Unfortunately, the ring structure on $E_{\infty }$ -terms of these LHS
spectral sequences does not coincide with the ring structure on $%
H^{\ast }(D_{8},\mathbb{Z})$.

\item[B.] The Bockstein spectral sequence of the exact couple
\begin{diagram}
H^{\ast }(D_{8},\mathbb{Z}) & &\rTo{\times 2}& & H^{\ast }(D_{8},%
\mathbb{Z}) \\
& \luTo{\delta}&&\ldTo{c} && \\
& & H^{\ast }(D_{8},\mathbb{F}_{2}) &&\\
\end{diagram}
where $d_{1}=c\circ \delta =\mathrm{Sq}^{1}:H^{\ast }(D_{8},\mathbb{F}%
_{2})\rightarrow H^{\ast +1}(D_{8},\mathbb{F}_{2})$ is given by $d_{1}\left(
x\right) =x^{2}$, $d_{1}\left( y\right) =y^{2}$ and $d_{1}\left( w\right)
=(x+y)w$ \cite[Theorem 2.7. page 127]{Adem-Milgram}. This approach allows
determination of the ring structure on $H^{\ast }(D_{8},\mathbb{Z})$.
\end{compactitem}

\subsubsection{\label{Sec:EvansView}Evens' view}

\noindent Let $V_{+-}^{\mathbb{C}}\oplus V_{-+}^{\mathbb{C}}=
\mathbb{C}\oplus\mathbb{C}$, $V_{--}^{\mathbb{C}}=\mathbb{C}$ and $U_{2}^{
\mathbb{C}}=\mathbb{C}$ be the complex $D_{8}$-representations given by

\begin{compactitem}
\item[A.] For $(u,v)\in V_{+-}^{\mathbb{C}}\oplus V_{-+}^{\mathbb{C}}$:
\begin{equation*}
\varepsilon _{1}\cdot (u,v)=(u,-v)\text{, \quad }\varepsilon _{2}\cdot
(u,v)=(-u,v)\text{, \quad }\sigma \cdot (u,v)=(v,u).
\end{equation*}

\item[B.] For $u\in V_{--}^{\mathbb{C}}$:
\begin{equation*}
\varepsilon _{1}\cdot u=-u,\text{\quad }\varepsilon _{2}\cdot u=-u,\text{%
\quad }\sigma \cdot u=u.
\end{equation*}

\item[C.] For $u\in U_{2}^{\mathbb{C}}$:
\begin{equation*}
\varepsilon _{1}\cdot u=u\text{, \quad }\varepsilon _{2}\cdot u=u\text{,
\quad }\sigma \cdot u=-u.
\end{equation*}
\end{compactitem}

\noindent There are isomorphisms of real $D_{8}$-representations
\begin{equation*}
V_{+-}^{\mathbb{C}}\oplus V_{-+}^{\mathbb{C}}\cong \left( V_{+-}\oplus V_{-+}\right) ^{\oplus 2},~\text{\quad }V_{--}^{
\mathbb{C}}\cong \left( V_{--}\right) ^{\oplus 2},~\text{\quad }U_{2}^{\mathbb{C}}=\left( U_{2}\right) ^{\oplus 2}.
\end{equation*}

\noindent Let $\chi _{1}$, $\xi \in H^{\ast }(D_{8},\mathbb{Z})$ be $1$-dimensional complex $D_{8}$-representations given by character
(here we assume the identification $c_{1}:\mathrm{Hom}(G,U\left( 1\right)
)\rightarrow H^{2}(G,\mathbb{Z})$, \cite[page 286]{Atiyah}):
\begin{equation*}
\begin{array}{lll}
\chi _{1}(\varepsilon _{1})=1, & \chi _{1}(\varepsilon _{2})=1, & \chi
_{1}(\sigma )=-1, \\
\xi (\varepsilon _{1})=-1, & \xi (\varepsilon _{2})=-1, & \xi (\sigma )=-1.%
\end{array}%
\end{equation*}%
Then $\chi _{1}=U_{2}^{\mathbb{C}}$, $\xi =U_{2}^{\mathbb{C}}\otimes V_{--}^{\mathbb{C}}$ and consequently
\begin{equation}
\mathrm{c}_{1}(U_{2}^{\mathbb{C}})=\chi _{1},\text{ }\qquad \text{and\qquad }\mathrm{c}_{1}(U_{2}^{
\mathbb{C}})+\mathrm{c}_{1}(V_{--}^{\mathbb{C}})=\xi .  \label{eq:1-chernClasses}
\end{equation}

\medskip

\noindent The cohomology $H^{\ast }(D_{8},\mathbb{Z})$ is given in \cite[pages 191-192]{Evans} by
\begin{equation}
H^{\ast }(D_{8},\mathbb{Z})=\mathbb{Z}
\lbrack \xi ,\chi _{1},\zeta ,\chi ]  \label{eq:Evans-Presentation}
\end{equation}%
where%
\begin{equation*}
\deg \xi =\deg \chi _{1}=2,~\deg \zeta =3,~\deg \chi =4
\end{equation*}%
and%
\begin{equation}
2\xi =2\chi _{1}=2\zeta =4\chi =0\text{, }\chi _{1}^{2}=\xi \cdot \chi _{1}%
\text{, }\zeta ^{2}=\xi \cdot \chi \text{.}  \label{relations}
\end{equation}%
There are four $1$-dimensional irreducible complex representations of $%
D_{8} $:%
\begin{equation*}
1,~\xi =U_{2}^{\mathbb{C}}\otimes V_{--}^{\mathbb{C}}\text{, }\chi _{1}=U_{2}^{
\mathbb{C}}\text{, }\xi \otimes \chi _{1}=V_{--}^{\mathbb{C}},
\end{equation*}%
and one $2$-dimensional complex representation which is denoted by $\rho $
in \cite[pages 191-192]{Evans}:%
\begin{equation*}
\rho =V_{+-}^{\mathbb{C}}\oplus V_{-+}^{\mathbb{C}}.
\end{equation*}%
It is computed in \cite[pages 191-192]{Evans} that
\begin{equation}
\mathrm{c}(V_{+-}^{\mathbb{C}}\oplus V_{-+}^{\mathbb{C}})=1+\xi +\chi \text{ \qquad and \qquad }\mathrm{c}_{2}(V_{+-}^{
\mathbb{C}}\oplus V_{-+}^{\mathbb{C}})=\chi .  \label{eq:2-chernClasses}
\end{equation}

\medskip

\noindent The relations (\ref{eq:1-chernClasses}) and (\ref%
{eq:2-chernClasses}) along with Proposition \ref{Prop:IndexCherClasses}
imply the following statement.

\begin{proposition}
$\mathrm{Index}_{D_{8},\mathbb{Z}}S(V_{--}^{
\mathbb{C}})=\langle \xi +\chi _{1}\rangle $, $\mathrm{Index}_{D_{8},\mathbb{Z}%
}S(U_{2}^{\mathbb{C}})=\langle \chi _{1}\rangle $, $\mathrm{Index}_{D_{8},\mathbb{Z}}S(V_{+-}^{
\mathbb{C}}\oplus V_{-+}^{\mathbb{C}})=\langle \chi \rangle $.
\end{proposition}

\medskip

\noindent Before proceeding to the Bockstein spectral sequence approach we
give descriptions of the $E_{2}$-terms of two LHS spectral sequences. Even
though it is not an easy consequence, it can be proved that both spectral
sequences stabilize and that $E_{2}=E_{\infty }$.

\paragraph*{LHS spectral sequences of the extension $1\rightarrow
H_{1}\rightarrow D_{8}\rightarrow D_{8}/H_{1}\rightarrow 1$.}

The LHS spectral sequence of this extension (\ref{eq:E2}) allows
computation of the cohomology ring $H^{\ast }(D_{8},\mathbb{F}_{2})$ with $%
\mathbb{F}_{2}$ coefficients. If we now consider $\mathbb{Z}$ coefficients, then the $E_{2}$-term has the form
\begin{equation}
E_{2}^{p,q}=H^{p}(D_{8}/H_{1},H^{q}\left( H_{1},\mathbb{Z}\right) )\cong H^{p}(\mathbb{Z}_{2},H^{q}\left( (
\mathbb{Z}_{2})^{2},\mathbb{Z}\right) ).  \label{eq:E2withZ}
\end{equation}%
The spectral sequence converges to the graded group $\mathrm{Gr}\left(
H^{p+q}(D_{8},\mathbb{Z})\right) $ associated with $H^{p+q}(D_{8},\mathbb{Z})$ appropriately filtered.
To present the $E_{2}$-term we choose generators
of $H^{\ast }\left( H_{1},\mathbb{Z}\right) $ consistent with the choices made in Lemma \ref%
{Lemma:CohomologyD8}. Let $c:\mathbb{Z}\rightarrow \mathbb{F}_{2}$ be reduction $\mathrm{mod}~2$ and $c_{\ast }:H^{\ast
}(D_{8},\mathbb{Z})\rightarrow H^{\ast }(D_{8},\mathbb{F}_{2})$ the induced
map in cohomology. Consider the following presentation of the $H_{1}$
cohomology ring:%
\begin{equation}
H^{\ast }\left( H_{1},\mathbb{Z}\right) =\mathbb{Z}\lbrack \alpha ,\alpha +\beta ]\otimes
\mathbb{Z}\lbrack \mu ]  \label{eq:CohomologyH1withZ}
\end{equation}%
where
\begin{compactitem}
\item[A.] $\deg (\alpha )=\deg (\beta )=2$, $\deg (\mu )=3;$

\item[B.] $2\alpha =2\beta =2\mu =0$ and $\mu ^{2}=\alpha \beta (\alpha
+\beta );$

\item[C.] $\sigma $ action on $H^{\ast }\left( H_{1},\mathbb{Z}\right)$
is given by $\sigma \cdot \alpha =\alpha +\beta $ and $\sigma
\cdot \mu =\mu ;$

\item[D.] $c_{\ast }(\alpha )=a^{2}$, $c_{\ast }\left( \beta \right) =b^{2}$%
, $c_{\ast }\left( \mu \right) =ab(a+b)$.
\end{compactitem}

\begin{figure}[htb]
\centering
\includegraphics[scale=0.50]{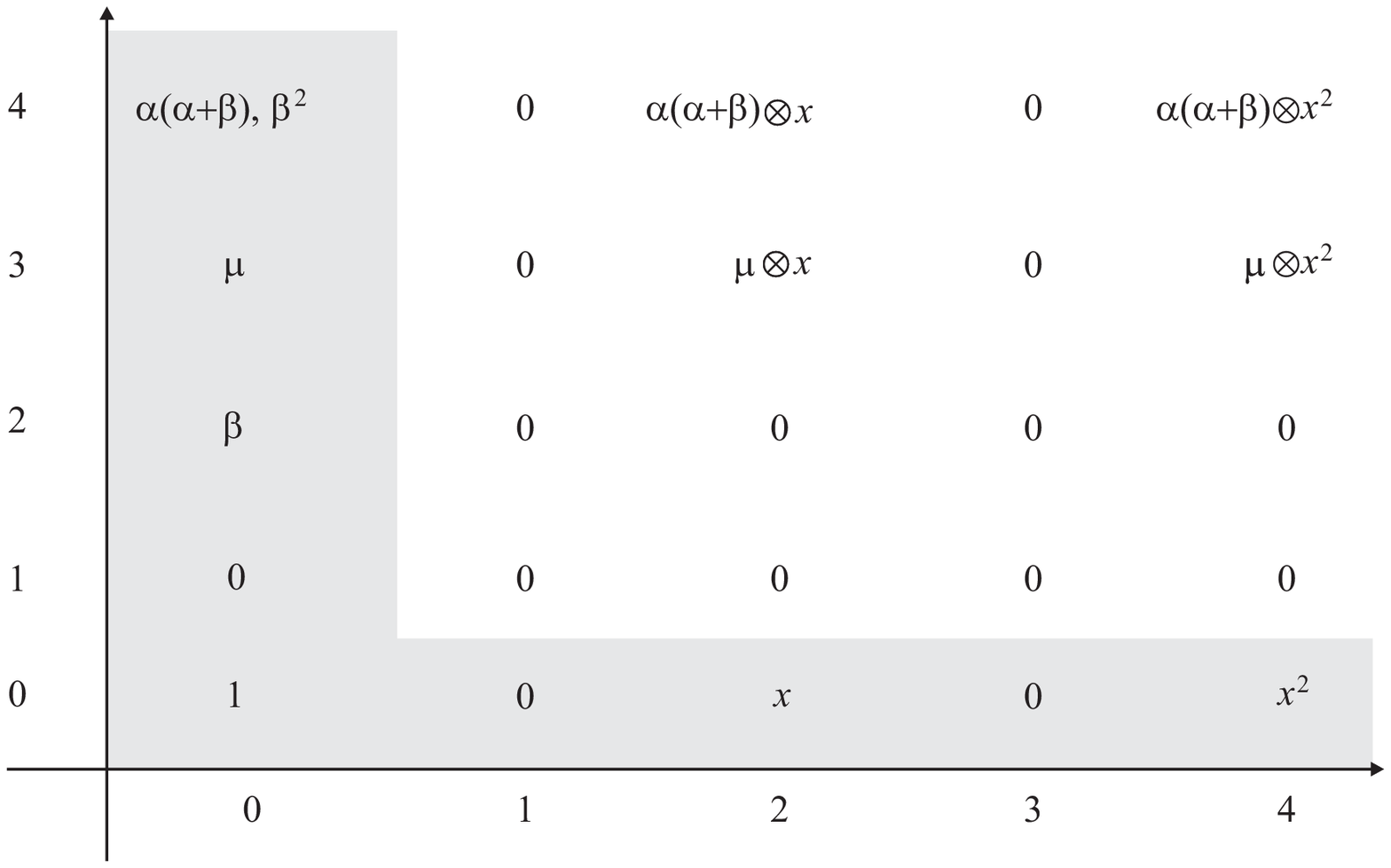}
\caption{{\protect\small {{$E_{2}$-term} of extension $1\rightarrow
H_{1}\rightarrow D_{8}\rightarrow D_{8}/H_{1}\rightarrow 1.$}}}
\label{Fig-3.1}
\end{figure}

\noindent Now the $E_{2}$-term (Figure \ref{Fig-3.1}) is given by%
\begin{equation*}
{\ E_{2}^{p,q}\cong H^{p}(\mathbb{Z}_{2},H^{q}\left( (\mathbb{Z}_{2})^{2},\mathbb{Z}\right) )\cong \left\{
\begin{array}{lll}
H^{p}(\mathbb{Z}_{2},\mathbb{Z}), &  & q=0 \\
0, &  & q=1 \\
H^{p}(\mathbb{Z}_{2},\mathbb{F}_{2}[\mathbb{Z}_{2}]), &  & q=2 \\
H^{p}(\mathbb{Z}_{2},\mathbb{F}_{2}), &  & q=3 \\
\ldots, &  & q>3.%
\end{array}%
\right. }
\end{equation*}%
The morphism of LHS spectral sequences of the extension $1\rightarrow
H_{1}\rightarrow D_{8}\rightarrow D_{8}/H_{1}\rightarrow 1$ induced by the
$\mathrm{mod}~2$ reduction $c:\mathbb{Z}\rightarrow \mathbb{F}_{2}$ (Proposition \ref{Prop:Res-1} E.3) gives a proof
that $E_{2}=E_{\infty }$ for $\mathbb{Z}$ coefficients. The ring structures
on $E_{\infty }$ and $H^{\ast }(D_{8},\mathbb{Z})$ do not coincide. Moreover there is no element in $E_{\infty }$ of
exponent $4$. One thing is clear: the element $\mu $ in the $E_{2}=E_{\infty }$%
-term coincides with the element $\zeta $ in the Evens' presentation (\ref%
{eq:Evans-Presentation}) of $H^{\ast }(D_{8},\mathbb{Z})$.

\paragraph*{LHS spectral sequences of the extension $1\rightarrow
H_{2}\rightarrow D_{8}\rightarrow D_{8}/H_{2}\rightarrow 1$.}

The $E_{2}$-term has the form:%
\begin{equation*}
E_{2}^{p,q}=H^{p}(D_{8}/H_{2},H^{q}\left( H_{2},\mathbb{Z}
\right) )\cong H^{p}(\mathbb{Z}_{2},H^{q}\left(\mathbb{Z}_{4},\mathbb{Z}\right) )\cong {\ \left\{
\begin{array}{lll}
H^{p}(\mathbb{Z}_{2},\mathbb{Z}), &  & q=0 \\
0, &  & q\text{ odd} \\
H^{p}(\mathbb{Z}_{2},\mathcal{Z}_{4}), &  & q\text{ even and }4\nmid q \\
H^{p}(\mathbb{Z}_{2},\mathbb{Z}_{4}), &  & q>0\text{ even and }4|q,
\end{array}
\right. }
\end{equation*}%
where $\mathcal{Z}_{4}\cong\mathbb{Z}_{4}$ is a non-trivial $\mathbb{Z}_{2}$-module. Using \cite[Example 2, pages 58-59]{Brown} the $E_{2}$-term has
the shape given in the Figure \ref{Fig-3.2}.
\begin{figure}[tbh]
\centering
\includegraphics[scale=0.50]{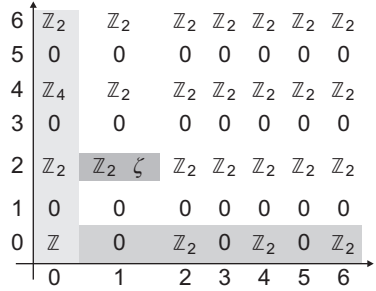}
\caption{{\protect\small {{$E_{2}$-term} of extension $1\rightarrow
H_{2}\rightarrow D_{8}\rightarrow D_{8}/H_{2}\rightarrow 1.$}}}
\label{Fig-3.2}
\end{figure}
This diagram provides just two hints: there might be elements of exponent $4$
in the cohomology $H^{\ast }(D_{8},\mathbb{Z})$ and definitely there is only one element $\zeta $ of degree $3$ from the
Evens' presentation.

\paragraph*{Conclusion.}

The LHS spectral sequences of different extensions gives an incomplete
picture of the cohomology ring with integer coefficients, $H^{\ast }(D_{8},\mathbb{Z})$.
Therefore, for the purposes of the computations with $\mathbb{Z}$ coefficients we use the Bockstein
spectral sequence utilizing results obtained from the LHS spectral sequence with $\mathbb{F}_{2}$ coefficients. Presentations
of these two spectral sequences will be used in the description of
the restriction diagram in Section \ref{Sec:RestrictionDiagram-Z}.

\subsubsection{The Bockstein spectral sequence view}

Let $G$ be a finite group. The exact sequence $0\rightarrow
\mathbb{Z}\overset{\times 2}{\rightarrow }\mathbb{Z}
\rightarrow \mathbb{F}_{2}\rightarrow 0$ induces a long exact sequence in
group cohomology, or an exact couple%
\begin{diagram}
H^{\ast }(G,\mathbb{Z})&&\rTo{\times 2}&& H^{\ast }(G,\mathbb{Z})\\
&\luTo{\delta}&&\ldTo{c}&\\
&&H^{\ast }(G,\mathbb{F}_{2}).&&\\
\end{diagram}
The spectral sequence of this exact couple is the Bockstein spectral
sequence. It converges to
\begin{equation*}
\left( H^{\ast }(G,\mathbb{Z})/\mathrm{torsion}\right) \otimes \mathbb{F}_{2}
\end{equation*}%
which in the case of a finite group $G$ is just $\mathbb{F}_{2}$ in
dimension $0$. Here \textquotedblleft torsion\textquotedblright\ means $%
\mathbb{Z}$-torsion. The first differential $d_{1}=c\circ \delta $ is the
Bockstein homomorphism and in this case coincides with the first Steenrod
square $\mathrm{Sq}^{1}:H^{\ast }(G,\mathbb{F}_{2})\rightarrow H^{\ast +1}(G,%
\mathbb{F}_{2})$.

\noindent Let $H$ be a subgroup of $G$. The restriction $\mathrm{res}%
_{H}^{G} $ commutes with the maps in the exact couples associated to the groups $%
G$ and $H $ and therefore induces a morphism of Bockstein spectral sequences
\cite[page 109 before 5.7.6]{Carlson}.

\medskip

Consider two Bockstein spectral sequences associated with $D_{8}$ and its
subgroup $H_{2}\cong \mathbb{Z}_{4}$.

\begin{compactitem}
\item[A.] \textbf{Group} $D_{8}$. The exact couple is
\begin{equation}
\begin{diagram}
H^{\ast }(D_{8},\mathbb{Z})&&\rTo{\times 2}&& H^{\ast}(D_{8},\mathbb{Z})\\
&\luTo{\delta}&&\ldTo{c}&\\
&&H^{\ast }(D_{8},\mathbb{F}_{2}).&&\\
\end{diagram}
\label{eq:Exact-1}
\end{equation}%
and $d_{1}=c\circ \delta =\mathrm{Sq}^{1}$ is given by $d_{1}\left( x\right)
=x^{2}$, $d_{1}\left( y\right) =y^{2}$ and $d_{1}\left( w\right) =(x+y)w$,
\cite[Theorem 2.7. page 127]{Adem-Milgram}. The derived couple is%
\begin{diagram}[size=2.5em,textflow]
2\cdot H^{\ast }(D_{8},\mathbb{Z})&&\rTo{\times 2}&& 2\cdot H^{\ast }(D_{8},\mathbb{Z})\\
&\luTo{\delta_1}&&\ldTo{c_1}&\\
&&\langle x^{2},y^{2},xw,yw,w^{2}\rangle /\langle x^{2},y^{2},xw+yw\rangle.&&\\
\end{diagram}
Then by \cite[Remark 5.7.4, page 108]{Carlson} there are elements $\mathcal{X%
},\mathcal{Y}\in H^{2}(D_{8},\mathbb{Z})$, $\mathcal{M}\in H^{3}(D_{8},%
\mathbb{Z})$ of exponent $2$ such that $c_{*}(\mathcal{X})=x^{2}$, $c_{*}(\mathcal{Y}%
)=y^{2}$, $c_{*}(\mathcal{M})=(x+y)w$ and $\mathcal{XY}=0$.

\item[B.] \textbf{Group} $\mathbb{Z}_{4}$. The exact couple is%
\begin{equation}
\begin{diagram}
H^{\ast }(\mathbb{Z}_{4},\mathbb{Z})&&\rTo{\times 2}&& H^{\ast }(\mathbb{Z}_{4},\mathbb{Z})\\
&\luTo{\delta}&&\ldTo{c}&\\
&&H^{\ast }(\mathbb{Z}_{4},\mathbb{F}_{2}).&&\\
\end{diagram}
\label{eq:Exact-1-1}
\end{equation}
Since $H^{\ast }(\mathbb{Z}_{4},\mathbb{Z})=\mathbb{Z}[U]/\langle 4U\rangle $%
, $\deg U=2$ and $H^{\ast }(\mathbb{Z}_{4},\mathbb{F}_{2})=\mathbb{F}_{2}{\
[e,u]/\langle e}^{2}{\rangle }$, $\deg e=1$, $\deg u=2$, the unrolling of
the exact couple to a long exact sequence \cite[Proposition 6.1, page 71]%
{Brown}%
\begin{equation*}
\begin{diagram}
0 & \rTo &
\begin{array}{c}
H^{0}(\mathbb{Z}_{4},\mathbb{Z}) \\
\begin{array}{cc}
\mathbb{Z} & ,~1%
\end{array}%
\end{array}
& \rTo{\times 2} &
\begin{array}{c}
H^{0}(\mathbb{Z}_{4},\mathbb{Z}) \\
\begin{array}{cc}
\mathbb{Z} & ,~1%
\end{array}%
\end{array}
& \rTo &
\begin{array}{c}
H^{0}(\mathbb{Z}_{4},\mathbb{F}_{2}) \\
\begin{array}{cc}
\mathbb{F}_{2} & ,~1%
\end{array}%
\end{array}
\\
& \rTo^{\delta } &
\begin{array}{c}
H^{1}(\mathbb{Z}_{4},\mathbb{Z}) \\
0%
\end{array}
& \rTo{\times 2} &
\begin{array}{c}
H^{1}(\mathbb{Z}_{4},\mathbb{Z}) \\
0
\end{array}
& \rTo &
\begin{array}{c}
H^{1}(\mathbb{Z}_{4},\mathbb{F}_{2}) \\
\begin{array}{cc}
\mathbb{F}_{2} & ,~e%
\end{array}%
\end{array}
\\
& \rTo^{\delta } &
\begin{array}{c}
H^{2}(\mathbb{Z}_{4},\mathbb{Z}) \\
\begin{array}{cc}
\mathbb{Z}_{4} & ,~U%
\end{array}%
\end{array}
& \rTo{\times 2} &
\begin{array}{c}
H^{2}(\mathbb{Z}_{4},\mathbb{Z}) \\
\begin{array}{cc}
\mathbb{Z}_{4} & ,~U%
\end{array}%
\end{array}
& \rTo&
\begin{array}{c}
H^{2}(\mathbb{Z}_{4},\mathbb{F}_{2}) \\
\begin{array}{cc}
\mathbb{F}_{2} & ,~u%
\end{array}%
\end{array}
\\
& \rTo^{\delta } &
\begin{array}{c}
H^{3}(\mathbb{Z}_{4},\mathbb{Z}) \\
0%
\end{array}
& \rTo{\times 2} &
\begin{array}{c}
H^{3}(\mathbb{Z}_{4},\mathbb{Z}) \\
0%
\end{array}
& \rTo &
\begin{array}{c}
H^{3}(\mathbb{Z}_{4},\mathbb{F}_{2}) \\
\begin{array}{cc}
\mathbb{F}_{2} & ,~eu%
\end{array}%
\end{array}
\\
& \rTo^{\delta } &
\begin{array}{c}
H^{4}(\mathbb{Z}_{4},\mathbb{Z}) \\
U^{2}%
\end{array}
&  & \ldots &  &
\end{diagram}
\end{equation*}%
allows us to show that for $j\geq 0:$%
\begin{equation*}
\delta (u^{i})=0\text{ and }\delta (eu^{i})=2U^{i+1}\text{.}
\end{equation*}%
Thus $d_{1}=0$ and the derived couple is
\begin{diagram}
2\cdot H^{\ast }(\mathbb{Z}_{4},\mathbb{Z})&&\rTo{\times 2}&& 2\cdot H^{\ast }(\mathbb{Z}_{4},\mathbb{Z})\\
&\luTo{\delta_1}&&\ldTo{c_1}&\\
&&H^{\ast }(\mathbb{Z}_{4},\mathbb{F}_{2}).&&\\
\end{diagram}
\noindent Moreover, by definition of the differential of a derived couple we
have that%
\begin{equation*}
d_{2}(u^{i})=0\text{ and }d_{2}(eu^{i})=u^{i+1}\text{.}
\end{equation*}
\end{compactitem}

\noindent The restriction map $\mathrm{res}_{H_{2}}^{D_{8}}:H^{\ast }(D_{8},%
\mathbb{F}_{2})\rightarrow H^{\ast }(H_{2},\mathbb{F}_{2})$ is determined by
the restriction diagram (\ref{eq:diagramD8-1}). Therefore, the morphism
between spectral sequences induced by the restriction $\mathrm{res}%
_{H_{2}}^{D_{8}}$ implies that:%
\begin{equation*}
\mathrm{res}_{H_{2}}^{D_{8}}\left( d_{2}[xw]\right) =d_{2}\left( \mathrm{res}%
_{H_{2}}^{D_{8}}[xw]\right) =d_{2}(eu)=u^{2}
\end{equation*}%
and consequently%
\begin{equation*}
d_{2}[xw]=[w^{2}]
\end{equation*}%
Here $[~\cdot~]$ denotes the class in the quotient $\langle
x^{2},y^{2},xw,yw,w^{2}\rangle /\langle x^{2},y^{2},xw+yw\rangle $. Thus, by
\cite[Remark 5.7.4, page 108]{Carlson} there is an element $\mathcal{W}\in
H^{4}(D_{8},\mathbb{Z})$ of exponent $4$ such that $c_{*}(\mathcal{W})=w^{2}$
and $\mathcal{M}^{2}=\mathcal{W(X+Y)}$. The second derived couple of (\ref%
{eq:Exact-1}) stabilizes. Thus the cohomology ring $H^{\ast }(D_{8},\mathbb{Z%
})$ and the map $c_{*}:H^{\ast }(D_{8},\mathbb{Z})\longrightarrow H^{\ast
}(D_{8},\mathbb{F}_{2})$ are described.

\begin{theorem}\label{Th:Bo-Presentation}
The cohomology ring $H^{\ast }(D_{8},\mathbb{Z})$ can be presented by%
\begin{equation*}
H^{\ast }(D_{8},\mathbb{Z})=\mathbb{Z}\lbrack \mathcal{X},\mathcal{Y},\mathcal{M},\mathcal{W}]/\mathcal{I}
\end{equation*}%
where%
\begin{equation*}
\deg \mathcal{X}=\deg \mathcal{Y}=2,~\deg \mathcal{M}=3,~\deg \mathcal{W}=4
\end{equation*}%
and the ideal $\mathcal{I}$ is generated by the equations
\begin{equation}
2\mathcal{X}=2\mathcal{Y}=2\mathcal{M}=4\mathcal{W}=0\text{, }\mathcal{XY}=0%
\text{, }\mathcal{M}^{2}=\mathcal{W(X+Y)}\text{.}  \label{relations-2}
\end{equation}%
The map $c_{*}:H^{\ast }(D_{8},\mathbb{Z})\longrightarrow H^{\ast }(D_{8},%
\mathbb{F}_{2})$, induced by the reduction of coefficients $\mathbb{%
Z\rightarrow F}_{2}$, is given by%
\begin{equation}
\mathcal{X\mapsto }x^{2}~,~\qquad \mathcal{Y\mapsto }y^{2}~,~\qquad \mathcal{%
M\mapsto }w(x+y)~,~\qquad \mathcal{W\mapsto }w^{2}\text{.}  \label{eq:Map-j}
\end{equation}
\end{theorem}

\medskip

\begin{remark*}
The correspondence between the Evens' and Bockstein spectral sequence view is given by%
\begin{equation}
\mathcal{X\leftrightarrow \chi }_{1}~,~\qquad \mathcal{Y\leftrightarrow \xi
+\chi }_{1}~,~\qquad \mathcal{M\leftrightarrow \zeta ~},~\qquad \mathcal{%
W\leftrightarrow \chi }  \label{eq:Correspondence}
\end{equation}
\end{remark*}

\subsection{\label{Sec:RestrictionDiagram-Z}The $D_{8}$-diagram with
coefficients in $\mathbb{Z}$}

Let $G$ be a finite group and $R$ and $S$ rings. A ring homomorphism $\phi
:R\rightarrow S$ induces a morphism of diagrams (natural transformation of
covariant functors) $\Phi~:~$\textrm{Res}$_{(R)}\rightarrow $\textrm{Res}$%
_{(S)}$. The morphism $\Phi$ on each object $H\in \mathrm{Ob}(\mathfrak{G})$ is defined by the coefficient reduction
$\Phi(H):H^{*}(H,R)\rightarrow H^{*}(H,S)$ induced by $\phi$. Particularly in this section, as a tool for the reconstruction of
the diagram \textrm{Res}$_{(\mathbb{Z})}$, we use the diagram morphism $C:
\mathrm{Res}_{(\mathbb{Z})}\rightarrow \mathrm{Res}_{(\mathbb{F}_{2})}$
induced by the coefficient reduction homomorphism $c:\mathbb{Z}\rightarrow \mathbb{F}_{2}$.

\subsubsection{The $\mathbb{Z}_{2}\times \mathbb{Z}_{2}$-diagram}

The cohomology restriction diagram \textrm{Res}$_{(\mathbb{F}_{2})}$ of the
elementary abelian $2$-group $\mathbb{Z}_{2}\times\mathbb{Z}_{2}$ is given in the diagram (\ref{eq:DiagramZ2xZ2}). Using the
presentation of cohomology $H^{\ast }(\mathbb{Z}_{2}\times\mathbb{Z}_{2},\mathbb{Z})$ and the homomorphism $H^{\ast }(
\mathbb{Z}_{2}\times\mathbb{Z}_{2},\mathbb{Z})\rightarrow H^{\ast }(\mathbb{Z}_{2}\times\mathbb{Z}_{2},\mathbb{F}_{2})$ given
in Example \ref{Ex:SxSwithZ} we can reconstruct the restriction diagram \textrm{Res}$_{(\mathbb{Z})}$:
\begin{equation}
\begin{diagram}
&&&
\begin{tabular}{|l|l|}
\hline
$\mathbb{Z}_{2}\times\mathbb{Z}
{2}$ & $\mathbb{Z[\tau }_{1},\tau _{2}\mathbb{]\otimes Z}[\mu ]$ \\ \hline
\multicolumn{2}{|l|}{$\deg \mathbb{\tau }_{1}=\deg \tau _{2}=2$, $\deg
\mathbb{\mu }=3$} \\ \hline
\multicolumn{2}{|l|}{$2\tau _{1}=2\tau _{2}=2\mu =0,$} \\ \hline
\multicolumn{2}{|l|}{$\mu ^{2}=\tau _{1}\tau _{2}(\tau _{1}+\tau _{2})$} \\
\hline
\end{tabular}
&&&\\
&&&&&&\\
&\ldTo{\begin{array}{cc}
{\mathbb{\tau }_{1}\mapsto 0,\tau _{2}\mapsto }\mathbb{\theta }_{1}, &  \\
\mu {\mapsto 0} &
\end{array}}&&\dTo{\begin{array}{ccc}
{\mathbb{\tau }_{1}\mapsto }\mathbb{\theta }_{2}, \\
\tau _{2}\mapsto 0,\\
 \mu {\mapsto 0}
\end{array}}&&\rdTo{\begin{array}{cc}
\mathbb{\tau }_{1}\mapsto \mathbb{\theta }_{3},\tau _{2}\mapsto \mathbb{\theta }_{3}, \\
\mu {\mapsto 0}&
\end{array}}&\\
&&&&&&\\
\begin{tabular}{|l|l|}
\hline
$\mathbb{Z}_{2}$ & $\mathbb{Z[\theta }_{1}]$ \\ \hline
\multicolumn{2}{|l|}{$\deg \mathbb{\theta }_{1}=2$} \\ \hline
\multicolumn{2}{|l|}{$2\mathbb{\theta }_{1}=0$} \\ \hline
\end{tabular}&&&\begin{tabular}{|l|l|}
\hline
$\mathbb{Z}_{2}$ & $\mathbb{Z[\theta }_{2}]$ \\ \hline
\multicolumn{2}{|l|}{$\deg \mathbb{\theta }_{2}=2$} \\ \hline
\multicolumn{2}{|l|}{$2\mathbb{\theta }_{2}=0$} \\ \hline
\end{tabular}
&&&
\begin{tabular}{|l|l|}
\hline
$\mathbb{Z}_{2}$ & $\mathbb{Z[\theta }_{3}]$ \\ \hline
\multicolumn{2}{|l|}{$\deg \mathbb{\theta }_{3}=2$} \\ \hline
\multicolumn{2}{|l|}{$2\mathbb{\theta }_{3}=0$} \\ \hline
\end{tabular}
\\
\end{diagram}
\label{eq:DiagramZ2xZ2withZ}
\end{equation}

\subsubsection{\label{Sec:RestrictionDiagram-D8-Z}The $D_{8}$-diagram}

In a similar fashion, using:

\begin{compactitem}
\item the $D_{8}$ restriction diagram (\ref{eq:diagramD8-1}) and (\ref%
{eq:diagramD8-2}) with $\mathbb{F}_{2}$ coefficients,

\item the $\mathbb{Z}_{2}\times\mathbb{Z}_{2}$ restriction diagrams (\ref{eq:DiagramZ2xZ2withZ}) with
$\mathbb{Z}$ coefficients,

\item the presentation of cohomology $H^{\ast }(H_{1},\mathbb{Z})$ given in (\ref{eq:CohomologyH1withZ}),

\item the Bockstein presentation of $H^{\ast
}(D_{8},\mathbb{Z})$ given in Theorem \ref{Th:Bo-Presentation},

\item a glance at the restriction maps $\mathrm{res}_{H_{1}}^{D_{8}}$ and $%
\mathrm{res}_{H_{2}}^{D_{8}}$ obtained from the $E_{2}=E_{\infty }$ terms of
the LHS spectral sequences Figure \ref{Fig-3.1} and Figure \ref{Fig-3.2}, and

\item the homomorphism $c_{*}:H^{\ast }(D_{8},\mathbb{Z})\rightarrow H^{\ast }(D_{8},\mathbb{F}_{2})$ described in (\ref{eq:Map-j}),
\end{compactitem}
\noindent we can reconstruct the restriction diagram of $D_{8}$ with $\mathbb{Z}$ coefficients.
\begin{equation}
\begin{diagram}
&&&\begin{tabular}{|l|l|}
\hline
${D}_{8}$ & $\mathbb{Z}\lbrack \mathcal{X},\mathcal{Y},\mathcal{M},\mathcal{W}]$ \\ \hline
\multicolumn{2}{|l|}{$\deg :2,~2,~3,~4$} \\ \hline
\multicolumn{2}{|l|}{$2\mathcal{X}=2\mathcal{Y}=2\mathcal{M}=4\mathcal{W}=0%
\text{,}$} \\ \hline
\multicolumn{2}{|l|}{$\mathcal{XY}=0\text{, }\mathcal{M}^{2}=\mathcal{W(X+Y)}
$} \\ \hline
\end{tabular}&&&\\
&&&&&&\\
&\ldTo{\begin{array}{cc}
{\mathcal{X}\mapsto 0},~\mathcal{Y}{\mapsto }\mathcal{\beta },\mathcal{M}{%
\mapsto }\mu &  \\
\mathcal{W}{\mapsto }\mathcal{\alpha (\alpha +\beta )} &
\end{array}}&&\dTo_{\begin{array}{ccc}
{\mathcal{X}\mapsto }2U \\
\mathcal{M}{\mapsto }0 \\
\end{array}}^{\begin{array}{ccc}
\mathcal{Y}\mapsto 2U \\
\mathcal{W}{\mapsto U}^{2}\\
\end{array}}&&\rdTo{\begin{array}{cc}
& {\mathcal{X}\mapsto \delta },~\mathcal{Y}{\mapsto 0},\mathcal{M}{\mapsto }%
\eta \\
 & \mathcal{W}{\mapsto }\mathcal{\gamma (\gamma +\delta )}%
\end{array}}&\\
&&&&&&\\
\begin{tabular}{|l|l|}
\hline
${H}_{1}$ & $\mathbb{Z}\lbrack \mathcal{\alpha },\mathcal{\alpha +\beta },\mu ]$ \\ \hline
\multicolumn{2}{|l|}{$\deg :2,~2,~3,$} \\ \hline
\multicolumn{2}{|l|}{$2\mathcal{\alpha }=2\mathcal{\beta }=2\mu =0\text{,}$}
\\ \hline
\multicolumn{2}{|l|}{$\mu ^{2}=\alpha \beta (\mathcal{\alpha +\beta })$} \\
\hline
\end{tabular}&&&\begin{tabular}{|l|l|}
\hline
${H}_{2}$ & $\mathbb{Z}\lbrack U]$ \\ \hline
\multicolumn{2}{|l|}{$\deg :2$} \\ \hline
\multicolumn{2}{|l|}{$4U=0\text{,}$} \\ \hline
\end{tabular}&&&\begin{tabular}{|l|l|}
\hline
${H}_{3}$ & $\mathbb{Z}\lbrack \mathcal{\gamma },\mathcal{\gamma +\delta },\eta ]$ \\ \hline
\multicolumn{2}{|l|}{$\deg :2,~2,~3,$} \\ \hline
\multicolumn{2}{|l|}{$2\mathcal{\gamma }=2\mathcal{\delta }=2\eta =0\text{,}$%
} \\ \hline
\multicolumn{2}{|l|}{$\eta ^{2}=\gamma \delta (\gamma +\delta )$} \\ \hline
\end{tabular}\\
~~&&&&&&~~\\
&\rdTo{
\begin{array}{cc}
\mathcal{\alpha }{\mapsto }\theta _{3},~\mathcal{\beta }{\mapsto }0 &  \\
\mu {\mapsto }0 &
\end{array}}&&\dTo{\begin{array}{c}
U{\mapsto }\theta _{3} \\
\end{array}}&&\ldTo{\begin{array}{cc}
& \gamma {\mapsto }\theta _{3},\delta {\mapsto }0 \\
 & \eta {\mapsto 0}%
\end{array}}&\\
&&&\begin{tabular}{|c|c|c|}
\hline
$K_{3}$ & $\mathbb{Z}\lbrack \theta _{3}]$ & $\deg \theta _{3}=2$ \\ \hline
\end{tabular}&&&\\
\end{diagram}
\label{eq:diagramD8-1-Z}
\end{equation}

\noindent Now the determination of the diagram morphism $C:$\textrm{Res}$_{(\mathbb{Z}
)}\rightarrow $\textrm{Res}$_{(\mathbb{F}_{2})}$ induced by the coefficient
reduction homomorphism $c:\mathbb{Z}\rightarrow \mathbb{F}_{2}$ is just a routine exercise.

\section{\label{Sec:IndexOfSpheresD8-F2}$\mathbf{Index}_{D_{8},\mathbb{F}%
_{2}}\mathbf{S(R}_{4}^{\oplus j}\mathbf{)}$}

In this section we show the following equality:
\begin{equation*}
\mathrm{Index}_{D_{8},\mathbb{F}_{2}}S(R_{4}^{\oplus j})=\mathrm{Index}%
_{D_{8},\mathbb{F}_{2}}^{3j+1}S(R_{4}^{\oplus j})=\langle w^{j}y^{j}\rangle .
\end{equation*}

\medskip

\noindent The $D_{8}$-representation $R_{4}^{\oplus j}$ can be decomposed
into a sum of irreducibles in the following way%
\begin{equation*}
R_{4}=\left( V_{-+}\oplus V_{+-}\right) \oplus V_{--}~\Rightarrow
~R_{4}^{\oplus j}=\left( V_{-+}\oplus V_{+-}\right) ^{\oplus j}\oplus
V_{--}^{\oplus j}
\end{equation*}%
where $V_{-+}\oplus V_{+-}$ is a $2$-dimensional irreducible $D_{8}$%
-representation. Since in this section $\mathbb{F}_{2}$ coefficients are
assumed, Proposition \ref{Prop:IndexOf Join} implies that computing the
indexes of the spheres $S(V_{-+}\oplus V_{+-})$ and $S(V_{--})$ suffices.
The strategy employed uses Proposition \ref{prop:Res-Index} and the
following particular facts.

\begin{compactitem}
\item[\textbf{A.}] Let $X=S(T)$ for some $D_{8}$-representation $T$. Then
the $E_{2}$-term of the Serre spectral sequence associated to $\mathrm{E}%
D_{8}\times _{D_{8}}X$ is
\begin{equation}
E_{2}^{p,q}=H^{p}(D_{8},\mathbb{F}_{2})\otimes H^{q}(X,\mathbb{F}_{2}).
\label{eq:SSS-F2}
\end{equation}%
The local coefficients are trivial since $X$ is a sphere and the
coefficients are $\mathbb{F}_{2}$. Since only $\partial _{\dim T,\mathbb{F}%
_{2}}$ may be $\neq 0$, from the multiplicative property of the spectral
sequence it follows that
\begin{equation*}
\mathrm{Index}_{D_{8},\mathbb{F}_{2}}X=\langle \partial _{\dim V,\mathbb{F}%
_{2}}^{0,\dim V-1}(1\otimes l)\rangle
\end{equation*}%
where $l\in H^{\dim V-1}(X,\mathbb{F}_{2})$ is the generator. Therefore, $%
\mathrm{Index}_{D_{8},\mathbb{F}_{2}}(X)=\mathrm{Index}_{D_{8},\mathbb{F}%
_{2}}^{\dim V+1}(X)$.

\item[\textbf{B.}] For any subgroup $H$ of $D_{8}$, with some abuse of
notation,
\begin{equation}
\Gamma _{\dim V}^{\dim V,0}\circ \partial _{\dim V,\mathbb{F}_{2}}^{0,\dim
V-1}(1\otimes l)=\partial _{\dim V,\mathbb{F}_{2}}^{0,\dim V-1}\circ \Gamma
_{\dim V}^{0,\dim V-1}(1\otimes l),  \label{eq:01}
\end{equation}%
where $\Gamma $ denotes the restriction morphism of Serre spectral sequences
introduced in Proposition \ref{Prop:Res-1}(D). Therefore, for every subgroup
$H$ of $D_{8}$ we get%
\begin{equation*}
\mathrm{Index}_{D_{8},\mathbb{F}_{2}}X=\langle a\rangle ,\quad \mathrm{Index}%
_{H,\mathbb{F}_{2}}X=\langle a_{H}\rangle \ \ \Longrightarrow \ \ \mathrm{res%
}_{K}^{G}(a)=a_{H}.
\end{equation*}%
In particular, if $a_{H}\neq 0$ then $a\neq 0$.

Our computation of $\mathrm{Index}_{D_{8},\mathbb{F}_{2}}X$ for $%
X=S(V_{-+}\oplus V_{+-})$ and $X=S(V_{--})$ has two steps:

\begin{compactitem}[$\bullet$]
\item compute $\mathrm{Index}_{H,\mathbb{F}_{2}}X=\langle a_{H}\rangle $ for
all proper subgroups $H$ of $D_{8}$,

\item search for an element $a\in H^{\ast }(D_{8},\mathbb{F}_{2})$ such that
for every computed $a_{H}$
\begin{equation*}
\mathrm{res}_{K}^{G}(a)=a_{H}.
\end{equation*}
\end{compactitem}
\end{compactitem}

\subsection{\label{sec:Index-W}$\mathrm{Index}_{D_{8},\mathbb{F}%
_{2}}S(V_{-+}\oplus V_{+-})=\langle w\rangle $}

Proposition \ref{Prop:IndexOfSphere} and the properties of the action of $D_{8}$ on
$V_{-+}\oplus V_{+-}$ provide the following information:
\begin{equation*}
\mathrm{Index}_{H_{1},\mathbb{F}_{2}}S(V_{-+}\oplus V_{+-})=\left\{
\begin{array}{cc}
\langle a(a+b)\rangle & \text{or} \\
\langle b(a+b)\rangle & \text{or} \\
\langle ab\rangle . &
\end{array}%
\right.
\end{equation*}%
Since initially we do not know which of the possible generators $a$, $b$, $a+b$
of $\mathbb{F}_{2}[a,b]$ correspond to the generators $\varepsilon _{1}$, $%
\varepsilon _{2}$, $\varepsilon _{1}\varepsilon _{2}$, we have to take all
three possibilities into account. Similarly:
\begin{equation*}
\mathrm{Index}_{H_{3},\mathbb{F}_{2}}S(V_{-+}\oplus V_{+-})=\left\{
\begin{array}{cc}
\langle c(c+d)\rangle & \text{or} \\
\langle d(c+d)\rangle & \text{or} \\
\langle cd\rangle . &
\end{array}%
\right.
\end{equation*}%
Furthermore,
\begin{equation*}
\begin{tabular}{lll}
$\varepsilon _{1}$ acts trivially on $V_{+-}$ & $\Rightarrow $ & $\mathrm{%
Index}_{K_{1},\mathbb{F}_{2}}S(V_{-+}\oplus V_{+-})=0$ \\
$\varepsilon _{2}$ acts trivially on $V_{-+}$ & $\Rightarrow $ & $\mathrm{%
Index}_{K_{2},\mathbb{F}_{2}}S(V_{-+}\oplus V_{+-})=0$ \\
$\sigma $ acts trivially on $\{(x,x)\in V_{-+}\oplus V_{+-}\}$ & $%
\Rightarrow $ & $\mathrm{Index}_{K_{4},\mathbb{F}_{2}}S(V_{-+}\oplus
V_{+-})=0$ \\
$\varepsilon _{1}\varepsilon _{2}\sigma $ acts trivially on $\{(x,-x)\in
V_{-+}\oplus V_{+-}\}$ & $\Rightarrow $ & $\mathrm{Index}_{K_{5},\mathbb{F}%
_{2}}S(V_{-+}\oplus V_{+-})=0$.%
\end{tabular}%
\end{equation*}%
The only nonzero element of $H^{2}(D_{8},\mathbb{F}_{2})$ satisfying all
requirements of commutativity with restrictions is~$w$. Hence,
\begin{equation}
\mathrm{Index}_{D_{8},\mathbb{F}_{2}}S(V_{-+}\oplus V_{+-})=\langle w\rangle.
\label{eq:Index=w}
\end{equation}

\begin{remark*}
The side information coming from this computation is that generators $%
\varepsilon _{1}$ and $\varepsilon _{2}$ of the group $H_{1}$ correspond to
generators $a$ and $a+b$ in the cohomology ring $H^{\ast }(H_{1},\mathbb{F}_{2})$.
This correspondence suggested the choice of generators in Lemma \ref%
{Lemma:CohomologyD8}(i).
\end{remark*}

\subsection{$\mathrm{Index}_{D_{8},\mathbb{F}_{2}}S(V_{--})=\langle y\rangle
$}

Again, $V_{--}$ is a concrete $D_{8}$-representation, and from Proposition %
\ref{Prop:IndexOfSphere}:%
\begin{equation*}
\mathrm{Index}_{H_{1},\mathbb{F}_{2}}S(V_{--})=\left\{
\begin{array}{cc}
\langle a+b\rangle \text{, } & \text{or} \\
\langle a+(a+b)\rangle \text{, } & \text{or} \\
\langle b+(a+b)\rangle . &
\end{array}%
\right.
\end{equation*}%
Again, we allow all three possibilities since we do not know the
correspondence between generators of $H_{1}$ and the chosen generators of $%
H^{\ast }(H_{q},\mathbb{F}_{2})$. Furthermore, since $K_{1}$ and $K_{2}$ act
nontrivially on $V_{--}$,
\begin{equation*}
\begin{array}{cc}
\mathrm{Index}_{K_{1},\mathbb{F}_{2}}S(V_{--})=\langle t_{1}\rangle , &
\mathrm{Index}_{K_{2},\mathbb{F}_{2}}S(V_{--})=\langle t_{2}\rangle%
\end{array}%
.
\end{equation*}%
On the other hand, $H_{3}$ acts trivially on $S(V_{--})$ and so%
\begin{equation*}
\mathrm{Index}_{H_{3},\mathbb{F}_{2}}S(V_{--})=0.
\end{equation*}%
By commutativity of the restriction diagram, or since the groups $K_{3}$, $%
K_{4}$ and $K_{5}$ act trivially on $V_{(1,1)}$, it follows that
\begin{equation*}
\mathrm{Index}_{K_{3},\mathbb{F}_{2}}S(V_{--})\ =\ \mathrm{Index}_{K_{4},%
\mathbb{F}_{2}}S(V_{--})\ =\ \mathrm{Index}_{K_{5},\mathbb{F}_{2}}S(V_{--})\
=\ 0.
\end{equation*}%
The only element satisfying the commutativity requirements is $y\in
H^{1}(D_{8},\mathbb{F}_{2})$, so%
\begin{equation}
\mathrm{Index}_{D_{8},\mathbb{F}_{2}}S(V_{--})=\langle y\rangle .
\label{eq:Index=y}
\end{equation}

\begin{remark*}
\label{Rm2.F_2-joinscheme}From the previous remark the fact that $\mathrm{Index}%
_{H_{1},\mathbb{F}_{2}}S(V_{--})=\langle b\rangle =\langle a+(a+b)\rangle $
follows directly. Alternatively, equation (\ref{eq:Index=y}) is a
consequence of (\ref{eq:IndexOfS0}) and (\ref{eq:diagramD8-1}).
\end{remark*}

\subsection{$\mathrm{Index}_{D_{8},\mathbb{F}_{2}}S(R_{4}^{\oplus
j})=\langle y^{j}w^{j}\rangle $}

From Proposition \ref{Prop:IndexOf Join} we get that
\begin{equation*}
\mathrm{Index}_{D_{8},\mathbb{F}_{2}}S(R_{4}^{\oplus j})=\mathrm{Index}%
_{D_{8},\mathbb{F}_{2}}S((V_{-+}\oplus V_{+-})^{\oplus j}\oplus
V_{--}^{\oplus j})=\langle y^{j}w^{j}\rangle .
\end{equation*}

\begin{remark*}
\label{Rem:Fail-1}In the same way we can compute that%
\begin{equation}
\mathrm{Index}_{D_{8},\mathbb{F}_{2}}(U_{2})=\langle x\rangle .
\label{eq:Index=x}
\end{equation}%
Therefore $\mathrm{Index}_{D_{8},\mathbb{F}_{2}}(U_{2}\oplus R_{4}^{\oplus
j})=0$. This means that on the join CS/TM scheme the Fadell--Husseini index
theory with $\mathbb{F}_{2}$ coefficients yields no obstruction to the existence of the
equivariant map in question.
\end{remark*}

\section{\label{Sec:IndexOfSpheresD8-Z}$\mathbf{Index}_{D_{8},\mathbb{Z}}\mathbf{S(R}_{4}^{\oplus j}\mathbf{)}$}

In this section we show that%
\begin{equation}
\mathrm{Index}_{D_{8},\mathbb{\mathbb{Z}}}S(R_{4}^{\oplus j})=\mathrm{Index}_{D_{8},\mathbb{\mathbb{Z}
}}^{3j+1}S(R_{4}^{\oplus j})=\left\{
\begin{array}{ll}
\langle \mathcal{Y}^{\frac{j}{2}}\mathcal{W}^{\frac{j}{2}}\rangle , & {\
\text{for }j\text{ even}} \\
\langle \mathcal{Y}^{\frac{j+1}{2}}\mathcal{W}^{\frac{j-1}{2}}\mathcal{M},%
\mathcal{Y}^{\frac{j+1}{2}}\mathcal{W}^{\frac{j+1}{2}}\rangle , & {\ \text{%
for }j\text{ odd.}}%
\end{array}%
\right.   \label{eq:IndexSpehere-Z}
\end{equation}

\subsection{The case when $j$ is even}

We give two proofs of the equation (\ref{eq:IndexSpehere-Z}) in the case
when $j$ is even.

\medskip

\noindent \emph{Method 1:} According to definition of the complex $D_{8}$
-representations $V_{+-}^{\mathbb{C}}\oplus V_{-+}^{\mathbb{C}}$ and $V_{--}^{\mathbb{C}}$,
in Section \ref{Sec:EvansView}, we have an isomorphism of real $D_{8}$%
-representations%
\begin{equation*}
R_{4}^{\oplus j}=\left( V_{-+}\oplus V_{+-}\right) ^{\oplus j}\oplus
V_{--}^{\oplus j}\cong \left( V_{+-}^{\mathbb{C}}\oplus V_{-+}^{\mathbb{C}}\right) ^{\oplus \frac{j}{2}}\oplus \left( V_{--}^{
\mathbb{C}}\right) ^{\oplus \frac{j}{2}}\text{.}
\end{equation*}%
Thus by Propositions \ref{Prop:IndexCherClasses} and \ref{Prop:IndexOf Join}%
, properties of Chern classes \cite[(5) page 286]{Atiyah} and equations (\ref%
{eq:1-chernClasses}) and (\ref{eq:2-chernClasses}) we have that
\begin{eqnarray*}
\mathrm{Index}_{D_{8},\mathbb{\mathbb{Z}}}S(R_{4}^{\oplus j}) &=&\langle \mathrm{c}_{\frac{3j}{2}}\left( \left(
V_{+-}^{\mathbb{C}}\oplus V_{-+}^{\mathbb{C}}\right) ^{\oplus \frac{j}{2}}\oplus \left( V_{--}^{
\mathbb{C}}\right) ^{\oplus \frac{j}{2}}\right) \rangle =\langle \mathrm{c}_{2}\left(
V_{+-}^{\mathbb{C}}\oplus V_{-+}^{\mathbb{C}}\right) ^{\frac{j}{2}}\cdot \mathrm{c}_{1}\left( V_{--}^{
\mathbb{C}}\right) ^{\frac{j}{2}}\rangle \\
&=&\langle \chi ^{\frac{j}{2}}\left( \xi +\chi _{1}\right) ^{\frac{j}{2}%
}\rangle .
\end{eqnarray*}%
The correspondence between Evens' and Bockstein spectral sequence views implies the statement.

\medskip

\noindent \emph{Method 2: }The group $D_{8}$ acts trivially on the
cohomology $H^{\ast }(S(R_{4}^{\oplus j}),\mathbb{Z})$. Then the $E_{2}$-term of the Serre spectral sequence associated to
$\mathrm{E}D_{8}\times _{D_{8}}S(R_{4}^{\oplus j})$ is a tensor product%
\begin{equation*}
E_{2}^{p,q}=H^{p}(D_{8},\mathbb{Z})\otimes H^{q}(S(R_{4}^{\oplus j}),\mathbb{Z}).
\end{equation*}
Since only $\partial _{3j,\mathbb{Z}}$ may be $\neq 0$, the multiplicative property of the spectral sequence
implies that
\begin{equation*}
\mathrm{Index}_{D_{8},\mathbb{Z}}S(R_{4}^{\oplus j})=\mathrm{Index}_{D_{8},\mathbb{Z}
}^{\dim V+1}S(R_{4}^{\oplus j})=\langle \partial _{3j,\mathbb{Z}}^{0,3j-1}(1\otimes l)\rangle
\end{equation*}%
where $l\in H^{3j-1}(S(R_{4}^{\oplus j}),\mathbb{Z})$ is a generator. The coefficient reduction morphism $c:\mathbb{Z}\rightarrow \mathbb{F}_{2}$ induces a morphism of Serre spectral sequences (Proposition \ref{Prop:Res-1}. E. 3) associated with the Borel construction
of the sphere $S(R_{4}^{\oplus j})$. Thus,
\begin{equation*}
s_{\ast }\left( \partial _{3j,\mathbb{\mathbb{Z}}}^{0,3j-1}(1\otimes l)\right) =\partial _{3j,\mathbb{F}_{2}}^{0,3j-1}\left(
s_{\ast }(1\otimes l)\right) \in H^{3j}(D_{8},\mathbb{F}_{2})
\end{equation*}%
and according to the result of the previous section%
\begin{equation*}
s_{\ast }\left( \partial _{3j,\mathbb{\mathbb{Z}}}^{0,3j-1}(1\otimes l)\right) =y^{j}w^{j}\text{.}
\end{equation*}%
Now, from the description of the map $c_{*}:H^{\ast }(D_{8},\mathbb{Z}%
)\longrightarrow H^{\ast }(D_{8},\mathbb{F}_{2})$ in (\ref{eq:Map-j})
follows the statement for $j$ even.

\subsection{The case when $j$ is odd}

The group $D_{8}$ acts nontrivially on the cohomology $H^{\ast
}(S(R_{4}^{\oplus j}),\mathbb{\mathbb{Z}})$. Precisely, the $D_{8}$-module $\mathcal{Z}=H^{3j-1}(S(R_{4}^{\oplus j}),%
\mathbb{\mathbb{Z}})$ is a nontrivial $D_{8}$-module and for $z\in \mathcal{Z}$:
\begin{equation*}
\varepsilon _{1}\cdot z=z\text{, \qquad }\varepsilon _{2}\cdot z=z\text{,
\qquad }\sigma \cdot z=-z.
\end{equation*}%
Then the $E_{2}$-term of the Serre spectral sequence associated to $\mathrm{E%
}D_{8}\times _{D_{8}}S(R_{4}^{\oplus j})$ is not a tensor product and
\begin{equation}
E_{2}^{p,q}=H^{p}(D_{8},H^{q}(S(R_{4}^{\oplus j}),\mathbb{\mathbb{Z}}))=\left\{
\begin{array}{ll}
H^{p}(D_{8},\mathbb{\mathbb{Z}}) & \text{, }q=0 \\
H^{p}(D_{8},\mathcal{Z}) & \text{, }q=3j-1 \\
0 & \text{, }q\neq 0,3j-1.
\end{array}%
\right.   \label{eq:SSS-Z}
\end{equation}%
To compute the index in this case we have to study the $H^{\ast }(D_{8},
\mathbb{Z})$-module structure of $H^{\ast }(D_{8},\mathcal{Z})$. Since the use of
LHS-spectral sequence, as in the case of field coefficients (Proposition %
\ref{Prop:Modulestructure}), cannot be of significant help we apply the
Bockstein spectral sequence associated with the following exact sequence of $%
D_{8}$-modules:%
\begin{equation}
0\rightarrow \mathcal{Z}\overset{\times 2}{\rightarrow }\mathcal{Z}%
\rightarrow \mathbb{F}_{2}\rightarrow 0 .  \label{eq:ExamSeqofModules-1}
\end{equation}

\begin{proposition}
\label{Prop:Modul_ZC}\qquad
\begin{compactenum}[\rm(A)]
\item $2\cdot H^{\ast }(D_{8},\mathcal{Z})=0$

\item $H^{\ast }(D_{8},\mathcal{Z})$ is generated as a $H^{\ast }(D_{8},%
\mathbb{\mathbb{Z}})$-module by three elements $\rho _{1}$, $\rho _{2}$, $\rho _{3}$ of degree
$1$, $2$, $3$ such that
\begin{equation*}
\rho _{1}\cdot \mathcal{Y}=0\text{, }\rho _{2}\cdot \mathcal{X}=0\text{, }%
\rho _{3}\cdot \mathcal{X}=0
\end{equation*}%
and
\begin{equation*}
c(\rho _{1})=x,~c(\rho _{2})=y^{2},~c(\rho _{3})=yw
\end{equation*}%
where $c$ is the map induced by the map $\mathcal{Z}\rightarrow \mathbb{F}%
_{2}$ from the exact sequence (\ref{eq:ExamSeqofModules-1}).

\end{compactenum}
\end{proposition}

\begin{proof}
The strategy of the proof is to consider four exact couples induced by the
exact sequence (\ref{eq:ExamSeqofModules-1}):
\begin{diagram}
H^{\ast }(D_{8},\mathcal{Z}) & &\rTo{\times 2}& & H^{\ast }(D_{8},\mathcal{Z}) & ~~& H^{\ast }(H_{1},\mathcal{Z}) & &\rTo{\times 2}& & H^{\ast }(H_{1},\mathcal{Z}) &~~&\\
& \luTo{\delta}&&\ldTo{c} && ~~& & \luTo{\delta}&&\ldTo{c} &&~~&\\
& & H^{\ast }(D_{8},\mathbb{F}_{2}) &&&  ~~& & & H^{\ast }(H_{1},\mathbb{F}_{2}) &&&~~&\\
\end{diagram}
\begin{diagram}
H^{\ast }(H_{2},\mathcal{Z}) & &\rTo{\times 2}& & H^{\ast }(H_{2},\mathcal{Z}) & ~~& H^{\ast }(K_{4},\mathcal{Z}) & &\rTo{\times 2}& & H^{\ast }(K_{4},\mathcal{Z}) &~~&\\
& \luTo{\delta}&&\ldTo{c} && ~~& & \luTo{\delta}&&\ldTo{c} &&~~&\\
& & H^{\ast }(H_{2},\mathbb{F}_{2}) &&&  ~~& & & H^{\ast }(K_{4},\mathbb{F}_{2}) &&&~~&\\
\end{diagram}
and the corresponding morphisms induced by $\mathrm{res}_{H_{1}}^{D_{8}}$, $%
\mathrm{res}_{H_{2}}^{D_{8}}$ and $\mathrm{res}_{K_{4}}^{D_{8}}$. Our
notation is as in the restriction diagram (\ref{eq:diagramD8-1}).

\begin{compactenum}
\item The module $\mathcal{Z}$ as a $H_{1}$-module is a trivial module.
Therefore in the $H_{1}$ exact couple $d_{1}$ is the usual Bockstein
homomorphism and so%
\begin{equation*}
d_{1}(a)=\mathrm{Sq}^{1}(a)=a^{2},~\qquad d_{1}(b)=\mathrm{Sq}^{1}(b)=b^{2}.
\end{equation*}%
Thus from the restriction homomorphism $\mathrm{res}_{H_{1}}^{D_{8}}$ we
have:%
\begin{equation}
\begin{array}{lll}
\mathrm{res}_{H_{1}}^{D_{8}}\left( d_{1}(1)\right) =d_{1}\left( \mathrm{res}%
_{H_{1}}^{D_{8}}\left( 1\right) \right) =d_{1}(1)=0 & \Rightarrow &
d_{1}(1)\in \ker \left( \mathrm{res}_{H_{1}}^{D_{8}}\right) , \\
\mathrm{res}_{H_{1}}^{D_{8}}\left( d_{1}(x)\right) =d_{1}\left( \mathrm{res}%
_{H_{1}}^{D_{8}}\left( x\right) \right) =d_{1}(0)=0 & \Rightarrow &
d_{1}(x)\in \ker \left( \mathrm{res}_{H_{1}}^{D_{8}}\right) , \\
\mathrm{res}_{H_{1}}^{D_{8}}\left( d_{1}(y)\right) =d_{1}\left( \mathrm{res}%
_{H_{1}}^{D_{8}}\left( y\right) \right) =d_{1}(b)=b^{2} & \Rightarrow &
d_{1}(y)\in y^{2}+\ker \left( \mathrm{res}_{H_{1}}^{D_{8}}\right) , \\
\mathrm{res}_{H_{1}}^{D_{8}}\left( d_{1}(w)\right) =d_{1}\left( \mathrm{res}%
_{H_{1}}^{D_{8}}\left( w\right) \right) =ba(a+b) & \Rightarrow & d_{1}(w)\in
yw+\ker \left( \mathrm{res}_{H_{1}}^{D_{8}}\right) .%
\end{array}
\label{proof:rel-1}
\end{equation}

\item The module $\mathcal{Z}$ as a $H_{2}$-module is a non-trivial module.
The $H_{2}\cong \mathbb{Z}_{4}$ exact couple unrolls into a long exact
sequence \cite[Proposition 6.1, page 71]{Brown}
\begin{diagram}
0 & \rTo &
\begin{array}{c}
H^{0}(\mathbb{Z}_{4},\mathcal{Z}) \\
0%
\end{array}
& \rTo{\times 2} &
\begin{array}{c}
H^{0}(\mathbb{Z}_{4},\mathcal{Z}) \\
0%
\end{array}
& \rTo &
\begin{array}{c}
H^{0}(\mathbb{Z}_{4},\mathbb{F}_{2}) \\
\begin{array}{cc}
\mathbb{F}_{2} & ,~1%
\end{array}%
\end{array}
\\
& \rTo^{\delta } &
\begin{array}{c}
H^{1}(\mathbb{Z}_{4},\mathcal{Z}) \\
\begin{array}{cc}
\mathbb{F}_{2} & ,\lambda%
\end{array}%
\end{array}
& \rTo{\times 2} &
\begin{array}{c}
H^{1}(\mathbb{Z}_{4},\mathcal{Z}) \\
\begin{array}{cc}
\mathbb{F}_{2} & ,\lambda%
\end{array}%
\end{array}
& \rTo &
\begin{array}{c}
H^{1}(\mathbb{Z}_{4},\mathbb{F}_{2}) \\
\begin{array}{cc}
\mathbb{F}_{2} & ,~e%
\end{array}%
\end{array}
\\
& \rTo^{\delta } &
\begin{array}{c}
H^{2}(\mathbb{Z}_{4},\mathcal{Z}) \\
0
\end{array}
& \rTo{\times 2} &
\begin{array}{c}
H^{2}(\mathbb{Z}_{4},\mathcal{Z}) \\
0
\end{array}
& \rTo&
\begin{array}{c}
H^{2}(\mathbb{Z}_{4},\mathbb{F}_{2}) \\
\begin{array}{cc}
\mathbb{F}_{2} & ,~u%
\end{array}%
\end{array}
\\
& \rTo^{\delta } &
\begin{array}{c}
H^{3}(\mathbb{Z}_{4},\mathcal{Z}) \\
\begin{array}{cc}
\mathbb{F}_{2} & ,\lambda U%
\end{array}%
\end{array}
& \rTo{\times 2} &
\begin{array}{c}
H^{3}(\mathbb{Z}_{4},\mathcal{Z}) \\
\begin{array}{cc}
\mathbb{F}_{2} & ,\lambda U%
\end{array}%
\end{array}
& \rTo &
\begin{array}{c}
H^{3}(\mathbb{Z}_{4},\mathbb{F}_{2}) \\
\begin{array}{cc}
\mathbb{F}_{2} & ,~eu%
\end{array}%
\end{array}
\\
& \rTo^{\delta } &
\begin{array}{c}
H^{4}(\mathbb{Z}_{4},\mathcal{Z}) \\
0
\end{array}
&  & \ldots &  &
\end{diagram}
Here we have used the facts that $H^{i}(\mathbb{Z}_{4},\mathcal{Z})=\left\{
\begin{array}{ll}
\mathbb{F}_{2} & { \text{, }i\text{ odd}} \\
0 & { \text{, }i\text{ even}}%
\end{array}%
\right. $ and that multiplication by $U\in H^{2}(\mathbb{Z}_{4},\mathbb{Z})$
in $H^{\ast }(\mathbb{Z}_{4},\mathcal{Z})$ is an isomorphism \cite[Section
XII. 7. pages 250-253]{CaEi}. The long exact sequence describes the boundary
operator:
\begin{equation*}
\delta (1)=\lambda \text{, \quad }\delta (e)=0\text{, \quad }\delta
(u)=\lambda U
\end{equation*}%
and consequently the first differential:%
\begin{equation*}
d_{1}(1)=e\text{, \quad }d_{1}(e)=0\text{, \quad }d_{1}(u)=eu.
\end{equation*}%
The restriction homomorphism $\mathrm{res}_{H_{2}}^{D_{8}}$ implies that:%
\begin{equation}
\begin{array}{lll}
\mathrm{res}_{H_{2}}^{D_{8}}\left( d_{1}(1)\right) =d_{1}\left( \mathrm{res}%
_{H_{1}}^{D_{8}}\left( 1\right) \right) =d_{1}(1)=e & \Rightarrow &
d_{1}(1)\in x+\ker \left( \mathrm{res}_{H_{2}}^{D_{8}}\right) , \\
\mathrm{res}_{H_{1}}^{D_{8}}\left( d_{1}(x)\right) =d_{1}\left( \mathrm{res}%
_{H_{1}}^{D_{8}}\left( x\right) \right) =d_{1}(e)=0 & \Rightarrow &
d_{1}(x)\in \ker \left( \mathrm{res}_{H_{2}}^{D_{8}}\right) , \\
\mathrm{res}_{H_{1}}^{D_{8}}\left( d_{1}(y)\right) =d_{1}\left( \mathrm{res}%
_{H_{1}}^{D_{8}}\left( y\right) \right) =d_{1}(e)=0 & \Rightarrow &
d_{1}(y)\in \ker \left( \mathrm{res}_{H_{2}}^{D_{8}}\right) , \\
\mathrm{res}_{H_{1}}^{D_{8}}\left( d_{1}(w)\right) =d_{1}\left( \mathrm{res}%
_{H_{1}}^{D_{8}}\left( w\right) \right) =d_{1}(u)=eu & \Rightarrow &
d_{1}(w)\in yw+\ker \left( \mathrm{res}_{H_{2}}^{D_{8}}\right) .%
\end{array}
\label{proof:rel-2}
\end{equation}

\item The module $\mathcal{Z}$ as a $K_{4}$-module is a non-trivial module.
Then the $K_{4}\cong \mathbb{Z}_{2}$ exact couple unrolls into%
\begin{equation}
\begin{diagram}
0 & \rTo &
\begin{array}{c}
H^{0}(\mathbb{Z}_{2},\mathcal{Z}) \\
0%
\end{array}
& \rTo{\times 2} &
\begin{array}{c}
H^{0}(\mathbb{Z}_{2},\mathcal{Z}) \\
0%
\end{array}
& \rTo &
\begin{array}{c}
H^{0}(\mathbb{Z}_{2},\mathbb{F}_{2}) \\
\begin{array}{cc}
\mathbb{F}_{2} & ,~1%
\end{array}%
\end{array}
\\
& \rTo^{\delta } &
\begin{array}{c}
H^{1}(\mathbb{Z}_{2},\mathcal{Z}) \\
\begin{array}{cc}
\mathbb{F}_{2} & ,\varphi%
\end{array}%
\end{array}
& \rTo{\times 2} &
\begin{array}{c}
H^{1}(\mathbb{Z}_{2},\mathcal{Z}) \\
\begin{array}{cc}
\mathbb{F}_{2} & ,\varphi%
\end{array}%
\end{array}
& \rTo &
\begin{array}{c}
H^{1}(\mathbb{Z}_{2},\mathbb{F}_{2}) \\
\begin{array}{cc}
\mathbb{F}_{2} & ,~t_4%
\end{array}%
\end{array}
\\
& \rTo^{\delta } &
\begin{array}{c}
H^{2}(\mathbb{Z}_{2},\mathcal{Z}) \\
0
\end{array}
& \rTo{\times 2} &
\begin{array}{c}
H^{2}(\mathbb{Z}_{2},\mathcal{Z}) \\
0
\end{array}
& \rTo&
\begin{array}{c}
H^{2}(\mathbb{Z}_{2},\mathbb{F}_{2}) \\
\begin{array}{cc}
\mathbb{F}_{2} & ,~t_4^{2}%
\end{array}%
\end{array}
\\
& \rTo^{\delta } &
\begin{array}{c}
H^{3}(\mathbb{Z}_{2},\mathcal{Z}) \\
\begin{array}{cc}
\mathbb{F}_{2} & ,\varphi T%
\end{array}%
\end{array}
& \rTo{\times 2} &
\begin{array}{c}
H^{3}(\mathbb{Z}_{2},\mathcal{Z}) \\
\begin{array}{cc}
\mathbb{F}_{2} & ,\varphi T%
\end{array}%
\end{array}
& \rTo &
\begin{array}{c}
H^{3}(\mathbb{Z}_{2},\mathbb{F}_{2}) \\
\begin{array}{cc}
\mathbb{F}_{2} & ,~t_4^{3}%
\end{array}%
\end{array}
\\
& \rTo^{\delta } &
\begin{array}{c}
H^{4}(\mathbb{Z}_{2},\mathcal{Z}) \\
0
\end{array}
&  & \ldots &  &
\end{diagram}
\label{proof:rel-2.1}
\end{equation}
Similarly, $H^{i}(\mathbb{Z}_{2},\mathcal{Z})=\left\{
\begin{array}{ll}
\mathbb{F}_{2} & { \text{, }i\text{ odd}} \\
0 & { \text{, }i\text{ even}}%
\end{array}%
\right. $ and multiplication by $T\in H^{2}(\mathbb{Z}_{2},\mathbb{Z})$ in $%
H^{\ast }(\mathbb{Z}_{2},\mathcal{Z})$ is an isomorphism \cite[Section XII.
7. pages 250-253]{CaEi}. Then%
\begin{equation*}
d_{1}(1)=t_{4},~\qquad d_{1}(t_{4}^{2i+1})=0,~\qquad
d_{1}(t_{4}^{2i})=t_{4}^{2i+1}
\end{equation*}%
for $i\geq 0$. This implies that%
\begin{equation}
\mathrm{res}_{K_{4}}^{D_{8}}\left( d_{1}(w)\right) =d_{1}\left( \mathrm{res}%
_{H_{1}}^{D_{8}}\left( w\right) \right) =d_{1}(0)=0  \label{proof:rel-3}
\end{equation}%
and%
\begin{equation}
\mathrm{res}_{K_{4}}^{D_{8}}\left( d_{1}(y)\right) =d_{1}\left( \mathrm{res}%
_{H_{1}}^{D_{8}}\left( y\right) \right) =d_{1}(0)=0.  \label{proof:rel-3-1}
\end{equation}
\end{compactenum}

\noindent From (\ref{proof:rel-1}), (\ref{proof:rel-2}) and the restriction
diagram (\ref{eq:diagramD8-1}) follows:%
\begin{equation*}
d_{1}(1)=x,~\qquad d_{1}(x)=0,
\end{equation*}%
and%
\begin{equation}
d_{1}(w)\in \{yw,~yw+x^{3}\}\text{ \quad and \quad }d_{1}(y)\in
\{y^{2},y^{2}+x^{2}\}\text{.}  \label{proof:rel-4}
\end{equation}%
Since $\mathrm{res}_{K_{4}}^{D_{8}}\left( yw\right) =0$, $\mathrm{res}%
_{K_{4}}^{D_{8}}\left( yw+x^{3}\right) =t_{4}^{3}\neq 0$ and $\mathrm{res}%
_{K_{4}}^{D_{8}}\left( y^{2}\right) =0$, $\mathrm{res}_{K_{4}}^{D_{8}}\left(
y^{2}+x^{2}\right) =t_{4}^{2}\neq 0$, then the equations (\ref{proof:rel-3})
and (\ref{proof:rel-3-1}) resolve the final dilemmas (\ref{proof:rel-4}).
Thus $d_{1}(w)=yw$. According to \cite[Remark 5.7.4, page 108]{Carlson}
there are elements $\rho _{1}$, $\rho _{2}$, $\rho _{3}$ of degree $1$, $2$,
$3$ and of exponent $2$ in $H^{\ast }(D_{8},\mathcal{Z})$ satisfying
property (B) of this proposition.

The property (A) follows from the properties of Bockstein spectral sequence
and the fact that the derived couple of the $D_{8}$ exact couple is:
\begin{diagram}[size=2em,textflow]
0&&\rTo&&0\\
&\luTo{\delta _1}&&\ldTo{c_1}&\\
&&\mathbb{F}_{2}&&\\
\end{diagram}

\noindent where $\mathbb{F}_{2}$ appears in dimension $0$.
\end{proof}

\begin{remark*}
The proposition does not describes the complete $H^{*}(D_8,\mathbb{Z})$-modulo structure on $H^{*}(D_8,\mathcal{Z})$. It gives only the necessary information for the computation of $\mathrm{Index}_{D_{8},\mathbb{\mathbb{Z}}}S(R_{4}^{\oplus j})$. The complete result can be found
in \cite[Theorem 5.11.(a)]{Hand}.
\end{remark*}
\medskip

\noindent Thus, the index is given by%
\begin{equation*}
\mathrm{Index}_{D_{8},\mathbb{\mathbb{Z}}}S(R_{4}^{\oplus j})=\langle \partial _{3j,\mathbb{
\mathbb{Z}}}^{1,3j-1}(\rho _{1}),\partial _{3j,\mathbb{\mathbb{Z}}}^{2,3j-1}(\rho _{2}),\partial _{3j,\mathbb{
\mathbb{Z}}}^{3,3j-1}(\rho _{3})\rangle .
\end{equation*}%
The morphism $C$ from spectral sequence (\ref{eq:SSS-Z}) to spectral
sequence (\ref{eq:SSS-F2}) induced by the reduction homomorphism $\mathbb{Z\rightarrow F}_{2}$ implies that:%
\begin{equation}
\begin{array}{cccccccc}
C(\partial _{3j,\mathbb{\mathbb{Z}}}^{1,3j-1}(\rho _{1})) & = & \partial _{3j,\mathbb{F}_{2}}^{1,3j-1}(c_{*}(\rho
_{1})) & = & \partial _{3j,\mathbb{F}_{2}}^{1,3j-1}(x) & = & 0 &  \\
C(\partial _{3j,\mathbb{\mathbb{Z}}}^{2,3j-1}(\rho _{2})) & = & \partial _{3j,\mathbb{F}_{2}}^{2,3j-1}(c_{*}(\rho
_{2})) & = & \partial _{3j,\mathbb{F}_{2}}^{2,3j-1}(y^{2}) & = & y^{j+2}w^{j}
& =y^{j+1}w^{j-1}(y+x)w \\
C(\partial _{3j,\mathbb{Z}}^{3,3j-1}(\rho _{3})) & = & \partial _{3j,\mathbb{F}_{2}}^{3,3j-1}(c_{*}(\rho
_{3})) & = & \partial _{3j,\mathbb{F}_{2}}^{3,3j-1}(yw) & = & y^{j+1}w^{j+1}
&
\end{array}
\label{proof:rel-5}
\end{equation}

\medskip

\noindent The sequence of $D_{8}$ inclusion maps
\begin{equation*}
S(R_{4}^{\oplus (j-1)})\subset S(R_{4}^{\oplus j})\subset S(R_{4}^{\oplus
(j+1)})
\end{equation*}%
provides (Proposition \ref{prop:basic}) a sequence of inclusions:%
\begin{equation}
\langle \mathcal{Y}^{\frac{j-1}{2}}\mathcal{W}^{\frac{j-1}{2}}\rangle =%
\mathrm{Index}_{D_{8},\mathbb{Z}}S(R_{4}^{\oplus (j-1)})\supseteq \mathrm{Index}_{D_{8},
\mathbb{Z}}S(R_{4}^{\oplus j})\supseteq \mathrm{Index}_{D_{8},\mathbb{Z}}S(R_{4}^{\oplus (j+1)})=\langle \mathcal{Y}^{\frac{j+1}{2}}\mathcal{W}^{
\frac{j+1}{2}}\rangle .  \label{proof:rel-6}
\end{equation}

\medskip

\noindent The relations (\ref{proof:rel-5}), (\ref{proof:rel-6}) and (\ref%
{eq:Map-j}), along with Proposition \ref{Prop:Modul_ZC} imply that for $j$
odd:

\begin{equation*}
\mathrm{Index}_{D_{8},\mathbb{\mathbb{Z}}}S(R_{4}^{\oplus j})=\langle \mathcal{Y}^{\frac{j+1}{2}}\mathcal{W}^{\frac{%
j-1}{2}}\mathcal{M},\mathcal{Y}^{\frac{j+1}{2}}\mathcal{W}^{\frac{j+1}{2}%
}\rangle .
\end{equation*}

\begin{remark*}
\label{Rem:Fail-2} The index $\mathrm{Index}_{D_{8},\mathbb{\mathbb{Z}
}}S(U_{k}\times R_{4}^{\oplus j})$ appearing in the join test map scheme
can now be computed. From Example \ref{Ex:IndexS0} and the restriction
diagram (\ref{eq:diagramD8-1-Z}) it follows that%
\begin{equation*}
\mathrm{Index}_{D_{8},\mathbb{
\mathbb{Z}}}S(U_{k})=\mathrm{Index}_{D_{8},\mathbb{
\mathbb{Z}}}D_{8}/H_{1}=\ker \left( \mathrm{res}_{H_{1}}^{D_{8}}:H^{\ast }(D_{8},
\mathbb{Z})\rightarrow H^{\ast }(H_{1},\mathbb{Z})\right) =\langle \mathcal{X}\rangle .
\end{equation*}%
The inclusions
\begin{equation*}
\mathrm{Index}_{D_{8},\mathbb{\mathbb{Z}}}S(U_{k}\times R_{4}^{\oplus j})\subseteq \mathrm{Index}_{D_{8},\mathbb{
\mathbb{Z}}}S(R_{4}^{\oplus j})\text{ \quad and \quad }\mathrm{Index}_{D_{8},\mathbb{
\mathbb{Z}}}S(U_{k}\times R_{4}^{\oplus j})\subseteq \mathrm{Index}_{D_{8},\mathbb{\mathbb{Z}}}S(U_{k})
\end{equation*}%
imply that
\begin{equation*}
\mathrm{Index}_{D_{8},\mathbb{\mathbb{Z}}}S(U_{k}\times R_{4}^{\oplus j})\subseteq \mathrm{Index}_{D_{8},\mathbb{
\mathbb{Z}}}S(R_{4}^{\oplus j})\cap \mathrm{Index}_{D_{8},\mathbb{\mathbb{Z}}}S(U_{k})=\{0\}.
\end{equation*}%
Thus, as in the case of $\mathbb{F}_{2}$ coefficients, the Fadell--Husseini
index theory with $\mathbb{Z}$ coefficients on the join CS/TM scheme does
not lead to any obstruction to the existence of the
equivariant map in question.
\end{remark*}

\section{\label{Sec:IndexProductD8}$\mathbf{Index}_{D_{8},\mathbb{F}_{2}}%
\mathbf{S}^{d}\mathbf{\times S}^{d}$}

This section is devoted to the proof of the equality%
\begin{equation}
\mathrm{Index}_{D_{8},\mathbb{F}_{2}}S^{d}\times S^{d}=\langle \pi
_{d+1},\pi _{d+2},w^{d+1}\rangle .  \label{eq:IndexSxS-F2}
\end{equation}%
The index will be determined by the explicit computation of the Serre
spectral sequence associated with the Borel construction%
\begin{equation*}
S^{d}\times S^{d}\rightarrow \mathrm{E}D_{8}\times _{D_{8}}\left(
S^{d}\times S^{d}\right) \rightarrow \mathrm{B}D_{8}.
\end{equation*}%
The group $D_{8}$ acts nontrivially on the cohomology of the fibre, and
therefore the spectral sequence has nontrivial local coefficients. The $%
E_{2} $-term is given by
\begin{equation}
\begin{array}{lll}
E_{2}^{p,q} & = & H^{p}(\mathrm{B}D_{8},\mathcal{H}^{q}(S^{d}\times S^{d},%
\mathbb{F}_{2}))=H^{p}(D_{8},H^{q}(S^{d}\times S^{d},\mathbb{F}_{2})) \\
&  &  \\
& = & \left\{
\begin{array}{lllll}
H^{p}(D_{8},\mathbb{F}_{2}) & \text{, }q=0,2d &  &  &  \\
H^{p}(D_{8},\mathbb{F}_{2}[D_{8}/H_{1}]) & \text{, }q=d &  &  &  \\
0 & \text{, }q\neq 0,d,2d. &  &  &
\end{array}%
\right. %
\end{array}
\label{eq:E2-Product}
\end{equation}%
The nontriviality of local coefficients appears in the $d$-th row of the
spectral sequence.

\medskip

In Section \ref{Sec:Alternative} there is a sketch of an alternative proof
of the fact (\ref{eq:IndexSxS-F2}) suggested by a referee for an earlier, $%
\mathbb{F}_{2}$-coefficient, version of the paper.

\subsection{\label{Sec:Row}The $d$-th row as an $H^{\ast }(D_{8},\mathbb{F}%
_{2})$-module}

Since the spectral sequence is an $H^{\ast}(D_{8},\mathbb{F}_{2})$-module
and the differentials are module maps we need to understand the $%
H^{\ast}(D_{8},\mathbb{F}_{2})$-module structure of the $E_{2}$-term.

\begin{proposition}
$H^{\ast }(D_{8},\mathbb{F}_{2}[D_{8}/H_{1}])\ \cong _{\mathrm{ring}}\
H^{\ast }(H_{1},\mathbb{F}_{2})$.
\end{proposition}

\begin{proof}
Here $H_{1}=\langle \varepsilon _{1},\varepsilon _{2}\rangle \cong
\mathbb{Z}_{2}\times\mathbb{Z}_{2}$ is a maximal (normal) subgroup of index $2$ in $D_{8}$.

\noindent \emph{Method 1:} The statement follows from Shapiro's lemma \cite[%
Proposition 6.2, page 73]{Brown} and the fact that when $[G:H]<\infty $,
then there is an isomorphism of $G$-modules \textrm{Coind}$_{H}^{G}M\cong ~$%
\textrm{Ind}$_{H}^{G}M$.

\noindent \emph{Method 2:} There is an exact sequence of groups%
\begin{equation*}
1\rightarrow H_{1}\rightarrow D_{8}\rightarrow D_{8}/H_{1}\rightarrow 1.
\end{equation*}%
The associated LHS spectral sequence \cite[Corollary 1.2, page 116]%
{Adem-Milgram} has the $E_{2}$-term:%
\begin{eqnarray*}
A_{2}^{p,q} &=&H^{p}(D_{8}/H_{1},H^{q}(H_{1},\mathbb{F}_{2}[D_{8}/H_{1}])) \\
&\cong &H^{p}(\mathbb{Z}_{2},H^{q}((\mathbb{Z}
_{2})^{2},\mathbb{F}_{2}\oplus \mathbb{F}_{2})) \\
&\cong &H^{p}(\mathbb{Z}_{2};H^{q}((\mathbb{Z}_{2})^{2},\mathbb{F}_{2})\oplus H^{q}((
\mathbb{Z}_{2})^{2},\mathbb{F}_{2})).
\end{eqnarray*}%
The action of the group $D_{8}/H_{1}\cong\mathbb{Z}_{2}$ on the sum is given by the conjugation action of $G$ on the pair $%
(H_{1},H^{q}(H_{1},\mathbb{F}_{2}[D_{8}/H_{1}]))$ \cite[Corollary 8.4, page 80]{Brown}. Since $\mathbb{F}_{2}[\mathbb{Z}
_{2}]$ is a free $\mathbb{Z}_{2}$-module
\begin{equation*}
H^{0}(\mathbb{Z}_{2};\mathbb{F}_{2}[\mathbb{Z}_{2}])=(\mathbb{F}_{2}[
\mathbb{Z}_{2}])^{\mathbb{Z}_{2}}=\mathbb{F}_{2}
\end{equation*}%
and $H^{p}(\mathbb{Z}_{2};\mathbb{F}_{2}[\mathbb{Z}_{2}])=0$ for $p>0$. Thus
\begin{eqnarray*}
A_{2}^{p,q} &\cong &H^{p}(D_{8}/H_{1};H^{q}((
\mathbb{Z}_{2})^{2},\mathbb{F}_{2})\oplus H^{q}((
\mathbb{Z}_{2})^{2},\mathbb{F}_{2})) \\
&\cong &H^{p}(D_{8}/H_{1};\mathbb{F}_{2}[
\mathbb{Z}_{2}]^{q+1}) \\
&\cong &H^{p}(D_{8}/H_{1};\mathbb{F}_{2}[\mathbb{Z}_{2}])^{q+1}\cong \left\{
\begin{array}{cc}
\left( H^{p}(\mathbb{Z}_{2};\mathbb{F}_{2}[\mathbb{Z}_{2}])^{q+1}\right) ^{
\mathbb{Z}_{2}}\cong \mathbb{F}_{2}^{q+1} & \text{, }p=0 \\
0 & \text{, }p>0.%
\end{array}%
\right.
\end{eqnarray*}%
Thus the $E_{2}$-term has the shape as in Figure \ref{Fig-2} (concentrated
in the $0$-column) and collapses.
\end{proof}

\begin{figure}[tbh]
\centering\includegraphics[scale=0.80]{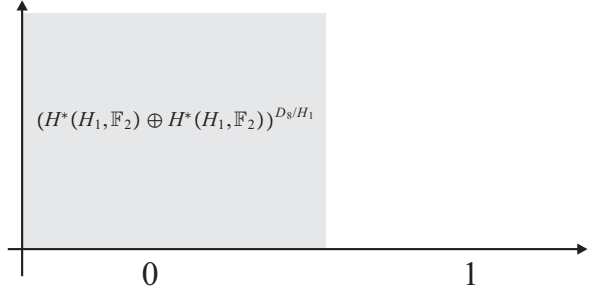}
\caption{{The $A_{2}$-term of the LHS spectral sequence}}
\label{Fig-2}
\end{figure}

The first information about the $H^{\ast }(D_{8},\mathbb{F}_{2})$-module
structure on $H^{\ast }(D_{8},\mathbb{F}_{2}[D_{8}/H_{1}])$, as well as the
method for revealing the complete structure, comes from the following
proposition.

\begin{proposition}
\label{Prop:MultiplyByX} We have $x\cdot H^{\ast }(D_{8},\mathbb{F}%
_{2}[D_{8}/H_{1}])=0$ for the nonzero element $x\in H^{1}(D_{8},\mathbb{F}_{2})$
that is characterized by $\mathrm{res}_{H_{1}}^{D_{8}}(x)=0$.
\end{proposition}

\begin{proof}
~

\noindent \emph{Method 1:} The isomorphism $H^{\ast}(D_{8},\mathbb{F}%
_{2}[D_{8}/H_{1}])\cong_{\text{{ring}}} H^{\ast}(H_{1},\mathbb{F}_{2})$
induced by Shapiro's lemma \cite[Proposition 6.2, page 73]{Brown} carries
the $H^{\ast }(D_{8},\mathbb{F}_{2})$-module structure to $H^{\ast }(H_{1},%
\mathbb{F}_{2})$ via the restriction homomorphism $\mathrm{res}%
_{H_{1}}^{D_{8}}:H^{\ast }(D_{8},\mathbb{F}_{2})\rightarrow H^{\ast }(H_{1},%
\mathbb{F}_{2})$. In this way the complete $H^{\ast }(D_{8},\mathbb{F}_{2})$%
-module structure is given on $H^{\ast }(D_{8},\mathbb{F}_{2}[D_{8}/H_{1}])$%
. In particular, since $\mathrm{res}_{H_{1}}^{D_{8}}(x)=0$, the proposition
is proved.

\noindent \emph{Method 2:} The exact sequence of groups%
\begin{equation*}
1\rightarrow H_{1}\rightarrow D_{8}\rightarrow D_{8}/H_{1}\rightarrow 1
\end{equation*}%
induces two LHS spectral sequences%
\begin{equation}
A_{2}^{p,q}=H^{p}\left( D_{8}/H_{1},H^{q}\left( H_{1},\mathbb{F}%
_{2}[D_{8}/H_{1}]\right) \right) ~\Longrightarrow ~H^{p+q}(D_{8},\mathbb{F}%
_{2}[D_{8}/H_{1}]),  \label{sq-1}
\end{equation}%
\begin{equation}
B_{2}^{p,q}=H^{p}\left( D_{8}/H_{1},H^{q}\left( H_{1},\mathbb{F}_{2}\right)
\right)~\hphantom{[D_{8}/H_{1}]}\Longrightarrow ~H^{p+q}(D_{8},\mathbb{F}%
_{2}).\hphantom{[D_{8}/H_{1}]}  \label{sq-2}
\end{equation}%
The spectral sequence (\ref{sq-2}) acts on the spectral sequence (\ref{sq-1})%
\begin{equation*}
B_{t}^{r,s}\times A_{t}^{u,v}\rightarrow A_{t}^{u+r,v+s}
\end{equation*}%
In the $E_{\infty}$-term this action becomes an action of $H^{\ast }(D_{8},%
\mathbb{F}_{2})$ on $H^{\ast }(D_{8},\mathbb{F}_{2}[D_{8}/H_{1}])$. Since we
already discussed both spectral sequences we know that%
\begin{equation*}
A_{2}^{p,q}=A_{\infty }^{p,q}\text{ ~~~and~~~ }B_{2}^{p,q}=B_{\infty }^{p,q}%
\text{.}
\end{equation*}%
From Figures \ref{Fig-1} and \ref{Fig-2} it is apparent that $x\in
B_{2}^{1,0}=B_{\infty }^{1,0}$ acts by $x\cdot A_{2}^{p,q}=0$ for every $p$ and $q$.
\end{proof}

\begin{corollary}
\label{Cor:Index-1}$\mathrm{Index}_{D_{8},\mathbb{F}_{2}}^{d+2}S^{d}\times
S^{d}=\mathrm{im}\big(\partial _{d+1}:E_{d+1}^{\ast ,d}\rightarrow
E_{d+1}^{\ast +d+1,0}\big)\subseteq y\cdot H^{\ast }(D_{8},\mathbb{F}_{2}).$
\end{corollary}

\begin{proof}
Let $\alpha \in E_{d+1}^{\ast ,d}$ and $\partial _{d+1}(\alpha )\notin
y\cdot H^{\ast }(D_{8},\mathbb{F}_{2})$. Then $x\cdot \partial _{d+1}(\alpha
)\neq 0$. Since $\partial _{d+1}$ is a $H^{\ast }(D_{8},\mathbb{F}_{2})$%
-module map and $x$ acts trivially on $H^{\ast }(D_{8},\mathbb{F}%
_{2}[D_{8}/H_{1}])$, as indicated by Proposition \ref{Prop:MultiplyByX},
there is a contradiction%
\begin{equation*}
0=\partial _{d+1}(x\cdot \alpha )=x\cdot \partial _{d+1}(\alpha )\neq 0.
\end{equation*}
\end{proof}

\begin{proposition}
\label{Prop:Modulestructure}$H^{\ast }(D_{8},\mathbb{F}_{2}[D_{8}/H_{1}])$
is generated as an $H^{\ast }(D_{8},\mathbb{F}_{2})$-module by
\begin{equation*}
H^{0}(D_{8},\mathbb{F}_{2}[D_{8}/H_{1}])\quad\text{and}\quad H^{1}(D_{8},%
\mathbb{F}_{2}[D_{8}/H_{1}]).
\end{equation*}
\end{proposition}

\begin{proof}
~

\noindent \emph{Method 1:} We already observed that Shapiro's lemma $H^{\ast
}(D_{8},\mathbb{F}_{2}[D_{8}/H_{1}])\cong _{\text{{ring}}}H^{\ast }(H_{1},%
\mathbb{F}_{2})$ carries the $H^{\ast }(D_{8},\mathbb{F}_{2})$-module
structure to $H^{\ast }(H_{1},\mathbb{F}_{2}) $ via the restriction
homomorphism $\mathrm{res}_{H_{1}}^{D_{8}}:H^{\ast }(D_{8},\mathbb{F}%
_{2})\rightarrow H^{\ast }(H_{1},\mathbb{F}_{2})$. Thus $H^{\ast }(H_{1},%
\mathbb{F}_{2})$ as an $H^{\ast }(D_{8},\mathbb{F}_{2}) $-module is
generated by $1\in H^{0}(H_{1},\mathbb{F}_{2})$ together with $a\in
H^{1}(H_{1},\mathbb{F}_{2})$.

\noindent \emph{Method 2:} There is the exact sequence of $D_{8}$-modules%
\begin{equation}
0\rightarrow \mathbb{F}_{2}\rightarrow \mathbb{F}_{2}[D_{8}/H_{1}]%
\rightarrow \mathbb{F}_{2}\rightarrow 0,  \label{sq-3}
\end{equation}%
where the left and right modules $\mathbb{F}_{2}$ are trivial $D_{8}$%
-modules. The first map is a diagonal inclusion while the second one is a
quotient map. The sequence (\ref{sq-3}) induces a long exact sequence on
group cohomology \cite[Proposition 6.1, page 71]{Brown},
\begin{equation}
\begin{array}{cl}
0\rightarrow & H^{0}\left( D_{8},\mathbb{F}_{2}\right) \overset{i_{0}}{%
\rightarrow }H^{0}\left( D_{8},\mathbb{F}_{2}[D_{8}/H_{1}]\right) \overset{%
q_{o}}{\rightarrow }H^{0}\left( D_{8},\mathbb{F}_{2}\right) \overset{\delta
_{0}}{\rightarrow } \\
& \qquad H^{1}\left( D_{8},\mathbb{F}_{2}\right) \overset{i_{1}}{\rightarrow
}H^{1}\left( D_{8},\mathbb{F}_{2}[D_{8}/H_{1}]\right) \overset{q_{1}}{%
\rightarrow }H^{1}\left( D_{8},\mathbb{F}_{2}\right) \overset{\delta _{1}}{%
\rightarrow}\ ~\ldots%
\end{array}
\label{sq-4}
\end{equation}%
From the exact sequence (\ref{sq-4}), compatibility of the cup product \cite[%
page 110, (3.3)]{Brown} and Proposition \ref{Prop:MultiplyByX} one can
deduce that $\delta _{0}(1)=x$. Then by chasing along sequence (\ref{sq-4})
with compatibility of the cup product \cite[page 110,(3.3)]{Brown} as a tool
it can be proved that $H^{\ast }(D_{8},\mathbb{F}_{2}[D_{8}/H_{1}])$ is
generated as a $H^{\ast }(D_{8},\mathbb{F}_{2})$-module by $I=i_{0}(1)$ and $%
A\in q_{1}^{-1}(\{y\})$.
\end{proof}

\subsection{\label{Sec:IndexSxS-F2}$\mathrm{Index}_{D_{8},\mathbb{F}%
_{2}}^{d+2}S^{d}\times S^{d}=\langle \protect\pi _{d+1},\protect\pi %
_{d+2}\rangle $}

The index by definition is
\begin{eqnarray*}
\mathrm{Index}_{D_{8},\mathbb{F}_{2}}^{d+2}S^{d}\times S^{d} &=&\mathrm{im}%
\big(\partial _{d+1}:E_{d+1}^{\ast ,d}\rightarrow E_{d+1}^{\ast +d+1,0}\big)
\\
&=&\mathrm{im}\left( \partial _{d+1}:H^{\ast }\left( D_{8},\mathbb{F}%
_{2}[D_{8}/H_{1}]\right) \rightarrow H^{\ast +d+1}(D_{8},\mathbb{F}%
_{2}\right) ).
\end{eqnarray*}%
From Proposition \ref{Prop:Modulestructure} this image is generated as a
module by the $\partial _{d+1}$-images of $H^{0}\left( D_{8},\mathbb{F}%
_{2}[D_{8}/H_{1}]\right) $ and of $H^{1}\left( D_{8},\mathbb{F}%
_{2}[D_{8}/H_{1}]\right) $. The $\partial _{d+1}$ image is computed by
applying restriction properties given in Proposition \ref{Prop:Res-1} to the
subgroup $H_{1}$. With the identification of $H^{\ast }\left( D_{8},\mathbb{F%
}_{2}[D_{8}/H_{1}]\right) $ given by Shapiro's lemma the morphism of
spectral sequences of Borel constructions induced by restriction is
specified in Figure \ref{Fig-3}. Also,%
\begin{equation*}
\mathrm{Index}_{D_{8},\mathbb{F}_{2}}^{d+2}S^{d}\times S^{d}=\langle
\partial _{d+1}^{D_{8}}(1),\partial _{d+1}^{D_{8}}(a),\partial
_{d+1}^{D_{8}}(b),\partial _{d+1}^{D_{8}}(a+b)\rangle .
\end{equation*}

\begin{figure}[htb]
\centering
\includegraphics[scale=0.65]{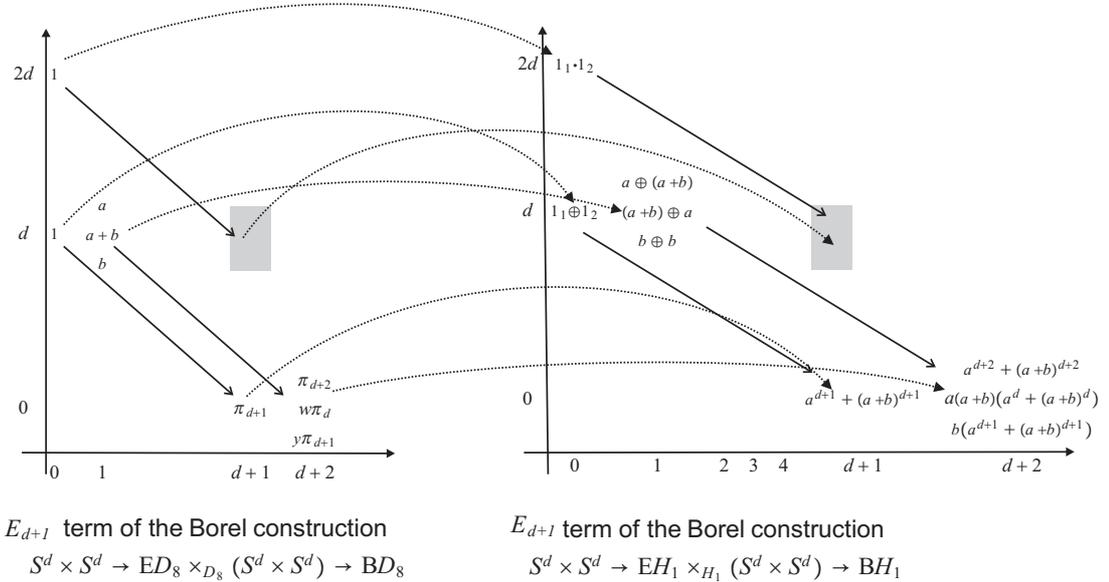}
\caption{{The morphism of spectral sequences}}
\label{Fig-3}
\end{figure}

To simplify notation let $\rho _{d}:=a^{d}+(a+b)^{d+1}$. Then from
\begin{equation*}
\begin{array}{ccccc}
1 & \overset{\mathrm{res}_{H_{1}}^{D_{8}}}{\longmapsto } & 1_{1}\oplus 1_{2}
& \overset{\partial _{d+1}^{H_{1}}}{\longmapsto } & \rho _{d+1} \\
&  &  &  &  \\
\left\{ a,a+b,b\right\} & \overset{\mathrm{res}_{H_{1}}^{D_{8}}}{\longmapsto
} & \left\{
\begin{array}{c}
a\oplus (a+b) \\
\left( a+b\right) \oplus a \\
b\oplus b%
\end{array}%
\right\} & \overset{\partial _{d+1}^{H_{1}}}{\longmapsto } & \left\{ \rho
_{d+2},a(a+b)\rho _{d},b\rho _{d+1}\right\}%
\end{array}%
\end{equation*}%
it follows that
\begin{equation*}
\mathrm{res}_{H_{1}}^{D_{8}}\big( \big\{ \partial
_{d+1}^{D_{8}}(1),~\partial _{d+1}^{D_{8}}(a),~\partial
_{d+1}^{D_{8}}(b)~,\partial _{d+1}^{D_{8}}(a+b)\big\} \big) {\ =\{\rho
_{d+2},~a(a+b)\rho _{d},~b\rho _{d+1}\}.}
\end{equation*}%
The formula%
\begin{eqnarray*}
\rho _{d+2} &=&a^{d+2}+(a+b)^{d+2} \\
&=&(a+a+b)\Big( \rho _{d+1}+a(a+b)\sum_{i=0}^{d-1}a^{i}(a+b)^{d-1-i}\Big) \\
&=&b\rho _{d+1}+a(a+b)(a+a+b)\sum_{i=0}^{d-1}a^{i}(a+b)^{d-1-i} \\
&=&b\rho _{d+1}+a(a+b)\rho _{d}
\end{eqnarray*}%
together with Remark \ref{Rem:definition} and the knowledge of the
restriction $\mathrm{res}_{H_{1}}^{D_{8}}$ implies that%
\begin{equation*}
\mathrm{res}_{H_{1}}^{D_{8}}(\pi _{d})={\ \rho _{d}.}
\end{equation*}%
Therefore, there exist $x\alpha ,x\beta ,x\gamma ,x\delta \in \mathrm{\ker }%
( \mathrm{res}_{H_{1}}^{D_{8}}) $ such that%
\begin{equation*}
\partial _{d+1}^{D_{8}}(1)=\pi _{d+1}+x\alpha
\end{equation*}%
and%
\begin{equation*}
\big\{ \partial _{d+1}^{D_{8}}(a),~\partial _{d+1}^{D_{8}}(b),~\partial
_{d+1}^{D_{8}}(a+b)\big\} =\left\{ \pi _{d+2}+x\beta ,~y\pi _{d+1}+x\gamma
,~w\pi _{d}+x\delta \right\} \text{.}
\end{equation*}%
Since $y$ divides $\pi _{d}$, Proposition \ref{Prop:MultiplyByX} implies
that $\alpha =\beta =\gamma =\delta =0$, and
\begin{eqnarray*}
\mathrm{Index}_{D_{8},\mathbb{F}_2}^{d+2}S^{d}\times S^{d} &=&\langle \partial
_{d+1}^{D_{8}}(1),\partial _{d+1}^{D_{8}}(a),\partial
_{d+1}^{D_{8}}(b),\partial _{d+1}^{D_{8}}(a+b)\rangle \\
&=&\langle \pi _{d+1},\pi _{d+2},~y\pi _{d+1},~w\pi _{d}\rangle \\
&=&\langle \pi _{d+1},\pi _{d+2}\rangle .
\end{eqnarray*}

\begin{remark*}
\label{Rem-Injective}The property that the concretely described homomorphism
\begin{equation*}
\mathrm{res}_{H_{1}}^{D_{8}}:H^{\ast }(D_{8},\mathbb{F}_{2}[D_{8}/H_{1}])%
\rightarrow H^{\ast }(H_{1},\mathbb{F}_{2}[D_{8}/H_{1}])
\end{equation*}%
is injective holds more generally \cite[Lemma on page 187]{Evans}.
\end{remark*}

\subsection{$\mathrm{Index}_{D_{8},\mathbb{F}_{2}}S^{d}\times S^{d}=\langle
\protect\pi _{d+1},\protect\pi _{d+2},w^{d+1}\rangle $}

In the previous section we described the differential $\partial
_{d+1}^{D_{8}}$ of the Serre spectral sequence associated with the Borel
construction%
\begin{equation*}
S^{d}\times S^{d}\rightarrow \mathrm{E}D_{8}\times _{D_{8}}\left(
S^{d}\times S^{d}\right) \rightarrow \mathrm{B}D_{8}.
\end{equation*}%
The only remaining, possibly non-trivial, differential is $\partial
_{2d+1}^{D_{8}}$.

The following proposition describing $E_{2d+1}^{\ast ,2d}$ can be obtained
from Figure \ref{Fig-3}.

\begin{proposition}
$E_{2d+1}^{\ast ,2d}=\ker \big( \partial _{d+1}^{D_{8}}:E_{d+1}^{\ast
,2d}\rightarrow E_{d+1}^{\ast +d+1,d}\big) =x\cdot H^{\ast }(D_{8},\mathbb{F}%
_{2})$
\end{proposition}

\begin{proof}
The restriction property from Proposition \ref{Prop:Res-1}(D), applied to
the element $1\in E_{d+1}^{0,2d}=H^{\ast }(D_{8},\mathbb{F}_{2})$ implies
that $\partial _{d+1}^{D_{8}}(1)\neq 0$. Proposition \ref{Prop:MultiplyByX},
together with the fact that multiplication by $y$ and by $w$ in $H^{\ast
}(D_{8},\mathbb{F}_{2}[D_{8}/H_{1}])$ is injective, implies that $\ker \big( %
\partial _{d+1}^{D_{8}}:E_{d+1}^{\ast ,2d}\rightarrow E_{d+1}^{\ast +d+1,d}%
\big) =xH^{\ast }(D_{8},\mathbb{F}_{2})$.
\end{proof}

The description of the differential $\partial _{2d+1}^{D_{8}}:E_{2d+1}^{\ast
,2d}\rightarrow E_{2d+1}^{\ast +2d+1,0}$ comes in an indirect way. There is
a $D_{8}$-equivariant map
\begin{equation*}
S^{d}\times S^{d}\rightarrow S^{d}\ast S^{d}\approx S(\left( V_{+-}\oplus
V_{-+}\right) ^{\oplus (d+1)})
\end{equation*}%
given by $S^d\times S^d\ni(t_1,t_2)\mapsto \frac{1}{2}t_1 + \frac{1}{2}t_2\in S^d * S^d$.
The result of Section \ref{sec:Index-W} and the basic property of the index
(Proposition \ref{prop:basic}) imply that%
\begin{equation*}
\mathrm{Index}_{D_{8},\mathbb{F}_{2}}S^{d}\times S^{d}\supseteq \mathrm{Index%
}_{D_{8},\mathbb{F}_{2}}S(\left( V_{+-}\oplus V_{-+}\right) ^{\oplus
(d+1)})=\langle w^{d+1}\rangle .
\end{equation*}%
Thus $w^{d+1}\in \mathrm{Index}_{D_{8},\mathbb{F}_{2}}S^{d}\times S^{d}$.
Since by Corollary \ref{Cor:Index-1} $w^{d+1}\notin \mathrm{Index}_{D_{8},%
\mathbb{F}_{2}}^{d+1}S^{d}\times S^{d}$ it follows that
\begin{equation*}
w^{d+1}\in \mathrm{im}\big(\partial
_{2d+1}^{D_{8}}:E_{2d+1}^{1,2d}\rightarrow E_{2d+1}^{2d+2,0}\big).
\end{equation*}%
But the only nonzero element in $E_{2d+1}^{1,2d}$ is $x$, therefore
\begin{equation*}
\partial _{2d+1}^{D_{8}}\left( x\right) =w^{d+1}.
\end{equation*}%
This concludes the proof of equation (\ref{eq:IndexSxS-F2}).

\subsection{\label{Sec:Alternative}An alternative proof, sketch}

The objective of our index calculation is to find the kernel of the map (cf.
Section \ref{Sec:Fadell--Husseini})
\begin{equation}
H^{\ast }(\mathrm{E}D_{8}\times _{D_{8}}\left( S^{d}\times S^{d}\right) ,%
\mathbb{F}_{2})=H_{D_{8}}^{\ast }(S^{d}\times S^{d},\mathbb{F}%
_{2})\leftarrow H_{D_{8}}^{\ast }(\mathrm{pt},\mathbb{F}_{2})=H^{\ast }(\mathrm{E}%
D_{8}\times _{D_{8}}\mathrm{pt},\mathbb{F}_{2})\text{.}  \label{1}
\end{equation}%
This map is induced by the map of spaces%
\begin{equation}
\mathrm{E}D_{8}\times _{D_{8}}(S^{d}\times S^{d})\rightarrow \mathrm{E}%
D_{8}\times _{D_{8}}\mathrm{pt}.  \label{2}
\end{equation}%
From the definition of the product $\times _{D_{8}}$ the map (\ref{2}) is
induced by $\mathrm{E}D_{8}\times (S^{d}\times S^{d})\rightarrow \mathrm{E}%
D_{8}\times \mathrm{pt}$, i.e. by $(S^{d}\times S^{d})\rightarrow \mathrm{pt}$. The map (\ref%
{2}), again by definition of product $\times _{D_{8}}$ is%
\begin{equation}
\left( \mathrm{E}D_{8}\times (S^{d}\times S^{d})\right) /D_{8}\rightarrow
\left( \mathrm{E}D_{8}\times \mathrm{pt}\right) /D_{8}.  \label{3}
\end{equation}%
Let $S_{2}\cong\mathbb{Z}_{2}$ denotes the quotient group $D_{8}/H_{1}$. There is a natural
homeomorphisms \cite[Proposition 1.59, page 40]{Kawa}
\begin{equation}
\left( \left( \mathrm{E}D_{8}\times (S^{d}\times S^{d})\right) /H_{1}\right)
/S_{2}\rightarrow \left( \left( \mathrm{E}D_{8}\times \mathrm{pt}\right)
/H_{1}\right) /S_{2}  \label{4}
\end{equation}%
which is induced by the map

\begin{equation}
\left( \mathrm{E}D_{8}\times (S^{d}\times S^{d})\right) /H_{1}\rightarrow
\left( \mathrm{E}D_{8}\times \mathrm{pt}\right) /H_{1}  \label{5}
\end{equation}%
Since $\mathrm{E}D_{8}$ is also a model for $\mathrm{E}H_{1}$, the map (\ref%
{5}) is a projection map in the Borel construction of $S^{d}\times S^{d}$
with respect to the group $H_{1}$:
\begin{equation}
\begin{diagram}[size=2em,textflow]
S^{d}\times S^{d} & \rTo & \left( \mathrm{E}D_{8}\times (S^{d}\times
S^{d})\right) /H_{1} \\
&  & \dTo \\
&  & \mathrm{B}H_{1}
\end{diagram}
\label{5.1}
\end{equation}

\medskip

\noindent The group $D_{8}$ acts freely on $\mathrm{E}D_{8}\times
(S^{d}\times S^{d})$ and on $\mathrm{E}D_{8}\times \mathrm{pt}$. Therefore the $S_{2}$
actions on the spaces $\left( \mathrm{E}D_{8}\times (S^{d}\times S^{d})\right) /H_{1}$
and $\left( \mathrm{E}D_{8}\times \mathrm{pt}\right) /H_{1}$ are also free. There are
natural homotopy equivalences
\begin{eqnarray*}
\left( \left( \mathrm{E}D_{8}\times (S^{d}\times S^{d})\right) /H_{1}\right)
/S_{2} &\simeq &\mathrm{E}S_{2}\times _{S_{2}}\left( \left( \mathrm{E}%
D_{8}\times (S^{d}\times S^{d})\right) /H_{1}\right) \text{, } \\
\left( \left( \mathrm{E}D_{8}\times \mathrm{pt}\right) /H_{1}\right) /S_{2} &\simeq &%
\mathrm{E}S_{2}\times _{S_{2}}\left( \left( \mathrm{E}D_{8}\times \mathrm{pt}\right)
/H_{1}\right)
\end{eqnarray*}%
which transform the map (\ref{4}) into a map of Borel constructions
\begin{equation}
\mathrm{E}S_{2}\times _{S_{2}}\left( \left( \mathrm{E}D_{8}\times
(S^{d}\times S^{d})\right) /H_{1}\right) \rightarrow \mathrm{E}S_{2}\times
_{S_{2}}\left( \left( \mathrm{E}D_{8}\times \mathrm{pt}\right) /H_{1}\right)
\label{6}
\end{equation}%
induced by the map (\ref{5}) on the fibres.

\noindent The map between Borel constructions (\ref{6}) induces a map of
associated Serre spectral sequences which on the $E_{2}$-term looks like%
\begin{equation}
\mathcal{E}_{2}^{p,q}=H^{p}(S_{2},H^{q}((\mathrm{E}D_{8}\times \left(
S^{d}\times S^{d}\right) )/H_{1},\mathbb{F}_{2}))\leftarrow
H^{p}(S_{2},H^{q}(\left( \mathrm{E}D_{8}\times \mathrm{pt}\right) /H_{1},\mathbb{F}%
_{2}))=\mathcal{H}_{2}^{p,q}\text{.}  \label{7}
\end{equation}%
The spectral sequence $\mathcal{H}_{2}^{p,q}$ is the one studied in section %
\ref{Sec:H*(D8,F_2)}. It converges to $H^{\ast }(D_{8},\mathbb{F}_{2})$ and $%
\mathcal{H}_{2}^{p,q}=\mathcal{H}_{\infty }^{p,q}$.

\begin{lemma}
$\mathcal{E}_{2}^{p,q}=\mathcal{E}_{\infty }^{p,q}$.
\end{lemma}

\begin{proof}
The action of $H_{1}$ on $S^{d}\times S^{d}$ is free. Therefore
\begin{equation}
\left( \mathrm{E}D_{8}\times \left( S^{d}\times S^{d}\right) \right)
/H_{1}\simeq \left( S^{d}\times S^{d}\right) /H_{1}=\mathbb{R}P^{d}\times\mathbb{R}
P^{d}  \label{8}
\end{equation}%
where the induced action of $S_{2}$ from $\left( \mathrm{E}D_{8}\times
\left( S^{d}\times S^{d}\right) \right) /H_{1}$ onto $\mathbb{R}P^{d}\times
\mathbb{R}P^{d}$ interchanges the copies of $\mathbb{R}P^{d}\times\mathbb{R}
P^{d}$. The $S_{2}$-homotopy equivalence (\ref{8}) induces an isomorphism of
induced Serre spectral sequences of Borel constructions
\begin{equation*}
\mathcal{E}_{2}^{p,q}=H^{p}(S_{2},H^{q}(\left( \mathrm{E}D_{8}\times \left(
S^{d}\times S^{d}\right) \right) /H_{1},\mathbb{F}_{2}))\cong
H^{p}(S_{2},H^{q}(\mathbb{R}P^{d}\times\mathbb{R}P^{d},\mathbb{F}_{2}))=\mathcal{G}_{2}^{p,q}\text{.}
\end{equation*}%
Since for the spectral sequence $\mathcal{G}_{2}^{p,q}$, by \cite[Theorem
1.7, page 118]{Adem-Milgram}, we know that $\mathcal{G}_{2}^{p,q}=\mathcal{G}%
_{\infty }^{p,q}$, the same must hold for the spectral sequence $\mathcal{E}%
_{\ast }^{\ast ,\ast }$.
\end{proof}

\medskip

\noindent We have obtained the following presentation of the map (\ref{1}) and
the related map of the fibres (\ref{5}).

\begin{proposition}
\qquad
\begin{compactenum}[\rm(A)]
\item The map $H_{D_{8}}^{\ast }(\mathrm{pt},\mathbb{F}_{2})\rightarrow
H_{D_{8}}^{\ast }(S^{d}\times S^{d},\mathbb{F}_{2})$ gives rise to
a map of spectral sequences of $S_{2}$-Borel constructions%
\begin{equation}
\mathcal{H}_{2}^{p,q}=H^{p}(S_{2},H^{q}(\left( \mathrm{E}D_{8}\times
\mathrm{pt}\right) /H_{1},\mathbb{F}_{2}))\rightarrow H^{p}(S_{2},H^{q}((\mathrm{E}%
D_{8}\times \left( S^{d}\times S^{d}\right) )/H_{1},\mathbb{F}_{2}))=%
\mathcal{E}_{2}^{p,q}  \label{9}
\end{equation}%
which is induced by the map on fibres $\left( \mathrm{E}D_{8}\times
(S^{d}\times S^{d})\right) /H_{1}\rightarrow \left( \mathrm{E}D_{8}\times
\mathrm{pt}\right) /H_{1}$.

\item The map on the fibres is the projection map in the $H_{1}$%
-Borel construction%
\begin{equation*}
\begin{array}{ccccc}
S^{d}\times S^{d} & \rightarrow & \left( \mathrm{E}D_{8}\times (S^{d}\times
S^{d})\right) /H_{1} & \rightarrow & \mathrm{B}H_{1}%
\end{array}%
.
\end{equation*}
It is completely determined in $\mathbb{F}_{2}$ cohomology by its kernel:%
\begin{equation*}
\ker \left( H^{\ast }(H_{1},\mathbb{F}_{2})\rightarrow H^{\ast }(\left(
\mathrm{E}D_{8}\times (S^{d}\times S^{d})\right) /H_{1},\mathbb{F}%
_{2})\right) =\mathrm{Index}_{H_{1},\mathbb{F}_{2}}S^{d}\times S^{d}=\langle
a^{d+1},(a+b)^{d+1}\rangle .
\end{equation*}
\end{compactenum}
\end{proposition}

\medskip

\noindent The $\mathcal{E}_{2}^{p,q}=\mathcal{E}_{\infty }^{p,q}$ and $%
\mathcal{H}_{2}^{p,q}=\mathcal{H}_{\infty }^{p,q}$ are described by \cite[%
Lemma 1.4, page 117]{Adem-Milgram}. Therefore, $\mathrm{Index}_{D_{8},%
\mathbb{F}_{2}}S^{d}\times S^{d}$ or the kernel of the map of spectral
sequences (\ref{9}) is completely determined by the kernel of the map of $%
S_{2}$-invariants%
\begin{equation}
\begin{array}{ccc}
\mathbb{F}_{2}[a,a+b]^{S_{2}} & \rightarrow & \left( \mathbb{F}%
_{2}[a,a+b]/\langle a^{d+1},(a+b)^{d+1}\rangle \right) ^{S_{2}} \\
\shortparallel &  & \shortparallel \\
H^{\ast }(H_{1},\mathbb{F}_{2})^{S_{2}} & \rightarrow & H^{\ast }(\left(
\mathrm{E}D_{8}\times (S^{d}\times S^{d})\right) /H_{1},\mathbb{F}%
_{2})^{S_{2}}%
\end{array}
\label{10}
\end{equation}%
where $S_{2}$ action is given by $a\longmapsto a+b$. The equation (\ref%
{eq:IndexSxS-F2})
\begin{equation*}
\mathrm{Index}_{D_{8},\mathbb{F}_{2}}S^{d}\times S^{d}=\langle \pi
_{d+1},\pi _{d+2},w^{d+1}\rangle .
\end{equation*}%
is a consequence of the previous discussion, identification of elements (\ref%
{eq:NotationOfElements}) in the spectral sequence (\ref{eq:E2}) and the
following proposition about symmetric polynomials.

\pagebreak

\begin{proposition}
\qquad
\begin{compactenum}[\rm(A)]
\item A symmetric polynomial $\sum a^{i_{k}}(a+b)^{j_{k}}\in \mathbb{F}%
_{2}[a,a+b]^{S_{2}}$ is in the kernel of the map {\rm (\ref{10})}
if and only if
for every monomial%
\begin{equation*}
a^{d+1}~|~a^{i_{k}}(a+b)^{j_{k}}\text{ or }%
(a+b)^{d+1}~|~a^{i_{k}}(a+b)^{j_{k}}.
\end{equation*}

\item The kernel of the map {\rm (\ref{10})}, as an ideal in $\mathbb{F}%
_{2}[a,a+b]^{S_{2}}$ is generated by%
\begin{equation*}
a^{d+1}+(a+b)^{d+1},~\qquad a^{d+2}+(a+b)^{d+2},~\qquad a^{d+1}(a+b)^{d+1}.
\end{equation*}
\end{compactenum}
\end{proposition}

\medskip

\noindent The approach presented here, with all its advantages, has two disadvantages:
\begin{compactenum}[\rm(1)]
\item The carrier of the combinatorial lower bound for the mass
partition
problem, the partial index $\mathrm{Index}_{D_{8},\mathbb{F}%
_{2}}^{d+2}S^{d}\times S^{d}$, cannot be obtained without extra
effort.

\item It cannot be used for computation of the index
$\mathrm{Index}_{D_{8}, \mathbb{Z}}^{d+2}S^{d}\times S^{d}$; the
spectral sequence $\mathcal{H}_{2}^{p,q}$, if considered with
$\mathbb{Z}$ coefficients, is the sequence (\ref{eq:E2withZ})
whose $E_{\infty}$-term has a ring structure that differs from
$H^{*}(D_8,\mathbb{Z})$.
\end{compactenum}

\noindent These were our reasons for presenting this idea just as a sketch.

\section{\label{Sec:IndexProductD8-Z}$\mathbf{Index}_{D_{8},\mathbb{Z}
}\mathbf{S}^{d}\mathbf{\times S}^{d}$}

Let $\Pi _{0}=0$, $\Pi _{1}=\mathcal{Y}$ and $\Pi _{n+2}=\mathcal{Y}\Pi
_{n+1}+\mathcal{W}\Pi _{n}$, for $n\geq 0$, be a sequence of polynomials in $%
H^{\ast }(D_{8},\mathbb{\mathbb{Z}})$. This section is devoted to the proof of the equality%
\begin{equation}
\mathrm{Index}_{D_{8},\mathbb{\mathbb{Z}}}^{d+2}S^{d}\times S^{d}=\left\{
\begin{array}{ll}
\langle \Pi _{\frac{d+2}{2}},\Pi _{\frac{d+4}{2}},\mathcal{M}\Pi _{\frac{d}{2%
}}\rangle & ,{\ \text{ for }}d{\ \text{ even}} \\
\langle \Pi _{\frac{d+1}{2}},\Pi _{\frac{d+3}{2}}\rangle & ,{\ \text{ for }d%
\text{ odd.}}%
\end{array}%
\right.  \label{eq:IndexSxS-ZZ}
\end{equation}%
The index is determined by the explicit computation of the $E_{d+2}$-term of
the Serre spectral sequence associated with the Borel construction%
\begin{equation*}
S^{d}\times S^{d}\rightarrow \mathrm{E}D_{8}\times _{D_{8}}\left(
S^{d}\times S^{d}\right) \rightarrow \mathrm{B}D_{8}.
\end{equation*}%
As in the previous section, the group $D_{8}$ acts nontrivially on the
cohomology of the fibre and thus the coefficients in the spectral sequence
are local. The $E_{2}$-term is given by%
\begin{equation}
\begin{array}{lll}
E_{2}^{p,q} & = & H^{p}(\mathrm{B}D_{8},\mathcal{H}^{q}(S^{d}\times S^{d},%
\mathbb{Z}))=H^{p}(D_{8},H^{q}(S^{d}\times S^{d},\mathbb{Z})) \\
&  &  \\
& = & \left\{
\begin{array}{lllll}
H^{p}(D_{8},\mathbb{Z}) & \text{, }q=0,2d &  &  &  \\
H^{p}(D_{8},H^{d}(S^{d}\times S^{d},\mathbb{Z})) & \text{, }q=d &  &  &  \\
0 & \text{, }q\neq 0,d,2d. &  &  &
\end{array}%
\right.
\end{array}
\label{eqSS-Z}
\end{equation}%
The local coefficients are nontrivial in the $d$-th row of the spectral
sequence.

\subsection{\label{Sec:Row-Z}The $d$-th row as an $H^{\ast }(D_{8},\mathbb{Z}
)$-module}

The $D_{8}$-module $M:=H^{d}(S^{d}\times S^{d},\mathbb{\mathbb{Z}})$, as an abelian group, is isomorphic to $\mathbb{
\mathbb{Z}\times\mathbb{Z}}$. Since the action of $D_{8}$ on $M$ depends on $d$ we distinguish two
cases.

\subsubsection{The case when $d$ is odd}

The action on $M$ is given by%
\begin{equation*}
\varepsilon _{1}\cdot (x,y)=(x,y),~\qquad \varepsilon _{2}\cdot
(x,y)=(x,y),~\qquad \sigma \cdot (x,y)=(y,x).
\end{equation*}%
Thus, there is an isomorphism of $D_{8}$-modules $M\cong \mathbb{Z}[D_{8}/H_{1}]$.
The situation resembles the one in Section \ref{Sec:Row},
and therefore the following propositions hold.

\begin{proposition}
$H^{\ast }(D_{8},\mathbb{\mathbb{Z}}[D_{8}/H_{1}])\ \cong _{\mathrm{ring}}\ H^{\ast }(H_{1},\mathbb{
\mathbb{Z}})$.
\end{proposition}

\begin{proof}
The claim follows from Shapiro's lemma \cite[Proposition 6.2, page 73]{Brown}
and the fact that when $[G:H]<\infty $ there is an isomorphism of $G$%
-modules \textrm{Coind}$_{H}^{G}M\cong ~$\textrm{Ind}$_{H}^{G}M$.
\end{proof}

\begin{proposition}
\label{Prop:MultiplyByX-Z} Let $\mathcal{T\in }H^{\ast }(D_{8},
\mathbb{Z})$ and $P\in H^{\ast }(H_{1},\mathbb{Z})\cong H^{\ast }(D_{8},
\mathbb{Z}[D_{8}/H_{1}])$.
\begin{compactenum}[\rm(A)]
\item The action of $H^{\ast }(D_{8},\mathbb{Z})$ on $H^{\ast }(D_{8},
\mathbb{Z}[D_{8}/H_{1}])$ is given by
\begin{equation*}
\mathcal{T\cdot }P:=\mathrm{res}_{H_{1}}^{D_{8}}\left( \mathcal{T}\right)
\mathcal{\cdot }P\text{.}
\end{equation*}%
Here $P$ on the right hand side is an element of $H^{\ast }(H_{1},\mathbb{Z})$ and on the left hand side is
its isomorphic image under the isomorphism from the previous proposition. In particular,
\[\mathcal{X\cdot }H^{\ast}(D_{8},\mathbb{Z}[D_{8}/H_{1}])=0.
\]

\item $H^{\ast }(D_{8},\mathbb{Z})$-module $H^{\ast }(D_{8},\mathbb{Z}[D_{8}/H_{1}])$ is generated by the two elements
\begin{equation*}
1,\alpha \in H^{\ast }(H_{1},\mathbb{Z})\cong H^{\ast }(D_{8},\mathbb{Z}[D_{8}/H_{1}])
\end{equation*}%
of degree $0$ and $2$.

\item The map $H^{\ast }(D_{8},\mathbb{Z}[D_{8}/H_{1}])\rightarrow H^{\ast }(D_{8},\mathbb{F}_{2}[D_{8}/H_{1}]),$
induced by the coefficient map $\mathbb{Z}
\rightarrow \mathbb{F}_{2}$, is given by $1,\alpha \longmapsto 1,a^{2}$.
\end{compactenum}
\end{proposition}

\begin{proof}
The isomorphism $H^{\ast }(D_{8},\mathbb{Z}
[D_{8}/H_{1}])\ \cong _{\mathrm{ring}}\ H^{\ast }(H_{1},
\mathbb{Z})$ induced by Shapiro's lemma \cite[Proposition 6.2, page 73]{Brown}
carries the $H^{\ast }(D_{8},\mathbb{Z})$-module structure to $H^{\ast }(H_{1},
\mathbb{Z})$ via $\mathrm{res}_{H_{1}}^{D_{8}}:H^{\ast }(D_{8},
\mathbb{Z})\rightarrow H^{\ast }(H_{1},\mathbb{Z})$. In this way the complete $H^{\ast }(D_{8},
\mathbb{Z})$-module structure is given on $H^{\ast }(D_{8},\mathbb{Z}[D_{8}/H_{1}])$. The claim (B) follows from the restriction diagram (\ref%
{eq:diagramD8-1-Z}). The morphism of restriction diagrams induced by
the coefficient reduction homomorphism $c:\mathbb{Z}\rightarrow \mathbb{F}_{2}$ implies the last statement.
\end{proof}

\subsubsection{The case when $d$ is even}

The action on $M$ is given by%
\begin{equation*}
\varepsilon _{1}\cdot (x,y)=(-x,y),~\qquad \varepsilon _{2}\cdot
(x,y)=(x,-y),~\qquad \sigma \cdot (x,y)=(y,x).
\end{equation*}%
In this case we are forced to analyze the Bockstein spectral sequence
associated with the exact sequence of $D_{8}$-modules%
\begin{equation}
0\rightarrow M\overset{\times 2}{\rightarrow }M\rightarrow \mathbb{F}%
_{2}[D_{8}/H_{1}]\rightarrow 0,  \label{a-1}
\end{equation}%
i.e. with the exact couple%
\begin{equation}
\begin{diagram}[size=2em,textflow]
H^{\ast }(D_{8},M)&&\rTo{\times 2}&&H^{\ast }(D_{8},M)\\
&\luTo{\delta }&&\ldTo{c}&\\
&&H^{\ast }(D_{8},\mathbb{F}_{2}[D_{8}/H_{1}]).&&\\
\end{diagram}
\label{a-2}
\end{equation}

\medskip

First we study the Bockstein spectral sequence%
\begin{equation}
\begin{diagram}[size=2em,textflow]
H^{\ast }(H_{1},M)&&\rTo{\times 2}&&H^{\ast }(H_{1},M)\\
&\luTo{\delta }&&\ldTo{c}&\\
&&H^{\ast }(H_{1},\mathbb{F}_{2}[D_{8}/H_{1}]).&&\\
\end{diagram}
\label{a-3}
\end{equation}
As in Section \ref{Sec:IndexSxS-F2}, we have that $H^{\ast }(H_{1},%
\mathbb{F}_{2}[D_{8}/H_{1}])=\mathbb{F}_{2}[a,a+b]\oplus \mathbb{F}%
_{2}[a,a+b]$. The module $M$ as an $H_{1}$-module can be decomposed into the sum
of two $H_{1}$-modules $Z_{1}$ and $Z_{2}$. The modules $Z_{1}\cong _{Ab}\mathbb{Z}$ and $Z_{2}\cong _{Ab}
\mathbb{Z}$ are given by
\begin{equation*}
\varepsilon _{1}\cdot x=-x,~\varepsilon _{2}\cdot x=x\text{ ~~~~and~~~~ }\varepsilon
_{1}\cdot y=y,~\varepsilon _{2}\cdot y=-y
\end{equation*}%
for $x\in Z_{1}$ and $y\in Z_{2}$. This decomposition also induces a
decomposition of $H_{1}$-modules $\mathbb{F}_{2}[D_{8}/H_{1}]\cong \mathbb{F}%
_{2}\oplus \mathbb{F}_{2}$. Thus, the exact couple (\ref{a-3}) decomposes
into the direct sum of two exact couples
\begin{equation}
\begin{diagram}[size=2em,textflow]
H^{\ast }(H_{1},Z_{1})&&\rTo{\times 2}&&H^{\ast }(H_{1},Z_{1})&~~~~&H^{\ast }(H_{1},Z_{2})&&\rTo{\times 2}&&H^{\ast }(H_{1},Z_{2})\\
&\luTo{\delta }&&\ldTo{c}&&~~~~&&\luTo{\delta }&&\ldTo{c}&\\
&&H^{\ast }(H_{1},\mathbb{F}_{2})&&&~~~~&&&H^{\ast }(H_{1},\mathbb{F}_{2})&&\\
\end{diagram}
\label{a-4}
\end{equation}
Since all the maps in these exact couples are $H^{\ast }(H_{1},\mathbb{Z})$-module maps, the following proposition completely determines both exact
couples.

\smallskip

\begin{proposition}
In the exact couples \textrm{(\ref{a-4})} differentials $d_{1}=c\circ \delta
$ are determined, respectively, by%
\begin{equation}
d_{1}(1)=a,~d_{1}(b)=b(b+a)~\qquad \text{and}~\qquad
d_{1}(1)=a+b,~d_{1}(a)=d_{1}(b)=ab.  \label{a-5}
\end{equation}
\end{proposition}

\begin{proof}
In both claims we use the following diagram of exact couples induced by
restrictions, where $i\in \{1,2\}$:

\smallskip
\mathsmall{
\[
\begin{diagram}[leftflush]
&&
\begin{tabular}{|l|}
\hline
\begin{diagram}[height=1.5em,width=0.3em]
H^{\ast }(H_{1},Z_{i})&&\rTo{\times 2}&&H^{\ast }(H_{1},Z_{i})\\
&\luTo{\delta }&&\ldTo{c}&\\
&&H^{\ast }(H_{1},\mathbb{F}_{2})&&\\
\end{diagram}
 \\ \hline
\end{tabular}
&&  \\
&\ldTo & \dTo & \rdTo&  \\
&&&&\\
\begin{tabular}{|l|}
\hline
\begin{diagram}[height=1.5em,width=1em]
H^{\ast }(K_{1},Z_{i})&&\rTo{\times 2}&&H^{\ast }(K_{1},Z_{i})~~\\
&\luTo{\delta }&&\ldTo{c}&\\
&&H^{\ast }(K_{1},\mathbb{F}_{2})&&\\
\end{diagram}
 \\ \hline
\end{tabular}
&&
\begin{tabular}{|l|}
\hline
\begin{diagram}[height=1.5em,width=1em]
H^{\ast }(K_{2},Z_{i})&&\rTo{\times 2}&&H^{\ast }(K_{2},Z_{i})~~\\
&\luTo{\delta }&&\ldTo{c}&\\
&&H^{\ast }(K_{2},\mathbb{F}_{2})&&\\
\end{diagram}
 \\ \hline
\end{tabular}
&&
\begin{tabular}{|l|}
\hline
\begin{diagram}[height=1.5em,width=1em]
H^{\ast }(K_{3},Z_{i})&&\rTo{\times 2}&&H^{\ast }(K_{3},Z_{i})\\
&\luTo{\delta }&&\ldTo{c}&\\
&&H^{\ast }(K_{3},\mathbb{F}_{2})&&\\
\end{diagram}
 \\ \hline
\end{tabular}\\
\end{diagram}
\]}
\smallskip

\noindent\emph{The first exact couple. }The module $Z_{1}$ is a non-trivial $K_{1}$
and $K_{3}$-module, but a trivial $K_{2}$-module. Therefore by the long
exact sequences  (\ref{proof:rel-2.1}), properties of Steenrod squares and the
assumption at the end of the Section \ref{Sec:RestrictionDiagram}:
\begin{compactenum}[\rm(A)]
\item $K_{1}$-exact couple: $d_{1}(1)=t_{1}$ and $d_{1}(t_{1})=0$;

\item $K_{2}$-exact couple: $d_{1}(1)=0$ and $d_{1}(t_{2})=t_{2}^{2}$;

\item $K_{3}$-exact couple: $d_{1}(1)=t_{3}$ and $d_{1}(t_{3})=0$.
\end{compactenum}
\noindent Now
\begin{equation*}
\left.
\begin{array}{c}
\mathrm{res}_{K_{1}}^{H_{1}}(d_{1}(1))=t_{1} \\
\mathrm{res}_{K_{2}}^{H_{1}}(d_{1}(1))=0 \\
\mathrm{res}_{K_{3}}^{H_{1}}(d_{1}(1))=t_{3}%
\end{array}%
\right\} \Rightarrow d_{1}(1)=a~\qquad \left.
\begin{array}{c}
\mathrm{res}_{K_{1}}^{H_{1}}(d_{1}(b))=0 \\
\mathrm{res}_{K_{2}}^{H_{1}}(d_{1}(b))=t_{2}^{2} \\
\mathrm{res}_{K_{3}}^{H_{1}}(d_{1}(b))=0%
\end{array}%
\right\} \Rightarrow d_{1}(b)=b(b+a).
\end{equation*}%
\emph{The second exact couple. }The module $Z_{2}$ is a non-trivial $K_{2}$
and $K_{3}$-module, while it is a trivial $K_{1}$-module. Therefore by the
long exact (\ref{proof:rel-2.1}), properties of Steenrod squares and the
assumption at the end of the Section \ref{Sec:RestrictionDiagram}:
\begin{compactenum}[\rm(A)]
\item $K_{1}$-exact couple: $d_{1}(1)=0$ and $d_{1}(t_{1})=t_{1}^{2}$;

\item $K_{2}$-exact couple: $d_{1}(1)=t_{2}$ and $d_{1}(t_{2})=0$;

\item $K_{3}$-exact couple: $d_{1}(1)=t_{3}$ and $d_{1}(t_{3})=0$.
\end{compactenum}
\noindent Now
\begin{equation*}
\left.
\begin{array}{c}
\mathrm{res}_{K_{1}}^{H_{1}}(d_{1}(1))=0 \\
\mathrm{res}_{K_{2}}^{H_{1}}(d_{1}(1))=t_{2} \\
\mathrm{res}_{K_{3}}^{H_{1}}(d_{1}(1))=t_{3}%
\end{array}%
\right\} \Rightarrow d_{1}(1)=a+b~\qquad \left.
\begin{array}{c}
\mathrm{res}_{K_{1}}^{H_{1}}(d_{1}(b))=t_{1}^{2} \\
\mathrm{res}_{K_{2}}^{H_{1}}(d_{1}(b))=0 \\
\mathrm{res}_{K_{3}}^{H_{1}}(d_{1}(b))=0%
\end{array}%
\right\} \Rightarrow d_{1}(b)=ab.
\end{equation*}
\end{proof}

\begin{remark*}
The result of the previous proposition can be seen as a key step in an
alternative proof of the equation (\ref{eq:Index-pom}).
\end{remark*}

\medskip

\begin{proposition}
In the exact couple \textrm{(\ref{a-2})}, with identification $H^{\ast
}(D_{8},\mathbb{F}_{2}[D_{8}/H_{1}])=\mathbb{F}_{2}[a,a+b]$, the
differential $d_{1}=s\circ \delta $ satisfies%
\begin{equation}
d_{1}(1)=a,~d_{1}(a+b)=d_{1}(b)=b(b+a),~d_{1}(a^{2})=a^{3}.  \label{a-6}
\end{equation}%
(This determines $d_{1}$ completely since $c$ and $\delta $ are $H^{\ast
}(D_{8},\mathbb{Z})$-module maps.)
\end{proposition}

\begin{proof}
Recall from the Remark \ref{Rem-Injective} that the restriction map%
\begin{equation*}
\mathrm{res}_{H_{1}}^{D_{8}}:H^{\ast }(D_{8},\mathbb{F}_{2}[D_{8}/H_{1}])%
\rightarrow H^{\ast }(H_{1},\mathbb{F}_{2}[D_{8}/H_{1}])
\end{equation*}%
is injective. Then the equations (\ref{a-6}) are obtained by filling the
empty places in the following diagrams%
\begin{equation*}
\begin{diagram}[size=2.2em,textflow]
1 & \rTo{d_{1}} &
\begin{tabular}{|l|}
\hline
\\ \hline
\end{tabular}
&  & a+b & \rTo{d_{1}} &
\begin{tabular}{|l|}
\hline
\\ \hline
\end{tabular}
&  & a^{2} & \rTo{d_{1}} &
\begin{tabular}{|l|}
\hline
\\ \hline
\end{tabular}
\\
\dTo &  & \dTo &  & \dTo &  & \dTo &  & \dTo &
& \dTo \\
1\oplus 1 & \rTo{d_{1}} & a\oplus (a+b) &~~~~  &
(a+b)\oplus a & \rTo{d_{1}} & b(b+a)\oplus ab &~~~~  &
a^{2}\oplus (a+b)^{2} & \rTo{d_{1}} & a^{3}\oplus
(a+b)^{3}%
\end{diagram}%
\end{equation*}
where all vertical maps are $\mathrm{res}_{H_{1}}^{D_{8}}$.
\end{proof}

\begin{corollary}
\label{Coroll-1}$H^{\ast }(D_{8},M)$ is generated as a $H^{\ast }(D_{8},
\mathbb{Z})$-module by three elements $\zeta _{1}$, $\zeta _{2}$, $\zeta _{3}$ of
degree $1$, $2$, $3$ such that%
\begin{equation*}
c(\zeta _{1})=a,~\qquad c(\zeta _{2})=b(a+b),~\qquad c(\zeta _{3})=a^{3}
\end{equation*}%
where $c$ is the map $H^{\ast }(D_{8},M)\rightarrow H^{\ast }(D_{8},\mathbb{F%
}_{2}[D_{8}/H_{1}])$ from the exact couple (\ref{a-2}).
\end{corollary}

\subsection{$\mathrm{Index}_{D_{8},\mathbb{Z}}^{d+2}S^{d}\times S^{d}$}

The relation between the sequences of polynomials $\pi _{d}\in H^{\ast }(D_{8},%
\mathbb{F}_{2})$ and $\Pi _{d}\in H^{\ast }(D_{8},\mathbb{Z})$ is described by the following lemma.

\begin{lemma}
Let $c_{*}:H^{\ast }(D_{8},\mathbb{Z})\rightarrow H^{\ast }(D_{8},\mathbb{F}_{2})$ be the map induced by the
coefficient morphism $\mathbb{Z}\rightarrow \mathbb{F}_{2}$ (explicitly given by (\ref{eq:Map-j})). Then for every $%
d\geq 0$,
\begin{equation*}
c_{*}(\Pi _{d})=\pi _{2d}.
\end{equation*}
\end{lemma}

\begin{proof}
Induction on $d\geq 0$. For $d=0$ and $d=1$ the claim is obvious. Let $d\geq
2$ and let us assume that claim holds for every $d\leq k+1$. Then%
\begin{equation*}
\begin{array}{lllll}
c_{*}(\Pi _{k+2}) & = & c_{*}(\mathcal{Y}\Pi _{k+1}+\mathcal{W}\Pi _{k})\overset{%
hypo.}{=}y^{2}\pi _{2k+2}+w^{2}\pi _{2k} & = & y^{2}\pi _{2k+2}+yw\pi
_{2d+1}+yw\pi _{2d+1}+w^{2}\pi _{2k} \\
& = & y(y\pi _{2k+2}+w\pi _{2d+1})+w(y\pi _{2d+1}+w\pi _{2k}) & = & y\pi
_{2k+3}+w\pi _{2k+2} \\
& = & \pi _{2k+4}. &  &
\end{array}%
\end{equation*}
\end{proof}

\smallskip

There is a sequence of $D_{8}$-inclusions%
\begin{equation*}
S^{1}\times S^{1}\subset S^{2}\times S^{2}\subset\cdots\subset S^{d-1}\times
S^{d-1}\subset S^{d}\times S^{d}\subset S^{d+1}\times S^{d+1}\subset\cdots
\end{equation*}%
implying a sequence of ideal inclusions%
\begin{equation}
\mathrm{Index}_{D_{8},\mathbb{Z}
}^{3}S^{1}\times S^{1}\supseteq \mathrm{Index}_{D_{8},
\mathbb{Z}}^{4}S^{2}\times S^{2}\supseteq\cdots\supseteq \mathrm{Index}_{D_{8},
\mathbb{Z}}^{d+1}S^{d-1}\times S^{d-1}\supseteq \mathrm{Index}_{D_{8},
\mathbb{Z}}^{d+2}S^{d}\times S^{d}\supseteq\cdots  \label{incusion-1}
\end{equation}

\subsubsection{The case when $d$ is odd}

In this section we prove that%
\begin{equation}
\mathrm{Index}_{D_{8},\mathbb{Z}}^{d+2}S^{d}\times S^{d}=\langle \Pi _{\frac{d+1}{2}},\Pi _{\frac{d+3}{2}%
}\rangle .  \label{eqq-1}
\end{equation}%
The proof can be conducted as in the case of $\mathbb{F}_{2}$ coefficients
(Section \ref{Sec:IndexSxS-F2}). The results of Section \ref%
{Sec:IndexSxS-F2} can also be used to simplify the proof of equation (\ref%
{eqq-1}). The morphism $c_{*}:H^{\ast }(D_{8},\mathbb{\mathbb{Z}})\rightarrow H^{\ast }(D_{8},\mathbb{F}_{2})$ induced by the coefficient
morphism $\mathbb{Z}\rightarrow \mathbb{F}_{2}$ is a part of the morphism $C$ of Serre spectral
sequences (\ref{eqSS-Z}) and (\ref{eq:E2-Product}). Thus, for $1\in
E_{d+1}^{0,d}=H^{0}(D_{8},H^{d}(S^{d}\times S^{d},\mathbb{Z}))$,
$\hat{1}\in E_{d+1}^{0,d}=H^{0}(D_{8},H^{d}(S^{d}\times S^{d},\mathbb{F}%
_{2}))$, $\alpha \in E_{d+1}^{2,d}=H^{2}(D_{8},H^{d}(S^{d}\times S^{d},
\mathbb{Z}))$ and $a\in E_{d+1}^{1,d}=H^{1}(D_{8},H^{d}(S^{d}\times S^{d},
\mathbb{Z}))$,
\begin{eqnarray*}
C(\partial _{d+1}(1)) &=&\partial _{d+1}(C(1))=\partial _{d+1}(\hat{1})=\pi
_{d+1}=C\left( \Pi _{\frac{d+1}{2}}\right) , \\
C(\partial _{d+1}(\alpha )) &=&\partial _{d+1}(C(\alpha ))=\partial
_{d+1}(a^{2})=\partial _{d+1}(w\cdot \hat{1}+y\cdot a)=w\pi _{d+1}+y\pi
_{d+2}=\pi _{d+3}=C\left( \Pi _{\frac{d+3}{2}}\right) .
\end{eqnarray*}%
From Proposition \ref{Prop:MultiplyByX-Z} and the sequence of inclusions (%
\ref{incusion-1}) it follows that%
\begin{equation*}
\partial _{d+1}(1)=\Pi _{\frac{d+1}{2}}\text{ and }\partial _{d+1}(\alpha
)=\Pi _{\frac{d+3}{2}}\text{.}
\end{equation*}%
Finally, the statement (B) of Proposition \ref{Prop:MultiplyByX-Z} implies
equation (\ref{eqq-1}).

\subsubsection{The case when $d$ is even}

In this section we prove that%
\begin{equation}
\mathrm{Index}_{D_{8},\mathbb{Z}}^{d+2}S^{d}\times S^{d}=\langle \Pi _{\frac{d+2}{2}},\Pi _{\frac{d+4}{2}},%
\mathcal{M}\Pi _{\frac{d}{2}}\rangle .  \label{eqq-2}
\end{equation}%
The previous section implies that%
\begin{equation}
\langle \Pi _{\frac{d}{2}},\Pi _{\frac{d+2}{2}}\rangle \supseteq \mathrm{%
Index}_{D_{8},\mathbb{Z}}^{d+2}S^{d}\times S^{d}\supseteq \langle \Pi _{\frac{d+2}{2}},\Pi _{\frac{%
d+4}{2}}\rangle .  \label{incusion-2}
\end{equation}%
From Corollary \ref{Coroll-1} we know that $\mathrm{Index}_{D_{8},
\mathbb{Z}}^{d+2}S^{d}\times S^{d}$ is generated by three elements $\partial
_{d+1}(\zeta _{1})$, $\partial _{d+1}(\zeta _{2})$, $\partial _{d+1}(\zeta
_{3})$ of degrees $d+2$, $d+3$, $d+4$. Thus, $\partial _{d+1}(\zeta
_{1})=\Pi _{\frac{d+2}{2}}$ and $\partial _{d+1}(\zeta _{2})=\mathcal{M}\Pi
_{\frac{d}{2}}$. Since $\Pi _{\frac{d+4}{2}}\notin \langle \Pi _{\frac{d+2}{2%
}},\mathcal{M}\Pi _{\frac{d}{2}}\rangle $, then $\partial _{d+1}(\zeta
_{3})=\Pi _{\frac{d+4}{2}}$. The proof of the equation (\ref{eqq-2}) is
concluded.

Alternatively, the proof can be obtained with the help of the morphism $C$
of Serre spectral sequences (\ref{eqSS-Z}) and (\ref{eq:E2-Product}).

\end{document}